\documentclass[12pt,reqno]{amsart}

\newcommand\version{\today}     

\usepackage[T1]{fontenc}
\usepackage[utf8]{inputenc}
\usepackage{amsmath,amsfonts,amsthm,amssymb,amsxtra}
\usepackage{bbm} 
\usepackage[english]{babel} 
\usepackage{mathtools}
\usepackage{bm}
\usepackage{url}
\usepackage{esvect}
\usepackage{enumerate}
\usepackage{accents}
\usepackage{xcolor}

\usepackage[parfill]{parskip}    

\newtheorem{theorem}{Theorem}[section]
\newtheorem{proposition}[theorem]{Proposition}
\newtheorem{lemma}[theorem]{Lemma}

\newtheorem{claim}{Claim}[section]

\theoremstyle{definition}

\theoremstyle{remark}

\newtheorem{remarks}[theorem]{Remarks}

\numberwithin{equation}{section}

\newcommand{\be}{\begin{equation*}}
\newcommand{\ee}{\end{equation*}}
\newcommand{\beq}{\begin{equation}}
\newcommand{\eeq}{\end{equation}}
\newcommand{\begincal}{\begin{eqnarray*}}
\newcommand{\fincal}{\end{eqnarray*}}
\newcommand{\ds}{\displaystyle}

\renewcommand{\epsilon}{\varepsilon}
\newcommand{\eps}{\varepsilon}

\newcommand{\N}{\mathbb{N}}

\newcommand*\diff{\mathop{}\!\mathrm{d}} 

\newcommand{\R}{\mathbb{R}}

\newcommand{\ieps}{{i,\eps}}
\newcommand{\jeps}{{j,\eps}}

\newcommand{\oeps}{{1,\eps}}

\newcommand{\ue}{u_\eps}
\newcommand{\tue}{\tilde{u}_\varepsilon}
\newcommand{\hue}{\hat{u}_\varepsilon}

\newcommand{\de}{d_\varepsilon}

\newcommand{\io}{\int_{\Omega}}
\newcommand{\ibo}{\int_{\partial\Omega}}

\begin{document}

\title[Multibubble analysis in the higher dimensional BN problem]{Fine multibubble analysis in the higher-dimensional Brezis--Nirenberg problem}

		\author{Tobias König}
	\address[Tobias König]{Institut für Mathematik, 
Goethe-Universität Frankfurt, 
Robert-Mayer-Str. 10, 60629 Frankfurt am Main, Germany}	 
\email{koenig@mathematik.uni-frankfurt.de}

\author{Paul Laurain}
\address[Paul Laurain]{
Institut de Mathématiques de Jussieu, Université de Paris, Bâtiment Sophie Germain, Case 7052, 75205
Paris Cedex 13, France \& DMA, Ecole normale supérieure, CNRS, PSL Research University, 75005 Paris.}
\email{paul.laurain@imj-prg.fr}

\thanks{\copyright\, 2022 by the authors. This paper may be reproduced, in its entirety, for non-commercial purposes.\\
\emph{Date}: \version \\
Partial support through ANR BLADE-JC ANR-18-CE40-002 is acknowledged. 
The authors are grateful to Shuibo Huang for valuable comments on our work \cite{KL}, which also led to the improvement of this manuscript.  
}

\begin{abstract}
For a bounded set $\Omega \subset \R^N$ and a perturbation $V \in C^1(\overline{\Omega})$, we analyze the concentration behavior of a blow-up sequence of positive solutions to  
\[ -\Delta u_\eps + \eps V = N(N-2) u_\eps^\frac{N+2}{N-2} \]
for dimensions $N \geq 4$, which are non-critical in the sense of the Brezis--Nirenberg problem.

For the general case of multiple concentration points, we prove that concentration points are isolated and characterize the vector of these points as a critical point of a suitable function derived from the Green's function of $-\Delta$ on $\Omega$. Moreover, we give the leading order expression of the concentration speed. This paper, with a recent one by the authors \cite{KL} in dimension $N = 3$, gives a complete picture of blow-up phenomena in the Brezis-Nirenberg framework.

\end{abstract}

\maketitle

\section{Introduction and main results}

For $N \geq 4$, let $\Omega \subset \R^N$ be a bounded open set, and let $u_\eps$ be a sequence of solutions to 
\begin{align}
-\Delta u_\eps + \eps V u_\eps &= N(N-2)  u_\eps^\frac{N+2}{N-2}
 \qquad \text{ on } \Omega, \nonumber \\
 u_\eps &> 0 \qquad \text{ on } \Omega, \label{brezis peletier additive}  \\
 u_\eps &= 0 \qquad \text{ on } \partial \Omega. \nonumber
\end{align}
For the perturbation profile $V$, the canonical choice is $V \equiv -1$, but we will only assume $V \in C^1(\overline{\Omega})$ and $V < 0$ on $\overline{\Omega}$ throughout this paper. The understanding of the behavior of solutions of this equation is pivotal in the Yamabe problem, see for instance \cite{DH05} and reference therein.

Existence and non-existence of solutions to \eqref{brezis peletier additive} is a delicate matter and has been investigated in a famous paper by Brezis and Nirenberg \cite{Brezis1983}. This is largely due to the Sobolev-critical value of the exponent $\frac{N+2}{N-2}=2^*-1$, which allows concentration of a sequence of solutions around one or even several points of $\Omega$. Starting with \cite{Atkinson1987, Budd1987} and particularly an influential paper by Brezis and Peletier \cite{Brezis1989}, in the latter, after studying the behaviour of radial solution, the authors conjecture an asymptotic expression for $\Vert u_\eps\Vert_\infty$ in the case where $(u_\eps)$ has precisely one blow-up point. The present paper, with \cite{KL}, completely settles this long-standing open question by giving the precise behavior of arbitrary sequences of solutions, notably ones with multiple concentration points.

For one-peak solutions and $N \geq 4$ the location and speed of concentration have been characterized in \cite{Rey1989, Han1991} for $V \equiv -1$ and in \cite{Molle2003} for non-constant $V$. For the  related subcritical problem, with $V \equiv 0$ and $u_\eps^{\frac{N+2}{N-2}-\eps}$ on the right side of \eqref{brezis peletier additive}, the properties of multi-peak solutions have been analyzed in  \cite{Rey1992, Bahri1995, Rey1999}. In the latter, the authors always assume that the number of concentration points is {\it a priori } finite, which is not the case in the present paper and \cite{KL}.

Conversely, besides the one-peak solutions arising as energy-minimizers from \cite{Brezis1983}, we mention that multi-peak solutions with various properties have been constructed e.g. in \cite{Musso2002, delPino2004, Musso2018, Premoselli2022}.

When $N=3$, even in the presence of only one concentration point, the leading order of the speed at which blow-up solutions to \eqref{brezis peletier additive} concentrate is harder to obtain.\footnote{To be completely precise, for $N=3$ the relevant equation fulfilled by a blowing-up sequence of solutions is $-\Delta u_\eps + (a+\eps V) u_\eps = 3 u_\eps^5$, with a non-zero $a \in C(\overline{\Omega})$ as a consequence of the Brezis--Nirenberg dimensional effect observed in \cite{Brezis1983}.} This is due to a certain cancellation in the energy expansion which forces one to push the asymptotic analysis to a higher degree of precision. The results analogous to \cite{Rey1989, Han1991} for one-peak solutions have been obtained only recently, by the first author and collaborators in a series of papers \cite{Frank2021, Frank2019b, Frank2021b}. The full analysis for $N=3$ comprising multi-peak solutions has been carried out by the authors of the present paper in the recent preprint \cite{KL}. 

Finally, the blow-up of solutions to \eqref{brezis peletier additive} in the case $N \geq 4$ has not been studied in the literature yet, notably because the fine analysis of the concentration points was not available, which is done in Appendix \ref{appendixB}. The goal of the present paper is to close this gap, using and adapting the new methods of \cite{KL}. Remarkably, differently from one-peak solutions in dimension $N \geq 4$, the multi-peak case can also feature a cancellation phenomenon which makes it harder to derive the concentration speed. We  will explain this in more detail in the following subsection, where we state our main result. 

\subsection{Main result}

Let us introduce the object that largely governs the asymptotic behavior of $(u_\eps)$, namely the Green's function $G: \Omega \times \Omega \to \R$. This is the unique function satisfying, for each fixed $y\in\Omega$,
\begin{equation} \label{Ga-pde}
\left\{
\begin{array}{l@{\quad}l}
-\Delta_x\, G(x,y) =  \delta_y & \quad \text{in} \ \ \Omega\,, \\
G(\cdot,y) = 0  & \quad \text{on} \ \ \partial\Omega \,.
\end{array}
\right.  
\end{equation}
Note that $G(x,y) > 0$ for every $x,y \in \Omega$. 
The regular part $H$ of $G$ is defined by 
\begin{equation} \label{ha-def}
H(x,y) := \frac{1}{(N-2)\omega_{N-1} |x-y|^{N-2}} - G(x,y)\, ,
\end{equation}
where $\omega_{N-1}$ is the volume of the sphere $\mathbb S^{N-1} \subset \R^N$. It is well-known that for each $y\in\Omega$ the function $H(\cdot,y)$ is a smooth function in $\Omega$. Thus we may define the \emph{Robin function}
$$
\phi(y) := H(y,y) \,.
$$
It is known that single-blow-up sequences of solutions to \eqref{brezis peletier additive} must concentrate at critical points $x_0$ of $\phi$ when $V$ is constant \cite{Brezis1989, Rey1989, Han1991} and of a suitable function depending on $\phi$ and $V$ when $V$ is non-constant \cite{Molle2003}.

For any number $n \in \N$ of concentration points, let 
\[ \Omega_\ast^n := \{ \bm{x}= (x_1,...,x_n)\in \Omega^n \, : \, x_i \neq x_j \text{  for all  } i \neq j \}. \]
 For $\bm{x} \in \Omega_\ast^n$ we denote $M(\bm{x}) \in \R^{n \times n} = (m_{ij})_{i,j=1}^n$ the matrix with entries 
\begin{equation}
\label{m ij definition}
m_{ij}(\bm{x}) := 
\begin{cases}
\phi(x_i) & \text{ for } i = j, \\
- G(x_i, x_j) & \text{ for } i \neq j. 
\end{cases}
\end{equation}
Its lowest eigenvalue $\rho(\bm{x})$ is simple and the corresponding eigenvector can be chosen to have strictly positive components. We denote by $\bm{\Lambda}(\bm{x}) \in \R^n$ the unique vector such that 
\[ M(\bm{x}) \cdot \bm{\Lambda}(\bm{x}) = \rho(\bm{x})\bm{\Lambda}(\bm{x}) , \qquad (\bm{\Lambda}(\bm{x}) )_1 = 1. \] 

Next, let us define, for $\bm{\kappa} \in (0,\infty)^n$ and $\bm{x} \in \Omega_*^n$, 
\begin{equation}
\label{F definition}
F(\bm{\kappa}, \bm{x}) := \frac 12 \langle \bm{\kappa}, M(\bm{x}) \bm \kappa \rangle  +  d_N \frac{N-2}{4} \sum_i V(x_{i}) \kappa_i^{\frac{4}{N-2}}
\end{equation}
 where the dimensional constant $d_N > 0$ is given by 
 \begin{equation}
 \label{d_N constant}
 d_N = \frac{\Gamma(\frac{N}{2}) \Gamma(\frac{N-4}{2})}{\Gamma(N-1) \omega_{N-1} (N-2)^2}. 
\end{equation}  

Moreover, we define the Aubin--Talenti type bubble function 
\[ B(x):= \left(1 + |x|^2 \right)^{-\frac{N-2}{2}} \]
and, for every $\mu > 0$ and $x_0 \in \R^N$ its rescaled and translated versions 
\[ B_{\mu, x_0}(x) = \mu^{-\frac{N-2}{2}} B\left( \frac{x - x_0}{\mu} \right) = \frac{\mu^{\frac{N-2}{2}}}{(\mu^2 + |x-x_0|^2)^{\frac{N-2}{2}}}. \]
We notice that $B_{\mu, x_0}$ satisfies $-\Delta B_{\mu, x_0} = N(N-2) B_{\mu,x_0}^\frac{N+2}{N-2}$ on $\R^N$, for every $\mu > 0$ and $x_0 \in \R^N$. 

For multiindices $\alpha \in \N_0^2$, we consider functions $W_{\alpha}$ which satisfy
\begin{equation}
\label{Wjk equation}
-\Delta W_{\alpha} - N(N+2) B^\frac{4}{N-2} W_{\alpha} = f_\alpha \quad \text{ on } \R^N, \quad W_{\alpha}(x) = \frac{1}{\alpha!} x^\alpha + o(|x|^2), 
\end{equation}
with 
\begin{equation}
\label{fjk}
f_\alpha = \begin{cases}
0 & \text{ if } x^\alpha = x_j x_k \, \text{ for some } j \neq k, \\
-B & \text{ if } x^\alpha = x_j^2 \, \text{ for some } j
\end{cases}
\end{equation} 
We construct these functions in Lemma \ref{lemma Wjk} below. 

Here is our main result. 

\begin{theorem}
\label{theorem multibubble}
Let $(u_\eps)$ be a sequence of solutions to \eqref{brezis peletier additive}, with $V \in C^1(\overline{\Omega})$ and $V < 0$ on $\overline{\Omega}$, such that $\|u_\eps\|_\infty \to \infty$. Then there exists $n \in \N$ and $n$ sequences of points $x_\oeps,...,x_{n,\eps} \in \Omega$ such that $x_\ieps \to x_{i,0} \in \Omega$,  $\mu_\ieps:= u_\eps(x_\ieps)^{-\frac{2}{N-2}} \to 0$ as $\eps \to 0$, $\nabla u_\eps (x_\ieps) = 0$ for every $\eps > 0$ and $u_\eps \to 0$ uniformly away from $x_1,...,x_n$. The ratio $\lambda_\ieps:= \left(\frac{\mu_\ieps}{\mu_\oeps}\right)^\frac{N-2}{2}$ has a finite, non-zero limit $\lambda_{i,0} \in (0, \infty)$. 

Moreover, the following holds. 
 
\begin{enumerate}[(i)]
 \setcounter{enumi}{0}
\item \label{item local} \textbf{Refined local asymptotics:} For any $i= 1,...,n$, denote $B_\ieps := B_{\mu_\ieps, x_\ieps}$ and 
\begin{equation}
\label{Wieps definition}
W_\ieps(x) :=  \mu_\ieps^2 \sum_{\alpha \in \N_0^N: |\alpha|=2}  W_{\alpha}\left( \frac{x - x_\ieps}{\mu_\ieps} \right) \partial_{\alpha} \left( u_\eps - B_\ieps \right)(x_\ieps). 
\end{equation}
Then  for $\delta >0$ small enough, and every $\nu \in (2,3)$,
\[ |(u_\eps - B_\ieps -  W_\ieps)(x)| \lesssim \left( \eps \mu_\eps^{-\frac{N}{2} + 4 - \nu} + \mu_\eps^\frac{N-2}{2}  \right) |x-x_\ieps|^\nu  \] 
for all $x \in B(x_\ieps, \delta)$. 
\item \label{item blowup} \textbf{Blow-up rate:} 
The matrix $M(\bm{x}_0)$ is semi-positive definite with simple lowest eigenvalue $\rho(\bm{x}_0) \geq 0$.

\begin{itemize}
\item Suppose $\rho(\bm x_0) > 0$. If $N \geq 5$, then 
\begin{equation}
\label{mu asymptotics thm}
\lim_{\eps \to 0} \eps \mu_\ieps^{-N+4}  =: \kappa_{i,0}^{-2 \frac{N-4}{N-2}} 
\end{equation}
exists and lies in $(0, \infty)$. Moreover, $(\bm \kappa_0, \bm x_0)$ is a critical point of $F(\bm \kappa, \bm x)$ defined in \eqref{F definition}. If $N \geq 6$, then $\bm \kappa_0$ is the unique critical point of $F(\cdot, \bm x_0)$.

If $N = 4$, then for every $i$, 
\begin{equation}
\label{mu asymptotics thm N=4}
\lim_{\eps \to 0} \eps \ln (\mu_\ieps^{-1})  = \kappa_0,
\end{equation} 
where $\kappa_0 > 0$ is the unique number such that $M~-~\kappa_0 \text{ diag}(\frac{1}{8 \pi^2} |V(x_{i,0})|)$ has its lowest eigenvalue equal to zero. Moreover, $(\bm \lambda_0, \bm x_0)$ is a critical point of 
\begin{equation}
\label{tilde F definition}
\tilde{F}(\bm \lambda, \bm x) = \frac 12 \langle \bm{\lambda}, M(\bm{x}) \bm \lambda \rangle  +  \frac{\kappa_0}{2} \frac{1}{8 \pi^2} \sum_i V(x_{i}) \lambda_i^{2}.
\end{equation}

\item If $\rho(\bm x_0) = 0$, then also $\nabla \rho(\bm x_0) = 0$. Moreover,
\begin{align}
\label{mu asymptotics thm degenerate case}
\lim_{\eps \to 0} \eps \mu_\ieps^{-N+4}  &= \mathcal O(\mu_\eps^2) \quad \text{ if } N \geq 5 \, \text{ and } \\
\label{mu asymptotics thm degenerate case N =4}
\lim_{\eps \to 0} \eps \ln (\mu_\eps^{-1})& = \mathcal O(\mu_\eps^2) \quad \text{ if } N = 4,
\end{align} 
and $\ds \Lambda_{i,0} = \lambda_{i,0} =  \lim_{\eps \to 0} \left(\frac{\mu_\ieps}{\mu_\oeps}\right)^\frac{N-2}{2}$.\\
 
Furthermore, we have the quantitative bounds
\[
 \rho(\bm{x}_\eps) = 
 \begin{cases}
o(\eps \mu_\eps^{-N+4} + \mu_\eps^{2}) & \text{ if } N \geq 5, \\
o(\eps \ln (\mu_\eps^{-1}) + \mu_\eps^{2})  & \text{ if } N = 4,
\end{cases} 
\]
and, for every $\delta > 0$,
\[
|\nabla \rho (\bm{x}_\eps)| \lesssim \mu_\eps^{2-\delta}. 
\] 

\end{itemize}

\end{enumerate}   
\end{theorem}

\begin{remarks}
\begin{enumerate}[(a)]
\item 
In order to keep the statement of our theorem reasonable, in the refined local asymptotics, we just give the expansion up to the first term after the bubble. But, in fact we can go further, as shown by Proposition \ref{proposition p}. More precisely, our technique, which consists in subtracting recursively a suitable solution of the inhomogeneous linearized equation, can give bounds on remainder terms $q_\ieps^{(l)}$ \emph{of arbitrary order $l \in \N_0$}. Let us sketch the general framework. Indeed, set $q_\ieps^{(0)} = u_\eps - B_\ieps$ and define recursively
\[ W_\ieps^{(l+1)} := \sum_{\alpha \in \N_0^N \, : \, |\alpha| = l+2} W^{(l+1)}_\alpha\left(\frac{x- x_\ieps}{\mu_\ieps}\right) \partial_\alpha (q_\ieps^{(l)}(x_\ieps)), \]
where $W^{(l+1)}_\alpha$ is a solution to 
\[ (-\Delta - N(N+2)B^\frac{4}{N-2}) W^{(l+1)}_\alpha = f_\alpha(x, W^{(1)}, ..., W^{(l)}) \quad \text{ on } \R^N \]
with 
\[ W^{(l+1)}_\alpha(x) = \frac{1}{\alpha!} x^\alpha + o(|x|^{l+2}) \quad \text{ as } x \to 0. \]
Here, the inhomogeneities $f_\alpha$ are determined recursively from the equation satisfied by the remainder term. Then by carrying out the technique used in Section \ref{section asymptotic analysis}, the remainder terms 
\[ q_\ieps^{(l)} = u - B_\ieps - \sum_{k=1}^l W_\ieps^{(k)} \]
can be expected to satisfy the estimate 
\[ 
|q_\ieps(x)| \lesssim \left( \eps \mu_\eps^{-\frac{N}{2} + 3+l - \nu} + \mu_\eps^\frac{N-2}{2}  \right) |x-x_\ieps|^\nu  \] 
for all $x \in B(x_\ieps, \delta)$ and $\nu\in (l+1,l+2)$. We carry this program out rigorously for $l=1,2$ in this paper; see Propositions \ref{proposition q} and \ref{proposition p}. 

\item A remarkable fact in Theorem \ref{theorem multibubble} is to improve the asymptotic bounds on the blow-up speed in \eqref{mu asymptotics thm degenerate case} and \eqref{mu asymptotics thm degenerate case N =4} in the degenerate case when $\rho(\bm x_0) = 0$. Indeed, in this case (and only then) the first term on the right side of the expansions \eqref{expansion M proposition} resp. \eqref{expansion M proposition N =4} cancels, as shown in Section \ref{section proof of Theorem}. Our analysis of the error terms is fine enough to push the estimates further by a factor of $\mu_\eps^2$ in the expansions  \eqref{expansion M proposition} and \eqref{expansion M proposition N =4}. 

This should in particular be compared with the analysis of the related equation $-\Delta u_\eps = u_\eps^{\frac{N+2}{N-2}-\eps}$ in \cite{Bahri1995}, where in the case $\rho(\bm x_0) = 0$ no improved asymptotics are derived. 

\item We also point out that in the case $n=1$ of only one concentration point $x_0 \in\Omega$, one simply has $\rho(x_0) = \phi(x_0) > 0$ by the maximum principle. Thus the possibility that $\rho(\bm x_0) = 0$ is indeed particular to the multi-peak case. 

In the case where $\Omega$ is convex, it is known \cite[Theorem 2.7]{Grossi2010} that no multiple blow-up can happen. Under the weaker assumption that $\Omega$ is star-shaped with respect to some $y_0 \in \Omega$, the same is not known. However, a simple argument shows that if multiple blow-up does happen for $\Omega$ star-shaped, we must always be in the non-degenerate case $\rho(\bm x_0) > 0$. Indeed, by Pohozaev's identity we have 
\[ -\eps \mu_{1,\eps}^{-N+2} \int_\Omega (2 V(x) + \nabla V(x) \cdot (x-y_0)) u_\eps^2 \diff x = \mu_{1, \eps}^{-N+2} \int_{\partial \Omega} \left| \frac{\partial u_\epsilon}{\partial n}\right|^2 (x - y_0) \cdot n \diff x . \]
By Proposition \ref{proposition preliminaries multi}.(v) below, the right side converges to $\int_{\partial \Omega} \left| \frac{\partial G_{\bm x, \bm \nu}}{\partial n}\right|^2 (x - y_0) \cdot n \diff x > 0$. On the other hand, by standard calculations as in the proof of Proposition \ref{proposition expansion}, the left side is equal to 
\begin{equation}
\label{sum Vj}
\begin{cases}
-\eps \mu_{1,\eps}^{-N+2} c_N \sum_j V(x_\jeps) \mu_\jeps^2 + o(\mu_\eps^2) & \text{ if } N \geq 5 , \\
-\eps \mu_{1,\eps}^{-2} c_4 \sum_j V(x_\jeps) \mu_\jeps^2 \ln (\mu_\jeps^{-1}) & \text{ if } N = 4. 
\end{cases} 
\end{equation} 
Since $V < 0$ by assumption and all the $\mu_\jeps$ are comparable by Proposition \ref{proposition preliminaries multi}, the left hand side is equal to a positive constant times $\eps \mu_\eps^{-N+4}$ if $N \geq 5$, respectively $\eps \ln (\mu_\eps^{-1})$ if $N = 4$. Since we have seen that the right side is strictly positive, the quantities $\eps \mu_\eps^{-N+4}$, resp. $\eps \ln (\mu_\eps^{-1})$, must have a strictly positive limit. In particular, $\rho(\bm x_0) > 0$ by Theorem \ref{theorem multibubble}. 

\item One may ask whether our hypothesis that $V < 0$ on $\overline{\Omega}$ can be further relaxed. Concerning this question, a few comments are in order. Firstly, if $\Omega$ is star-shaped, then by Pohozaev's identity as in Remark (c) it is clear that for $V \equiv \text{const. } > 0$ there cannot be a solution $u_\eps$ to \eqref{brezis peletier additive}. For non-constant $V$, the situation is less clear. Still for star-shaped $\Omega$, say, the quantity \eqref{sum Vj} seems to suggest that at some blow-up points $x_j$, positive values $V(x_j) > 0$ might be allowed as long as they are compensated for by others. On the other hand, we are not aware of examples in the literature for a blow-up pattern different from that of Theorem \ref{theorem multibubble} (e.g. by exhibiting unbounded energy, clusters of concentration points and/or concentration on the boundary) in a situation where $V$ is not strictly negative. We point out that both our a priori analysis in Appendix \ref{appendixB} and the proof of our main results in  Section \ref{section proof of Theorem} require that $V < 0$ everywhere, independently from each other. 

\item Surprisingly, the concentration speed is uniquely determined in terms of $\Omega$, $V$, $n$ and $x_0$ in dimensions $N=4$ and $N \geq 6$, but not $N = 5$. Indeed, in that case we cannot exclude that the function $F$ may fail to be convex. 
\end{enumerate}
\end{remarks}

The structure of the rest of this paper is as follows. In Section \ref{section asymptotic analysis}, starting from some qualitative information about the blow-up of $u_\eps$, we derive very precise pointwise bounds on $u_\eps$ near the concentration points, which form the technical core of our method. These are used in turn to derive the main energy expansions in Section \ref{section main expansion}. Once these are established, the proof of Theorem \ref{theorem multibubble} can be concluded in Section \ref{section proof of Theorem} by rather soft argument. We have added several appendices in an attempt to make the analysis self-contained.

\section{Asymptotic analysis}
\label{section asymptotic analysis}
We start with some by now classical estimates, which says that a blowing-up sequence can only develop finitely many bubbles and the solutions are controlled by the bubble. Here the hypothesis $V<0$ plays a crucial role. This kind of analysis has been initiated by Druet, Hebey and Robert \cite{Druet2004} on a manifold. In the domain case an extra difficulty occurs since we have to avoid concentration near the boundary. This has alredy been done in dimension $N=3$ by Druet and the second author \cite{Druet2010} in a similar context. In higher dimension $N\geq 4$ the proof is largely analogous. We give it in Appendix B, for the sake of completeness and in the hope of providing a useful future reference for the case of a domain $\Omega$.

\begin{proposition}
\label{proposition preliminaries multi}
Let $(u_\eps)$ be a sequence of solutions to \eqref{brezis peletier additive} such that $\Vert u_\eps\Vert_\infty \rightarrow +\infty$. Then, up to extracting a subsequence, there exists $n \in \N$ and points $x_{1,\eps},..., x_{n,\eps}$ such that the following holds. 
\begin{enumerate}[(i)]
\item $x_\ieps \to x_i \in \Omega$ for some $x_i \in \Omega$ with $x_i \neq x_j$ for $i \neq j$. 
\item $\mu_\ieps := u_\eps(x_\ieps)^{-\frac{2}{N-2}} \to 0$ as $\eps \to 0$ and $\nabla u_\eps(x_\ieps) = 0$ for every $i$. 
\item \label{item lambda i} $\displaystyle \lambda_{i,0} := \lim_{\eps \to 0} \lambda_\ieps :=  \lim_{\eps \to 0} \frac{\mu_\ieps^\frac{N-2}{2}}{\mu_\oeps^\frac{N-2}{2}}$ exists and lies in $(0, \infty)$ for every $i$. 
\item \label{item u to bubble} $\mu_{i,\eps}^{\frac{N-2}{2}} u_\eps(x_{i,\eps} + \mu_{i,\eps} x) \to B$ in $C^1_\text{loc}(\R^N)$. 
\item \label{item tilde G} There are $\nu_i > 0$ such that $\displaystyle \mu_{1,\eps}^{-\frac{N-2}{2}} u_{\eps} \to  \sum_{i=1}^{n} \nu_i G(x_{i,\eps}, \cdot) =: G_{\bm{x},\bm{\nu}}$ uniformly in $C^1$ away from $\{x_1,...,x_n\}$.
\item There is $C > 0$ such that $ u_\eps \leq C \sum_{i=1}^{n} B_\ieps$ on $\Omega$.  Moreover, on every compact subset of $\Omega$, there is $C > 0$ such that  $\frac{1}{C} \sum_{i=1}^{n} B_\ieps \leq u_\eps$.
\end{enumerate}
\end{proposition}

Up to reordering the $x_{i,\eps}$, we assume that $\mu_{1,\eps}=\max_i \mu_{i,\eps}$ and we set $\mu_\eps=\mu_{1,\eps}$.

We also define the small ball 
\[ \mathsf{b}_\ieps := B(x_\ieps, \delta_0) \]
around $x_\ieps$, with some number $\delta_0 >0$ independent of $\eps$ and chosen so small that $\delta_0 < \frac 12  \min_{i\neq j} |x_\ieps - x_\jeps|$.  

The main result of this section consists in quantitative bounds on the remainder 
\begin{equation}
\label{r definition}
 r_\ieps := u_\ieps - B_\ieps 
\end{equation}
as well as the improved remainders 
\begin{equation}
\label{q definition}
q_\ieps := r_\ieps - W_\ieps
\end{equation} 
and
\begin{equation}
\label{p definition}
p_\ieps := 
q_\ieps -  \widetilde{W}_\ieps
\end{equation} 
 on $\mathsf b_\ieps$. Here, $W_\ieps$ is the function defined in \eqref{Wieps definition}. Similarly, the term $\widetilde{W}_\ieps$ is defined to be 
 \begin{equation}
 \label{Wieps tilde definition}
 \widetilde{W}_\ieps :=  \mu_{\ieps}^3 \sum_{\alpha \in \N_0^N \, : \, |\alpha| = 3} \widetilde{W}_\alpha \left( \frac{x-x_\ieps}{\mu_\ieps} \right) \partial_{\alpha} (q_\ieps)(x_\ieps),
  \end{equation}
 for functions $\widetilde{W}_\alpha$ satisfying, for every multiindex $\alpha \in \N_0^N$ with $|\alpha| = 3$,
 \begin{equation}
\label{Wjkl equation}
-\Delta \widetilde{W}_\alpha - N(N+2) B^\frac{4}{N-2} \widetilde{W}_\alpha = f_\alpha \quad \text{ on } \R^N, \quad \widetilde{W}_\alpha(x) = \frac{1}{\alpha!} x^\alpha + o(|x|^3) \quad \text{ as } x \to 0,
\end{equation}
with
\begin{equation}
\label{fjkl equation}
f_\alpha = \begin{cases}
0 & \text{ if } x^\alpha = x_j x_k x_l \, \text{ for some } j \neq k \neq l \neq j, \\
-B x_l & \text{ if } x^\alpha = x_j^2 x_l \text{ for some } j, l. 
\end{cases}
\end{equation}
We construct the functions $W_{jkl}$ in Lemma  \ref{lemma Wjkl} below. 

Notice that the correction terms $W_\ieps$ and $\widetilde{W}_\ieps$ are chosen to ensure that 
\begin{equation}
\label{q derivative zero}
\partial_\alpha q_\ieps(x_\ieps) = 0 \quad \text{ for all } \alpha \in \N_0^N \, \text{ with } |\alpha| \leq 2
\end{equation} 
and
\begin{equation}
\label{p derivative zero}
\partial_\alpha p_\ieps(x_\ieps) = 0 \quad \text{ for all } \alpha \in \N_0^N \, \text{ with } |\alpha| \leq 3. 
\end{equation} 

 The bounds on $r_\ieps$, $q_\ieps$ and $p_\ieps$ are stated in the subsections below as Propositions \ref{proposition r eps}, \ref{proposition q} and \ref{proposition p}.

An important ingredient in the proof of Theorem \ref{theorem multibubble} will be a non-degeneracy property of the bubble $B$. Namely, consider the linearized equation
\begin{equation}
\label{lin eq}
-\Delta u = N(N+2)  B^{\frac{4}{N-2}} u \qquad \text{ on } \R^N. 
\end{equation}
Then the behavior of non-trivial solutions to \eqref{lin eq} is restricted by the following proposition \cite[Corollary A.2]{KL}. 

\begin{proposition}
\label{proposition u = 0}
Let $u$ be a solution to \eqref{lin eq} and suppose that $|u(x)| \lesssim |x|^\tau$ on $\R^{N}$ for some $\tau \in (1, \infty) \setminus \N$. Then $u \equiv 0$. 
\end{proposition}

Before we go on, let us note a simple a priori estimate which will simplify the following estimates on $r_\ieps$ and $q_\ieps$. 

\begin{lemma}
\label{lemma eps lesssim mu}
Suppose that $V < 0$. If $N \geq 5$, then $\eps \lesssim \mu_\eps^{N-4}$. If $N = 4$, then $\eps \lesssim \frac{1}{ \ln (\mu_\eps^{-1})}$. 
\end{lemma}

\begin{proof}
By Pohozaev's identity (see Appendix \ref{Pohozaevidentities}), we have, for any $i$,  
\begin{align*}
& \qquad -2 \eps \int_{{\mathsf{b}_\ieps}} V u_\eps^2 - \eps \int_{\mathsf{b}_\ieps} u_\eps^2 \nabla V(x)\cdot(x-x_\ieps) \diff x \\
&=  2 \int_{\partial {\mathsf{b}_\ieps} } \left(\delta_0 \left(\partial_\nu u_\eps\right)^2 - \delta_0 \left( |\nabla u_\eps|^2 + \frac{2 u_\eps^{p+1}}{p+1} - \eps V u_\eps^2 \right) + (N-2) u_\eps \partial_\nu u_\epsilon  \right) . 
\end{align*} 
Since $V < 0$, by using Proposition \ref{proposition preliminaries multi}.\eqref{item u to bubble},  the left side is proportional to $\eps \mu_\eps^2$ if $N \geq 5$ and to $\eps \mu_\eps^2 \ln (\mu_\eps^{-1})$ if $N = 4$. 

On the other hand, by Proposition \ref{proposition preliminaries multi}.\eqref{item tilde G} the modulus of the right side is bounded by a constant times $\mu_\eps^{N-2}$. This concludes the proof. 
\end{proof}

    \subsection{The bound on $r_\ieps$}

\begin{proposition}
\label{proposition r eps}
Let $i=1,...,n$ and let $r_\ieps$ be defined by \eqref{r definition}. As $\eps \to 0$, for every $\theta\in (0,1) \cup (1,2)$ and, 
\[ |r_\ieps(x)| \lesssim (\eps \mu_\eps^{-\frac{N}{2} + 3 - \theta} + \mu_\eps^\frac{N-2}{2})  |x-x_\ieps|^{\theta} \qquad \text{ on } \mathsf{b_\ieps}. \]
Moreover, for $\theta = 0$, we have 
\[ r_\ieps(x) \lesssim
\begin{cases}
\eps \mu_\eps^{-\frac{N}{2} +3} + \mu_\eps^\frac{N-2}{2} & \text{ if } N \geq 5, \\
\eps \mu_\eps \ln (\mu_\eps^{-1}) + \mu_\eps  & \text{ if } N = 4. 
\end{cases}
\]
\end{proposition}

\begin{proof}
We first assume that $\theta \in (0,1) \cup (1,2)$. The case $\theta = 0$ will be treated below be a separate argument. 

Recall $r_\ieps= u_\eps - B_\ieps$. We denote  
\begin{equation}
\label{R eps definition}
R_\ieps(x):= \frac{r_\ieps(x)}{|x-x_\ieps|^{\theta}}.
\end{equation}
Fix some $z_\ieps \in \mathsf{b}_\ieps$ such that
\begin{equation}
\label{R eps (z eps)}
R_\ieps (z_\ieps) \geq \frac 12 \|R_\ieps \|_{L^\infty(\mathsf{b}_\ieps)}. 
\end{equation}
Moreover, we denote $d_\ieps:= |x_\ieps - z_\ieps|$. Let us define the rescaled and normalized version
\begin{equation}
\label{bar r eps definition}
\bar{r}_\ieps(x):= \frac{r_\ieps(x_\ieps + d_\ieps x)}{r_\ieps(z_\ieps)}, \qquad x \in B(0, d_\ieps^{-1} \delta_0). 
\end{equation}
By the choice of $B_\ieps$, and observing \eqref{R eps (z eps)}, we have 
\begin{equation}
\label{r lesssim x vartheta}
\bar{r}_\ieps(0)= \nabla \bar{r}_\ieps(0) = 0, \qquad  \bar{r}_\ieps(x) \lesssim |x|^{\theta}, \qquad x \in B(0, d_\ieps^{-1} \delta_0),
\end{equation}
in particular $\bar{r}_\eps$ is uniformly bounded on compacts of $\R^N\setminus\{0\}$. 

On $B(0, d_\ieps^{-1} \delta_0)$, we have
\begin{equation}
\label{bar r eps equation}
-\Delta \bar{r}_\ieps - \bar{r}_\ieps d_\ieps^2 Q(\bar{u}_\ieps, \bar{B}_\ieps)  = - \eps d_\ieps^2 \bar{V}_\ieps \frac{\bar{u}_\ieps}{r_\ieps(z_\ieps)} , 
\end{equation}
where $Q(u,v):= N(N-2) \frac{u^\frac{N+2}{N-2}-v^\frac{N+2}{N-2}}{u-v}$. Moreover we wrote $\bar{u}_\ieps(x) :=  u_\eps(x_\ieps + d_\ieps x)$ and likewise $\bar{a}_\ieps (x) :=  a_\eps(x_\ieps + d_\ieps x)$ and $\bar{B}_\ieps(x) :=  B_\ieps(x_\ieps + d_\ieps x) = \mu_\ieps^{-\frac{N-2}{2}} B (\mu_\ieps^{-1} d_\ieps x)$. 

We treat three cases separately, depending on the ratio between $\mu_\eps$ and $d_\ieps$. It will be useful to observe the bounds 
\begin{equation}
\label{B bar bounds}
\bar{B}_\ieps(x) = \left( \frac{\mu_\ieps}{\mu_\ieps^2 + d_\ieps^2|x|^2} \right)^\frac{N-2}{2} \lesssim 
\begin{cases}
\mu_\ieps^{-\frac{N-2}{2}}, &\text{ if } \mu_\ieps \gtrsim d_\ieps, \\
\mu_\ieps^\frac{N-2}{2} d_\ieps^{-N+2} , &\text{ if } \mu_\ieps \lesssim  d_\ieps, 
\end{cases}
\end{equation}
uniformly for $x$ in compacts of $\R^N \setminus \{0\}$. 

\textit{Case 1. $\mu_\eps >> d_\ieps$ as $\eps \to 0$. } Since $\bar{u}_\ieps \lesssim \bar{B}_\ieps$ on $\mathsf{b}_\ieps$ and $|Q(u,v)| \lesssim |u|^{\frac{4}{N-2}}+|v|^{\frac{4}{N-2}}$, the  second summand on the left side of \eqref{bar r eps equation}  tends to zero uniformly on compacts by \eqref{B bar bounds}, because $d_\ieps^2 \mu_\eps^{-2} \to 0$. 

Using $\bar{u}_\ieps \lesssim \bar{B}_\ieps \lesssim \mu_\ieps^{-\frac{N-2}{2}}$ and $\frac{1}{r_\ieps(z_\ieps)} \lesssim d_\ieps^{-\theta} \frac{1}{\|R_\ieps\|_\infty}$ by \eqref{R eps (z eps)}, the right side of \eqref{bar r eps equation} is bounded by 
\begin{align*}
\left| \eps d_\ieps^2 \bar{V}_\ieps \frac{\bar{u}_\ieps}{r_\ieps(z_\ieps)} \right| \lesssim \frac{\eps d_\ieps^{2-\theta} \mu_\ieps^{-\frac{N-2}{2}}}{\|R_\ieps\|_{L^\infty(\mathsf{b}_\ieps)}}  \lesssim \frac{\eps \mu_\eps^{-\frac{N}{2} + 3 - \theta}}{\|R_\ieps\|_{L^\infty(\mathsf{b}_\ieps)}},
\end{align*}

Now suppose for contradiction that $\|R_\ieps\|_{L^\infty(\mathsf{b}_\ieps)} >> \eps \mu_\eps^{-\frac{N}{2} + 3 - \theta}$ as $\eps \to 0$. Then this term goes to zero uniformly. Thus, by elliptic estimates, we have convergence on any compact of $\R^N\setminus \{0\}$, and the limit $\displaystyle \bar{r}_{i,0} :=  \lim_{\eps \to 0} \bar{r}_\ieps$ satisfies 
\[ -\Delta \bar{r}_{i,0} = 0 \qquad \text{ on } \R^N\setminus\{0\}. \]
By Bôcher's and Liouville's theorems, the growth bound \eqref{r lesssim x vartheta} implies that $\bar{r}_{i,0} \equiv 0$. But by the choice of $d_\ieps$, there is $\xi_\ieps := \frac{z_\ieps - x_\ieps}{d_\ieps} \in \mathbb S^{N-1}$ such that $\bar{r}_\ieps(\xi_\ieps) = 1$. Up to a subsequence, $\ds \xi_{i,0} :=  \lim_{\eps \to 0} \xi_\ieps \in \mathbb S^{N-1}$ exists and satisfies $\bar{r}_{i,0}(\xi_{i,0}) = 1$. This contradicts $\bar{r}_{i,0} \equiv 0$. 

Thus we must have $\|R_\ieps\|_{L^\infty(\mathsf{b}_\ieps)} \lesssim \eps \mu_\eps^{-\frac{N}{2} + 3 - \theta}$, i.e. $r_\ieps(x) \lesssim \eps \mu_\eps^{-\frac{N}{2} + 3 - \theta} |x-x_\ieps|^{\theta}$.

\textit{Case 2.a) $\mu_\eps << d_\ieps <<1$ as $\eps \to 0$. } In this case, we have 
\[ \bar{r}_\ieps d_\ieps^2 F(\bar{u}_\ieps, \bar{B}_\ieps) \lesssim d_\ieps^2 \bar{B}_\ieps^{\frac{4}{N-2}} \lesssim \mu_\eps^2 d_\ieps^{-2} \to 0 \]
and 
\[ \left| \eps d_\ieps^2 \bar{V}_\ieps \frac{\bar{u}_\ieps}{r_\ieps(z_\ieps)} \right| \lesssim \frac{\eps d_\ieps^{-N+4-\theta} \mu_\eps^{\frac{N-2}{2}}}{\|R_\ieps\|_{L^\infty(\mathsf{b}_\ieps)}}  \lesssim \frac{ \eps \mu_\eps^{-\frac{N}{2} +3 - \theta}}{\|R_\ieps\|_{L^\infty(\mathsf{b}_\ieps)}} \]
uniformly on compacts of $\R^N \setminus \{0\}$. 
If $\|R_\ieps\|_\infty >> \eps \mu_\eps^{-\frac{N}{2} + 3 - \theta}$, then, using that still $\ds d_\ieps \to 0$, $\bar{r}_{i,0} :=  \lim_{\eps \to 0} \bar{r}_\ieps$ satisfies 
\[ -\Delta \bar{r}_{i,0} = 0 \qquad \text{ on } \R^N \setminus \{0\}. \]
Using again the Bôcher and Liouville theorems, $\bar{r}_{i,0}\equiv 0$. As in Case 1, we can now derive a contradiction.  

Thus we must have $\|R_\ieps\|_{L^\infty(\mathsf{b}_\ieps)} \lesssim \eps \mu_\eps^{-\frac{N}{2} + 3 - \theta}$, i.e. $r_\ieps(x) \lesssim \eps \mu_\eps^{-\frac{N}{2} + 3 - \theta}|x-x_\ieps|^{\theta}$, also in this case. 

\textit{Case 2.b) $d_\ieps \sim 1$ as $\eps \to 0$. } In this case there is no need for a blow-up argument. Instead, we can simply bound, by the definition of $z_\ieps$, 
\[ \frac{|r_\ieps(x)|}{|x-x_\ieps|^{\theta}} \lesssim \frac{|r_\ieps(z_\ieps)|}{d_\ieps^{\theta}} \lesssim |r_\ieps(z_\ieps)| \lesssim \mu_\ieps^{\frac{N-2}{2}}, \]
where the last inequality simply comes from the bound $|u_\eps| \lesssim B_\ieps$ on $\mathsf{b}_\ieps$ and the observation that $d_\ieps \sim 1$ implies $B_\ieps(z_\ieps) \lesssim \mu_\eps^{\frac{N-2}{2}}$. Thus 
\[ |r_\ieps(x)| \lesssim \mu_\eps^{\frac{N-2}{2}} |x-x_\ieps|^{\theta}, \]
which completes the discussion of this case.

\textit{Case 3. $\mu_\eps \sim d_\ieps$ as $\eps \to 0$. } This is the most delicate case because the second summand on the left side of \eqref{bar r eps equation} now tends to a non-trivial limit. Indeed, $\displaystyle \beta_{i,0} :=  \lim_{\eps \to 0} \beta_\ieps := \lim_{\eps \to 0} \frac{\mu_\ieps}{d_\ieps}$ exists and $\beta_{i,0}  \in (0,\infty)$. Then 
\[ d_\ieps^\frac{N-2}{2} \bar{B}_\ieps = \frac{´\beta_\ieps^{\frac{N-2}{2}} }{(\beta_\ieps^2 + |x|^2)^{\frac{N-2}{2}}} \to  \frac{´\beta_{i,0}^{\frac{N-2}{2}} }{(\beta_{i,0}^2 + |x|^2)^{\frac{N-2}{2}}} = B_{0,\beta_{i,0}}. \]
By the convergence of $u_\eps$ from Proposition \ref{proposition preliminaries multi}, we also have $d_\ieps^{\frac{N-2}{2}} \bar{u}_\ieps \to B_{0,\beta_{i,0}}$ uniformly on compacts of $\R^N$. Thus $d_\ieps^2 Q(\bar{u}_\ieps, \bar{B}_\ieps) \to N(N+2) B_{0, \beta_{i_0}}$ uniformly on compacts of $\R^N$.

On the other hand, 
\[ \left|\eps d_\ieps^2 \bar{V}_\ieps \frac{\bar{u}_\ieps}{r_\ieps(z_\ieps)} \right| \lesssim \frac{\eps \mu_\eps^{-\frac{N}{2} + 3 - \theta}}{\|R_\ieps\|_{L^\infty(\mathsf{b}_\ieps)}} . \]
If $\|R_\ieps\|_{L^\infty(\mathsf{b}_\ieps)} >> \eps \mu_\eps^{-\frac{N}{2} + 3 - \theta}$, we therefore recover the limit equation
\[  -\Delta \bar{r}_{i,0} = N(N+2) \bar{r}_{i,0} B_{0, \beta_{i,0}}^\frac{4}{N-2}  \qquad \text{ on } \R^N, \]
which is precisely the linearized equation \eqref{lin eq}. By \eqref{r lesssim x vartheta}, we have $|r_{i,0}(x)| \lesssim |x|^{\theta}$ for all $x \in \R^N$. Thus by the classification, see \cite[Proposition A.1]{KL}, and the fact that $\bar{r}_{i,0}(0)=\nabla \bar{r}_{i,0}(0) = 0$,  we conclude $\bar{r}_{i,0} \equiv 0$. This contradicts $\bar{r}_{i,0}(\xi_{i,0}) = 1$, as desired. 

Thus we  have shown $\|R_\ieps\|_{L^\infty(\mathsf{b}_\ieps)} \lesssim \eps \mu_\eps^{-\frac{N}{2} + 3 - \theta}$, i.e. $r_\ieps(x) \lesssim \eps \mu_\eps^{-\frac{N}{2} + 3 - \theta} |x-x_\ieps|^{\theta}$, also in the third and final case. This finishes the proof for $\theta \in (0,1) \cup (1,2)$. 

Let us finally prove the assertion in case $\theta = 0$, i.e. 
\begin{equation}
\label{theta =0 bound}
 r_\ieps(x) \lesssim
\begin{cases}
\eps \mu_\eps^{-\frac{N}{2} +3} + \mu_\eps^\frac{N-2}{2} & \text{ if } N \geq 5, \\
\eps \mu_\eps \ln (\mu_\eps^{-1}) + \mu_\eps  & \text{ if } N = 4. 
\end{cases}
 \text{ for } x \in \mathsf b_\ieps. 
\end{equation}

To prove \eqref{theta =0 bound}, we consider the Green's formula 
\begin{align*}
r_\ieps(x) &= \int_\Omega (-\Delta r_\ieps)(y) G(x,y) \diff y - \int_{\partial \Omega} r_\ieps(y) \frac{\partial G(x,y)}{\partial \nu} \diff \sigma(y) \\
&= \int_\Omega (u_\eps^\frac{N+2}{N-2}(y) - B_\ieps^\frac{N+2}{N-2}(y) - \eps V(y) u_\eps(y))  G(x,y) \diff y - \int_{\partial \Omega} r_\ieps(y) \frac{\partial G(x,y)}{\partial \nu} \diff \sigma(y).
\end{align*} 
Since $r_\ieps \lesssim \sum_j B_\jeps \lesssim \mu_\eps^\frac{N-2}{2}$ on $\partial \Omega$, the second term is bounded by 
\[ \left| \int_{\partial \Omega} r_\ieps(y) \frac{\partial G(x,y)}{\partial \nu} \diff \sigma(y) \right| \lesssim \mu_\eps^\frac{N-2}{2}. \]
A similar bound, which we do not detail, gives 
\[ \left| \int_{\Omega \setminus \bigcup_j \mathsf b_\jeps} (-\Delta r_\ieps)(y) G(x,y) \diff y \right| \lesssim \eps \mu_\eps^\frac{N-2}{2} + \mu_\eps^\frac{N+2}{2} \lesssim 
\begin{cases}
\eps \mu_\eps^{-\frac{N}{2} +3} + \mu_\eps^\frac{N-2}{2} & \text{ if } N \geq 5, \\
\eps \mu_\eps \ln (\mu_\eps^{-1}) + \mu_\eps  & \text{ if } N = 4. 
\end{cases}
\]
To evaluate the remaining integral over $\mathsf b_\ieps$, we use 
\[ |-\Delta r_\ieps| = |u_\eps^\frac{N+2}{N-2} - B_\ieps^\frac{N+2}{N-2} - \eps V u_\ieps| \lesssim B_\ieps^\frac{4}{N-2} r_\ieps + \eps B_\ieps \qquad \text{ on } \mathsf b_\ieps.  \]
The term containing $\eps$ is bounded by 
\[ \eps \int_{\mathsf b_\ieps} B_\ieps \frac{1}{|x-y|^{N-2}} \diff y \leq \eps \int_{\mathsf b_\ieps} B_\ieps \frac{1}{|x_\ieps-y|^{N-2}} \diff y \lesssim 
\begin{cases}
\eps \mu_\eps^{-\frac{N}{2} + 3} &\text{ if } N \geq 5, \\
\eps \mu_\eps \ln (\mu_\eps^{-1}) &\text{ if } N = 4.
\end{cases}
\]
Here, the first inequality follows by the Hardy--Littlewood rearrangement inequality (see e.g. \cite[Theorem 3.4]{Lieb2001}),  because both $B$ and $z \mapsto |z|^{-N+2}$ are symmetric-decreasing functions. 
 
To control the last remaining term, we choose some $\theta \in (0,1) \cup (1,2)$ and reinsert the bound already proved for this $\theta$. This yields 
\begin{align*}
&\qquad \int_{\mathsf b_\ieps} B_\ieps^\frac{4}{N-2}( y) |r_\ieps(y)| \frac{1}{|x-y|^{N-2}} \diff y  \\
& \lesssim (\eps \mu_\eps^{-\frac{N}{2} + 3 - \theta} + \mu_\eps^{\frac{N-2}{2}}) \int_{\mathsf b_\ieps} B_\ieps^\frac{4}{N-2}(y) |x_\ieps - y|^\theta \frac{1}{|x-y|^{N-2}} \diff y \\
&= (\eps \mu_\eps^{-\frac{N}{2} + 3} + \mu_\eps^{\frac{N-2}{2} + \theta})  \int_{B(0, \delta_0 \mu_\ieps^{-1})} B^\frac{4}{N-2}(z) |z|^\theta \frac{1}{|z - \frac{x - x_\ieps}{\mu_\ieps}|^{N-2}} \diff z \\
& \leq (\eps \mu_\eps^{-\frac{N}{2} + 3} + \mu_\eps^{\frac{N-2}{2} + \theta}) \int_{\R^N} (1 + |z|^2)^{-2 + \frac{\theta}{2}} \frac{1}{|z - \frac{x - x_\ieps}{\mu_\ieps}|^{N-2}} \diff z \\
& \leq (\eps \mu_\eps^{-\frac{N}{2} + 3} + \mu_\eps^{\frac{N-2}{2} + \theta}) \int_{\R^N} (1 + |z|^2)^{-2 + \frac{\theta}{2}} \frac{1}{|z|^{N-2}} \diff z \\
& \lesssim \eps \mu_\eps^{-\frac{N}{2} + 3} + \mu_\eps^{\frac{N-2}{2} + \theta}.
\end{align*}
The second to last inequality follows again from the Hardy-Littlewood rearrangement inequality, because $z \mapsto (1 + |z|^2)^{-2 + \frac{\theta}{2}}$ and $z \mapsto |z|^{-N+2}$ are symmetric-decreasing functions. Combining all the above estimates,  the proof in case $\theta = 0$ is complete.  
\end{proof}

  \subsection{The bound on $q_\ieps$}

\begin{proposition}
\label{proposition q}
Let $i= 1,...,n$ and let $q_\ieps$ be defined by \eqref{q definition}. As $\eps \to 0$, for all $ \nu  \in  (2,3)$,
\[ |q_\ieps(x)| \lesssim \left( \eps \mu_\eps^{-\frac{N}{2} + 4 - \nu} + \mu_\eps^\frac{N-2}{2}  \right) |x-x_\ieps|^\nu  \qquad \text{ for all } x \in \mathsf{b}_\ieps. \]
\end{proposition}

\begin{proof}
Let $Q_\ieps(x):= \frac{q_\ieps(x)}{|x-x_\ieps|^\nu}$,  fix a point $z_\ieps$ with $Q_\ieps(z_\ieps) \geq \frac{1}{2} \|Q_\ieps\|_{L^\infty(\mathsf{b}_\ieps)}$ and let $d_\ieps := |x_\ieps - z_\ieps|$. 
(Notice that, by \eqref{q derivative zero} and Taylor's theorem, we have $\Vert Q_\ieps\|_{L^\infty(\mathsf{b}_\ieps)} < \infty$.)
When $d_\ieps \gtrsim 1$, we have 
\begin{equation}
\label{q eps estimate d sim 1}
Q_\ieps(x) \lesssim  \frac{q_\ieps(z_\ieps)}{d_\ieps^\nu} \lesssim |B_\ieps(z_\ieps)| + |\eps \mu_\ieps^{-\frac{N}{2} +3}  W\left(\frac{z_\ieps-x_\ieps}{\mu_\ieps}\right)| \lesssim \mu_\eps^\frac{N-2}{2} + \eps \mu_\eps^{-\frac{N}{2}+3} \lesssim \mu_\eps^\frac{N-2}{2}, 
\end{equation} 
where we used Lemma \ref{lemma eps lesssim mu} and the fact that $W$ is bounded by Lemma \ref{lemma W}. So it remains to treat the case $d_\ieps = o(1)$ in the following. 

In the following, let us assume $N \geq 6$. Then $\frac{N+2}{N-2} \leq 2$. 
Using the definition of the $W_{\alpha}$ in \eqref{Wjk equation} and of $W_\ieps$ in \eqref{Wieps definition}, we find 
\begin{align*}
-\Delta W_\ieps - N(N+2) B_\ieps^\frac{4}{N-2} W_\ieps &= \mu_\ieps^\frac{N-2}{2} B_\ieps (\Delta u_\eps - \Delta B_\ieps)(x_\ieps) \\
&= - \mu_\ieps^\frac{N-2}{2} B_\ieps \eps V(x_\ieps) u(x_\ieps) = - \eps V(x_\ieps) B_\ieps. 
\end{align*} 
With this, the equation satisfied by $q_\ieps$ can be written as 
\[ -\Delta q_\ieps - N(N+2) B_\ieps^{\frac{4}{N-2}} q_\ieps = \eps B_\ieps (V(x_\ieps) - V(x)) - \eps V r_\ieps + \mathcal O(r_\ieps^\frac{N+2}{N-2}), \quad \text{ on } \mathsf{b}_\ieps. \]
(When $N =4, 5$, and hence $\frac{N+2}{N-2} > 2$, the last term need to be replaced by $\mathcal O\left(r_\ieps^2 B_\ieps^{\frac{6-N}{N-2}}\right)$.) 

Then $\bar{q}_\ieps(x) := \frac{q_\ieps(x_\ieps + d_\ieps x)}{q_\ieps(z_\ieps)}$ satisfies 
\begin{align}
&\qquad  -\Delta \bar{q}_\ieps - N(N+2) B_\ieps^{\frac{4}{N-2}} \bar{q}_\ieps \nonumber = \\
& \frac{d_\ieps^{2 - \nu}}{ \|Q_\ieps\|_{L^\infty(\mathsf{b}_\ieps)}} \left( \eps \bar{B}_\ieps (V(x_\ieps) - \bar{V})) - \eps \bar{V} \bar{r}_\ieps + \mathcal O(\bar{r}_\ieps^\frac{N+2}{N-2}) \right) \nonumber \\
\end{align}
and
\[ \bar{q}_\ieps(0) = \nabla \bar{q}_\ieps (0) = 0, \qquad  |\bar{q}_\ieps(x) | \leq |x|^\nu \quad \text{ on } B(d_\ieps^{-1} \delta, 0) , \]
By Proposition \ref{proposition r eps} with $\theta=\nu-1$, 
$$\vert \bar{r}_{i,\eps}(x)\vert \leq (\eps \mu_\eps^{-\frac{N}{2}+4-\nu} +\mu_\eps^\frac{N-2}{2})d_\eps^{\nu-1}\vert x\vert$$ 
Then by Lemma \ref{lemma eps lesssim mu}, and using $N\geq 6$ and $d_\ieps \lesssim 1$, we  see that $\vert \bar{r}_\eps\vert \leq 1$ for $\epsilon$ small, which gives
\begin{align}
\qquad  -\Delta \bar{q}_\ieps - N(N+2) B_\ieps^{\frac{4}{N-2}} \bar{q}_\ieps \nonumber&= \frac{1}{\|Q_\ieps\|_{L^\infty(\mathsf{b}_\ieps)}} \mathcal O\left(\eps \bar{B}_\ieps d_\ieps^{3-\nu}  +d_\ieps^{2-\nu}\vert \bar{r}_{i,\eps}\vert \right)\\
&= \frac{1}{\|Q_\ieps\|_{L^\infty(\mathsf{b}_\ieps)}} \mathcal O\left(\eps \bar{B}_\ieps d_\ieps^{3-\nu}  + \eps \mu_\eps^{-\frac{N}{2}+4-\nu} +\mu_\eps^\frac{N-2}{2} \right). \label{q eq error terms}
\end{align} 
For completeness, we show  how to bound the term $\mathcal O\left(r_\ieps^2 B_\ieps^{\frac{N+2}{N-2}-2}\right)$ that occurs for $N = 4,5$. We have, by Proposition \ref{proposition r eps} with $2 \theta \in [\nu - 2, \nu - 2 + 6 - N]$, 
\begin{align*}
 d_\ieps^{2 - \nu} r_\ieps^2 B_\ieps^{\frac{N+2}{N-2}-2} &\lesssim 
\begin{cases}
(\eps^2 \mu_\eps^{- N + 6 - 2 \theta} + \mu_\eps^\frac{N-2}{2})  d_\ieps^{2 - \nu + 2\theta} \mu_\eps^{- \frac{6-N}{2}} & \text{ if } d_\ieps \lesssim \mu_\ieps, \\
 (\eps^2 \mu_\eps^{- N + 6 - 2 \theta} + \mu_\eps^\frac{N-2}{2})  d_\ieps^{2 - \nu + 2\theta + N - 6} \mu_\eps^{ \frac{6-N}{2}} & \text{ if } d_\ieps \gtrsim \mu_\ieps
\end{cases}  \\
&\lesssim \eps^2 \mu_\eps^{-\frac{N}{2}  + 5 - \nu} + \mu_\eps^\frac{N-2}{2}. 
\end{align*}
Let us now estimate the remaining first term on the right side of \eqref{q eq error terms}. 
By \eqref{B bar bounds} and the fact that $2 < \nu < 3$, we have 
\[ 
\eps \bar{B}_\ieps d_\ieps^{3-\nu} \lesssim 
\begin{cases}
\eps \mu_\eps^{-\frac{N-2}{2}} d_\eps^{3 - \nu} \lesssim  \eps \mu_\eps^{-\frac{N}{2} + 4 - \nu} & \text{ if } d_\ieps \lesssim \mu_\ieps, \\
\eps \mu_\eps^\frac{N-2}{2} d_\ieps^{-N+5-\nu} \lesssim \eps \mu_\eps^{-\frac{N}{2}+ 4 - \nu} & \text{ if } \mu_\ieps \lesssim d_\ieps << 1.
\end{cases}
\]

In both cases $d_\eps \lesssim \mu_\eps$ and $o(1) = d_\eps \gtrsim \mu_\eps$, the blow-up argument detailed in the proof of Proposition \ref{proposition r eps} now yields that $Q_\ieps$ is bounded by a constant times $\eps \mu_\eps^{-\frac N2 +4 - \nu} + \mu_\eps^\frac{N-2}{2}$. Taking into account \eqref{q eps estimate d sim 1}, we get the conclusion. 
\end{proof}

\subsection{The bound on $p_\ieps$}

\begin{proposition}
\label{proposition p}
Let $i= 1,...,n$ and let $p_\ieps$ be defined by \eqref{p definition}. As $\eps \to 0$, for all $ \nu  \in  (3,4)$,
\[ |p_\ieps(x)| \lesssim (\eps \mu_\eps^{-\frac{N}{2}+5-\nu} +\mu_\eps^\frac{N-2}{2}) |x-x_\ieps|^\nu  \qquad \text{ for all } x \in \mathsf{b}_\ieps. \]
\end{proposition}

\begin{proof}
The proof works exactly the same than the one of Propositions \ref{proposition r eps} and \ref{proposition q}. There is only one subtility we point out, the rest is exactly the same.
Let $P_\ieps(x):= \frac{p_\ieps(x)}{|x-x_\ieps|^\nu}$,  fix a point $z_\ieps$ with $P_\ieps(z_\ieps) \geq \frac{1}{2} \|P_\ieps\|_{L^\infty(\mathsf{b}_\ieps)}$ and let $d_\ieps := |x_\ieps - z_\ieps|$. 
(Notice that, by \eqref{p derivative zero} and Taylor's theorem, we have $\Vert P_\ieps\|_{L^\infty(\mathsf{b}_\ieps)} < \infty$.)

When $d_\ieps \gtrsim 1$, we have 
\begin{equation}
P_\ieps(x) \lesssim \mu_\eps^\frac{N-2}{2}, 
\end{equation} 
So it remains to treat the case $d_\ieps = o(1)$ in the following. We also assume $N \geq 6$. Then the equation satisfied by $p_\ieps$ can be written as 
\begin{equation}
\begin{split}
-\Delta p_\ieps - N(N+2) B_\ieps^{\frac{4}{N-2}} p_\ieps &= \eps B_\ieps (V(x_\ieps)+\nabla V(x_{i.\eps}.x - V(x))\\ 
&- \eps V (W_{i,\eps}+q_{i,\eps}) + \mathcal O(r_\ieps^\frac{N+2}{N-2}), \quad \text{ on } \mathsf{b}_\ieps.
\end{split}
\end{equation} 
(When $N =4, 5$, and hence $\frac{N+2}{N-2} > 2$, the last term need to be replaced by $\mathcal O\left(r_\ieps^2 B_\ieps^{\frac{N+2}{N-2}-2}\right)$. This term can be estimated identically to the proof of Proposition \ref{proposition q}. Notice that the range $2 \theta \in [\nu - 2, \nu - 2 + 6 - N]$ is still compatible with $\theta \in (0,2)$ and $\nu \in (3,4)$, and that the resulting bound $\eps^2 \mu_\eps^{-\frac{N}{2} + 5 - \nu} + \mu_\eps^\frac{N-2}{2}$ is strong enough also for the present case.) 

Then $\bar{p}_\ieps(x) := \frac{p_\ieps(x_\ieps + d_\ieps x)}{p_\ieps(z_\ieps)}$ satisfies 
\begin{align}
&\qquad  -\Delta \bar{p}_\ieps - N(N+2) B_\ieps^{\frac{4}{N-2}} \bar{p}_\ieps \nonumber \\
&=  \frac{d_\ieps^{2 - \nu}}{ \|P_\ieps\|_{L^\infty(\mathsf{b}_\ieps)}} \left( \eps \bar{B}_\ieps (\bar{V}(0)+\nabla\bar{V}(0) \cdot x - \bar{V})) - \eps \bar{V} \bar{W}_\ieps - \eps \bar{V} \bar{q}_\ieps+ \mathcal O((\bar{W}_{i,\eps}+\bar{q}_\ieps)^\frac{N+2}{N-2} \right) \nonumber 
\end{align}
and
\[ \bar{p}_\ieps(0) = \nabla \bar{p}_\ieps (0) = 0, \qquad  |\bar{p}_\ieps(x) | \leq |x|^\nu \quad \text{ on } B(d_\ieps^{-1} \delta, 0) . \]

By Proposition \ref{proposition q} applied with exponent $\nu-1 \in (2,3)$, 
\begin{equation}
\label{qest}
\vert \bar{q}_{i,\eps}(x)\vert \leq (\eps \mu_\eps^{-\frac{N}{2}+5-\nu} +\mu_\eps^\frac{N-2}{2})d_\eps^{\nu-1}
\end{equation}
hence, since $N\geq6$, as $\bar{W}_{i,\eps}$, $\vert \bar{q}_{i,\eps}(x)\vert\leq 1$ for $\epsilon$ small enough. Then
\begin{align}
  -\Delta \bar{p}_\ieps - N(N+2) B_\ieps^{\frac{4}{N-2}} \bar{p}_\ieps =  \frac{1}{ \|P_\ieps\|_{L^\infty(\mathsf{b}_\ieps)}} \left( \eps d_\ieps^{4 - \nu} \bar{B}_\ieps + \eps \mu_\eps^{-\frac{N}{2}+5-\nu} +\mu_\eps^\frac{N-2}{2} \right) \nonumber  .
\end{align}
Moreover we easily check, since $W(0)=\nabla W(0)=0$, that 
$$d_\ieps^{2 - \nu}\vert \bar{W}_{i,\eps}\vert= \mathcal O( \eps \mu^{-\frac{N}{2}+5-\nu})$$
which gives with \eqref{qest}
\begin{align*}
  -\Delta \bar{p}_\ieps - N(N+2) B_\ieps^{\frac{4}{N-2}} \bar{p}_\ieps  =  \frac{1}{ \|P_\ieps\|_{L^\infty(\mathsf{b}_\ieps)}} \left( \eps d_\ieps^{4 - \nu} \bar{B}_\ieps + \eps \mu_\eps^{-\frac{N}{2}+5-\nu} +\mu_\eps^\frac{N-2}{2} \right)   .
\end{align*}

Let us now estimate the remaining first term on the right side of \eqref{q eq error terms}. 
By \eqref{B bar bounds} and the fact that $3 < \nu < 4$, we have 
\[ 
\eps \bar{B}_\ieps d_\ieps^{4-\nu} \lesssim 
\begin{cases}
\eps \mu_\eps^{-\frac{N-2}{2}} d_\eps^{4 - \nu} \lesssim  \eps \mu_\eps^{-\frac{N}{2} + 5 - \nu} & \text{ if } d_\ieps \lesssim \mu_\ieps, \\
\eps \mu_\eps^\frac{N-2}{2} d_\ieps^{-N+6-\nu} \lesssim \eps \mu_\eps^{-\frac{N}{2}+ 5 - \nu} & \text{ if } \mu_\ieps \lesssim d_\ieps << 1.
\end{cases}
\]

In both cases $d_\eps \lesssim \mu_\eps$ and $o(1) = d_\eps \gtrsim \mu_\eps$, the blow-up argument detailed in the proof of Proposition \ref{proposition r eps} now yields that $P_\ieps$ is bounded by a constant times $\eps \mu_\eps^{-\frac N2 +5 - \nu} + \mu_\eps^\frac{N-2}{2}$.
\end{proof}

  \section{The main expansions}
\label{section main expansion}

We will also need the matrix $\tilde{M}^l(\bm{x}) \in \R^{n \times n} = (\tilde{m}_{ij}^l(\bm{x}))_{i,j = 1}^n$ with entries
\begin{equation}
\label{m ij tilde definition}
\tilde{m}_{ij}^l(\bm{x}) := 
\begin{cases}
\partial_l \phi(x_i) & \text{ for } i = j, \\
- 2 \partial_l^x G(x_i, x_j) & \text{ for } i \neq j. 
\end{cases}
\end{equation}

Recall that the matrix $M(\bm x)$ has been defined in \eqref{m ij definition}. 

The main results of this section are collected in the following two propositions. 

\begin{proposition}
\label{proposition expansion}
If $N \geq 5$, as $\eps \to 0$, 
\begin{equation}
\label{expansion M proposition}
 \sum_{j} m_{ij}(\bm{x}_\eps) \mu_\jeps^{\frac{N-2}{2}} = - d_N (V(x_\ieps) +o(1)) \eps \mu_\ieps^{-\frac{N}{2} + 3} +  \mathcal O(\mu_\eps^\frac{N+2}{2})
\end{equation}
where $d_N$ is given by \eqref{d_N constant}.   

If $N = 4$, as $\eps \to 0$, 
\begin{equation}
\label{expansion M proposition N =4}
  \sum_{j} m_{ij}(\bm{x}_\eps) \mu_\jeps = - \frac{1}{8 \pi^2} (V(x_\ieps) +o(1)) \eps \mu_\ieps \ln (\mu_\ieps^{-1}) +  \mathcal O(\mu_\eps^3)
\end{equation}
\end{proposition}


\begin{proposition}
\label{proposition expansion gradient}
If $N \geq 5$, as $\eps \to 0$, for every $l = 1,...,N$ and every $\delta > 0$,
\begin{equation}
\label{expansion gradient}
\sum_{j} \tilde{m}^l_{ij}(\bm{x}_\eps) \mu_\jeps^\frac{N-2}{2} = - d_N \frac{N-2}{2} \eps \mu_\eps^{-\frac{N}{2}+ 3} (\partial_{x_l} V(x_\ieps) + o(1)) +  \mathcal O(\mu_\eps^{\frac{N+2}{2} - \delta}),
\end{equation}
where $d_N$ is given by \eqref{d_N constant}.

If $N = 4$,  as $\eps \to 0$, for every $l = 1,...,N$ and every $\delta > 0$, 
\[ \sum_{j} \tilde{m}^l_{ij}(\bm{x}_\eps) \mu_\jeps^\frac{N-2}{2} = -\frac{1}{8 \pi^2} (\partial_{x_l} V(x_\ieps)+o(1)) \eps \mu_\eps \ln (\mu_\eps^{-1})  + \mathcal O(\mu_\eps^{3-\delta}). \]
\end{proposition}

\begin{proof}
[Proof of Proposition \ref{proposition expansion}]
We multiply equation \eqref{brezis peletier additive} by $G(x, x_\ieps)$ and integrate over $x$. Then the left side becomes 
\begin{equation}
\begin{split}
\label{leftside}
&\quad \int_\Omega (-\Delta u_\eps + \eps V u_\eps) G(x, x_\ieps) \diff x \\
& = u_\eps(x_\ieps) + \mu_\ieps^{-\frac N2 + 3} \eps V(x_\ieps) \int_{B(0, \delta_0\mu_\ieps^{-1})} B \frac{1}{\omega_{N-1}(N-2)|z|^{N-2}} \diff z  + o(\eps  \mu_\eps^{-\frac N2 + 3}) . 
\end{split}
\end{equation}

The right side is 
\begin{align}
&N(N-2) \int_\Omega u_\epsilon^\frac{N+2}{N-2} G(x,x_\ieps) \diff x \nonumber = N(N-2) \sum_j \int_{\mathsf b_\jeps} B_\jeps^\frac{N+2}{N-2}  G(x,x_\jeps) \diff x \label{proof expansion right side} \\
& \quad +  N(N+2) \int_{\mathsf b_\ieps} B_\ieps^{\frac{4}{N-2}} W_\ieps \frac{1}{\omega_{N-1}(N-2) |x-x_\ieps|^{N-2}} \diff x \nonumber  \\
& \quad + \mathcal O \Big( \int\displaylimits_{\mathsf b_\ieps} (B_\ieps^{\frac{4}{N-2}}  q_\ieps +|r_\ieps|^\frac{N+2}{N-2}  ) G(x,x_\ieps) \diff x + \int_{\mathsf b_\ieps} B_\ieps^\frac{4}{N-2} |r_\ieps| H(\cdot, x_\ieps) \diff x  \\
&\quad + \sum_{j \neq i} \int\displaylimits_{\mathsf b_\jeps} B_\jeps^{\frac{4}{N-2}} |r_\jeps| G(\cdot,x_\ieps) \diff x   +  \int\displaylimits_{\Omega \setminus \bigcup_j \mathsf b_\jeps} u_\eps^\frac{N+2}{N-2} G(\cdot, x_\ieps)\diff x  \Big) \nonumber 
\end{align}
When $N = 4,5$, similarly to the remark in the proof of Proposition \ref{proposition q}, the term $r_\ieps^\frac{N+2}{N-2}$ in the above error term needs to be replaced by $B_\ieps^{\frac{N+2}{N-2}-2} r_\ieps^2$. The ensuing estimates are very similar to the case $N \geq 6$ presented below and we leave the details to the reader. 

Let us first evaluate the two main terms in \eqref{proof expansion right side}. We have
\begin{align*}
&  \sum_j \int_{\mathsf b_\jeps} B_\jeps^\frac{N+2}{N-2}  G(x,x_\jeps) \diff x \\
&\quad = \int_{\mathsf b_\ieps} B_\ieps^\frac{N+2}{N-2} \left(\frac{1}{\omega_{N-1}(N-2) |x-x_\ieps|^{N-2}} - H(x, x_\ieps)\right) \diff x + \sum_{j \neq i }  \int_{\mathsf b_\ieps} B_\jeps^\frac{N+2}{N-2} G(x, x_\ieps) \diff x. 
\end{align*}
We compute the terms on the right side separately. First, by direct computation,
\[ \int_{\mathsf b_\ieps} B_\ieps^\frac{N+2}{N-2} \frac{N}{\omega_{N-1} |x-x_\ieps|^{N-2}}\diff x = \mu_\ieps^{-\frac{N-2}{2}} + \mathcal O(\mu_\eps^\frac{N+2}{2}). \]
Next, by radial symmetry of $B$ and the mean value property of the harmonic function $x \mapsto H(x, x_\ieps)$ , it is easy to see that
\begin{align*}
&\quad -N(N-2) \int_{\mathsf b_\ieps} B_\ieps^\frac{N+2}{N-2} H(x, x_\ieps) \diff x  = -N(N-2) \phi(x_\ieps) \int_{\mathsf b_\ieps} B_\ieps^\frac{N+2}{N-2}  \diff x \\
 &= -\omega_{N-1}(N-2) \mu_\ieps^\frac{N-2}{2} \phi(x_\ieps)+\mathcal O(\mu_\eps^\frac{N+2}{2}).  
\end{align*}  
Finally, by a similar argument, using that $G(x,x_\ieps)$ is harmonic for $x \in \mathsf b_\jeps$, for every $j \neq i$ we have 
\begin{align*}
&\quad N(N-2) \int_{\mathsf b_\jeps} B_\jeps^\frac{N+2}{N-2} G(x, x_\ieps) \diff x = N(N-2)  G(x_\jeps, x_\ieps) \int_{\mathsf b_\jeps} B_\jeps^\frac{N+2}{N-2} \diff x \\
& = \omega_{N-1}(N-2) \mu_\jeps^\frac{N-2}{2} G(x_\jeps, x_\ieps)  + \mathcal O(\mu_\eps^\frac{N+2}{2}).
\end{align*}
This completes the computation of the first main term of \eqref{proof expansion right side}.

We now treat the second main term of \eqref{proof expansion right side}, namely 
\begin{align*}
N(N+2) \int_{\mathsf b_\ieps} B_\ieps^{\frac{4}{N-2}} W_\ieps \frac{1}{\omega_{N-1}(N-2)|x-x_\ieps|^{N-2}} \diff x.
\end{align*}   
Recall that $W_\ieps(x) = \mu_\ieps^2 \sum_{|\alpha|=2} c_{\alpha,\eps} W_{\alpha}(\frac{x-x_\ieps}{\mu_\ieps})$ with $c_{\alpha, \eps} = \partial_{\alpha}(u_\eps - B_\ieps)(x_\ieps)$. Now if $x^\alpha= x_j x_k$ for some $j \neq k$, then $W_{\alpha}(x) = f(x) Y_2(x/|x|)$ for some spherical harmonic $Y_2$ of degree 2; see Lemma \ref{lemma Wjk}. Hence its integral over the ball $\mathsf{b}_\ieps$ against the radial function $|x-x_\ieps|^{-N+2} B_\ieps^\frac{4}{N-2}$ vanishes. Thus only the terms with $x^\alpha = x_j^2$ remain. In that case, $W_{\alpha}(x) = f(x) Y_2(x/|x|) + W(x)$, where $W(x) = W(|x|)$ is the function from Lemma \ref{lemma W}. Again, the term $f(x) Y_2(x/|x|)$ integrates to zero against $|x-x_\ieps|^{-N+2} B_\ieps^\frac{4}{N-2}$. Hence it remains to integrate against $|x-x_\ieps|^{-N+2} B_\ieps^\frac{4}{N-2}$ the term  
\begin{align*}
 \mu_\ieps^2 \sum_j \partial_{jj}(u_\eps - B_\ieps)(x_\ieps) W\left(\frac{x - x_\ieps}{\mu_\ieps}\right) &= \mu_\ieps^2 W\left(\frac{x - x_\ieps}{\mu_\ieps}\right) \Delta (u_\eps - B_\ieps)(x_\ieps) \\
 &=\eps V(x_\ieps) \mu_\ieps^2 W\left(\frac{x - x_\ieps}{\mu_\ieps}\right).
\end{align*} 
We obtain, using that $N(N+2)B^\frac{4}{N-2}  W = -\Delta W  + B$, 
\begin{align*}
& \qquad  N(N+2) \int_{\mathsf b_\ieps} B_\ieps^{\frac{4}{N-2}} W_\ieps \frac{1}{\omega_{N-1}(N-2)|x-x_\ieps|^{N-2}}  \diff x \\
&= \eps V(x_\ieps) N(N+2) \mu_\ieps^2 \int_{\mathsf b_\ieps} B_\ieps^{\frac{4}{N-2}}  W\left(\frac{x - x_\ieps}{\mu_\ieps}\right) \frac{1}{\omega_{N-1}(N-2)|x-x_\ieps|^{N-2}} \diff x \\
 &= \eps \mu_\ieps^{-\frac N2 + 3} V(x_\ieps)  \int_{B(0, \delta_0\mu_\ieps^{-1})}  (-\Delta W) \frac{1}{\omega_{N-1}(N-2) |z|^{N-2}}  \diff z \\
& \qquad + \eps \mu_\ieps^{-\frac N2 + 3} V(x_\ieps)  \int_{B(0, \delta_0\mu_\ieps^{-1})}  B \frac{1}{\omega_{N-1}(N-2) |z|^{N-2}} \diff z  + o(\eps  \mu_\eps^{-\frac N2 + 3}).
\end{align*}
The second term cancels precisely with the corresponding term in \eqref{leftside}. The term containing $\Delta W$ can be evaluated as follows. By the Green's formula and $W(0) = 0$, for every $R > 0$, 
\begin{align*}
\int_{B_R} (-\Delta W(z)) |z|^{-N+2} \diff z &= \int_{\partial B_R} W \frac{\partial |z|^{-N+2}}{\partial \nu} - \frac{\partial W}{\partial \nu} |z|^{-N+2} \\
&= - \omega_{N-1} (W'(R) R + W(R) (N-2)).
\end{align*}
By Lemma \ref{lemma W} we have $W'(R) = o(R^{-1})$ and $W(R) \to \frac{c_N}{N-2}$ as $R \to \infty$, with $c_N = \frac{\Gamma(N/2) \Gamma((N-4)/2)}{\Gamma(N-1)}$ if $N \geq 5$, and $W'(R)= o(R^{-1} \ln R)$ and $W(R) = \frac{1}{2} \ln (R) + o(\ln R)$ if $ N=4$. Thus
\[\frac{\eps \mu_\ieps^{-\frac N2 + 3} V(x_\ieps)}{\omega_{N-1}(N-2)}  \int_{B(0, \delta_0 \mu_\ieps^{-1})}  (-\Delta W) \frac{1}{ |z|^{N-2}}  \diff z =
\begin{cases}
 - \frac{c_N}{N-2} \eps \mu_\ieps^{-\frac N2 + 3} (V(x_\ieps)+o(1))  , & N \geq 5, \\
 -\frac 12 \eps \mu_\ieps \ln \mu_\ieps^{-1} (V(x_\ieps) +o(1)), & N = 4.
 \end{cases}
 \]
Putting everything together and observing that the divergent terms $u(x_\ieps) = \mu_\ieps^{-\frac{N-2}{2}}$ cancel precisely, we obtain the assertion, 
provided that we can prove that the error terms from above are negligible, i.e.
\begin{align}
\nonumber
& \int_{b_\ieps} (B_\ieps^{\frac{4}{N-2}}  q_\ieps + r_\ieps^\frac{N+2}{N-2} ) G(x,x_\ieps) \diff x + \sum_{j \neq i} \int_{\mathsf b_\jeps} B_\jeps^{\frac{4}{N-2}} r_\jeps G(x,x_\ieps) \diff x  \\
& \quad +  \int_{\Omega \setminus \bigcup_j \mathsf b_\jeps} u_\eps^\frac{N+2}{N-2} G(x, x_\ieps)\diff x = o(\eps \mu_\eps^{-\frac{N}{2}+3})  + \mathcal O(\mu_\eps^{\frac{N+2}{2}}). \label{proof expansion error terms}
\end{align}

To bound the first error term, we apply Proposition \ref{proposition q} with $2<\nu <3$. Then 
\begin{align*}
\left|\int_{b_\ieps} B_\ieps^{\frac{4}{N-2}}  q_\ieps G(x,x_\ieps) \diff x \right| &\lesssim (\eps \mu_\eps^{-\frac{N}{2} + 4 - \nu} + \mu_\eps^\frac{N-2}{2} ) \mu_\eps^\nu \int_{B(0, \delta_0 \mu_\eps^{-1})} B^{\frac{4}{N-2}} |x|^{-N+2+\nu} \diff x \\
&\lesssim \eps \mu_\eps^{-\frac{N}{2} + 6 - \nu} + \mu_\eps^\frac{N+2}{2} = o(\eps \mu_\eps^{-\frac{N}{2} +3}) + \mathcal O(\mu_\eps^\frac{N+2}{2}) = 
\end{align*}
because $\nu <3$. For the next term, 
by Proposition \ref{proposition r eps} and straightforward estimates we obtain
\[ \int_{\mathsf b_\ieps} B_\ieps^\frac{4}{N-2} |r_\ieps| H(\cdot, x_\ieps) \diff x  \lesssim \eps \mu_\eps^{-\frac{N}{2} + 5} + \mu_\eps^\frac{N+2}{2} =  o\left(\eps \mu_\eps^{-\frac{N}{2} +3}\right) + \mathcal O\left(\mu_\eps^\frac{N+2}{2}\right). \]
Next, we observe 
\[ |r_\ieps|^\frac{N+2}{N-2} \lesssim (\eps \mu_\ieps^{-\frac{N}{2} + 3})^\frac{N+2}{N-2}  W(\mu_\eps^{-1}(x - x_\ieps))^\frac{N+2}{N-2} + |q_\ieps|^\frac{N+2}{N-2} \lesssim \mu_\eps^\frac{N+2}{2} + B_\ieps^\frac{4}{N-2} |q_\ieps| \]
where we used Lemma \ref{lemma eps lesssim mu}, $|q_\ieps| \lesssim B_\ieps$ and the fact that $W$ is bounded by Lemma \ref{lemma W}. Thus
\begin{align*}
\int_{b_\ieps}  r_\ieps^\frac{N+2}{N-2}  G(x,x_\ieps) \diff x &\lesssim \mu_\eps^\frac{N+2}{2} \int_{\mathsf{b}_\ieps} \frac{1}{|x-x_\ieps|^{N-2}} \diff x + \int_{\mathsf{b}_\ieps} B_\ieps^{\frac{4}{N-2}} q_\ieps G(x,x_\ieps) \diff x  \\
&= o\left(\eps \mu_\eps^{-\frac{N}{2} +3}\right) + \mathcal O\left(\mu_\eps^\frac{N+2}{2}\right),
\end{align*}
by the bound we already proved.

Next, for any $j \neq i$, by Proposition \ref{proposition r eps}, for fixed $\theta \in (0,2)$ we estimate
\begin{align*}
\left| \int_{\mathsf b_\jeps} B_\jeps^{\frac{4}{N-2}} r_\jeps G(x,x_\ieps) \diff x \right| &\lesssim \left(\eps \mu_\eps^{-\frac{N}{2} + 3 - \theta} + \mu_\eps^\frac{N-2}{2}\right) \mu_\eps^{-2+N + \theta} \int_{B(0, \delta_0 \mu_\eps^{-1})} B^{\frac{4}{N-2}} |x|^\theta \diff x \\
&\lesssim \eps \mu_\eps^{-\frac{N}{2} + 5 - \theta}  + \mu_\eps^\frac{N+2}{2} = o(\eps \mu_\eps^{-\frac{N}{2} + 3}) + \mathcal O( \mu_\eps^\frac{N+2}{2})
\end{align*}
because $\theta < 2$. Finally to estimate the last remaining term in \eqref{proof expansion error terms}, we simply recall $u_\eps \lesssim \sum_j B_\jeps \lesssim \mu_\eps^\frac{N-2}{2}$ as well as $G(x, x_\ieps) \lesssim 1$ on $\Omega \setminus \bigcup_j \mathsf b_\jeps$, so that 
\[  \int_{\Omega \setminus \bigcup_j \mathsf b_\jeps} u_\eps^\frac{N+2}{N-2} G(x, x_\ieps)\diff x \lesssim \mu_\eps^\frac{N+2}{2}. \]
This completes the proof of \eqref{proof expansion error terms}, and hence of the proposition. 
\end{proof}

\begin{proof}
[Proof of Proposition \ref{proposition expansion gradient}]
The overall strategy and the nature of the multiple estimates needed is very similar to the preceding proof of Proposition \ref{proposition expansion}, which is why in the following we will be shorter in places. 

We multiply equation \eqref{brezis peletier additive} against $\partial_{y_l} G(x, x_\ieps)$ and integrate over $\diff x$. Since by definition of $G$ and $x_\ieps$,
\[ \int_\Omega u_\eps \nabla_y G(x,x_\ieps) \diff x = \nabla u_\eps(x_\ieps) = 0, \]
the resulting identity is (for any fixed $l = 1,...,N$)
\begin{equation}
\label{proof gradient expansion main id}
\eps \int_\Omega V u_\eps \partial_{y_l} G(x, x_\ieps) \diff x  = N(N-2) \int_\Omega u_\eps^\frac{N+2}{N-2} \partial_{y_l} G(x, x_\ieps) \diff x. 
\end{equation} 
In the following, we will repeatedly decompose 
\[ \nabla_y G(x, x_\ieps) = \frac{1}{\omega_{N-1}} \frac{x - x_\ieps}{|x-x_\ieps|^N} - \nabla_y H(x, x_\ieps)\]
and use that $\nabla H(\cdot, x_\ieps)$ is bounded on $\Omega$. 

We first evaluate the left side. 
Since $u_\eps \lesssim \sum_j B_\jeps$, clearly
\[ \int_{\Omega \setminus \bigcup \mathsf{b}_\jeps} \eps V u_\eps \partial_{y_l} G(x, x_\ieps) \diff x = \mathcal O(\eps \mu_\eps^\frac{N-2}{2}) = 
\begin{cases}
 o(\eps \mu_\eps^{-\frac{N}{2}+3}) & \text{ if } N \geq 5, \\
 o (\eps \mu_\eps \ln (\mu_\eps^{-1})) & \text{ if } N =4.
\end{cases} 
 \]

On $\mathsf{b}_\ieps$, we have 
\[ \eps \int_{\mathsf{b}_\ieps}  V u_\eps \nabla_y H(x, x_\ieps) \diff x = \mathcal O (\eps \mu_\eps^{\frac{N-2}{2}})  = 
\begin{cases}
 o(\eps \mu_\eps^{-\frac{N}{2}+3}) & \text{ if } N \geq 5, \\
 o (\eps \mu_\eps \ln (\mu_\eps^{-1})) & \text{ if } N =4.
\end{cases}  \]
To evaluate the integrals involving the singular term of $\nabla G$, we also decompose $u_\eps = B_\ieps + r_\ieps$ and $V(x) = V(x_\ieps) + \nabla V(x_\ieps) \cdot (x - x_\ieps) + o(|x-x_\ieps|)$, as well as 

Then by antisymmetry the main term vanishes, namely 
\begin{align*}
\eps V(x_\ieps) \int_{\mathsf{b}_\ieps} B_\ieps \frac{(x-x_\ieps)_l}{|x-x_\ieps|^{N}} \diff x = 0.
\end{align*}
The gradient term, for every $l = 1,...,N$, yields, if $N \geq 5$,
\begin{equation}
\label{V cancel term}
\frac{\eps}{\omega_{N-1}} \partial_{x_l} V(x_\ieps) \int_{\mathsf{b}_\ieps} B_\ieps \frac{(x-x_\ieps)^2_l}{|x-x_\ieps|^{N}} \diff x = \frac{1}{\omega_{N-1}} \eps \mu_\ieps^{-\frac{N}{2} + 3} \partial_{x_l} V(x_\ieps) \int_{B(0, \delta_0 \mu_\ieps^{-1})} B \frac{z_l^2}{|z|^N} \diff z. 
\end{equation} 
If $N = 4$, this gives
\[  \frac{\eps}{\omega_{N-1}} \partial_{x_l} V(x_\ieps) \int_{\mathsf{b}_\ieps} B_\ieps \frac{(x-x_\ieps)^2_l}{|x-x_\ieps|^{N}} \diff x =  \frac 14 \eps \mu_\eps \ln (\mu_\eps^{-1}) (\partial_{x_l} V(x_\ieps) + o(1)). \]
If $N \geq 5$, this term will exactly cancel with another contribution coming from the error term in $q_\ieps$ on the right side. 

Finally, by the bound for $\theta = 0$ from Proposition \ref{proposition r eps} and Lemma \ref{lemma eps lesssim mu},
\[
\eps \int_{\mathsf{b}_\ieps} V  |r_\ieps| \frac{x-x_\ieps}{|x-x_\ieps|^N} \diff x \lesssim  
\begin{cases} \eps^2 \mu_\eps^{-\frac{N}{2} + 3 } + \eps \mu_\eps^\frac{N-2}{2} =  o(\eps \mu_\eps^{-\frac{N}{2}+3}) & \text{ if } N \geq 5, \\
\eps^2 \mu_\eps \ln(\mu_\eps^{-1}) + \eps \mu_\eps \lesssim \eps \mu_\eps = o\left( \eps\mu_\eps \ln (\mu_\eps^{-1})\right) &\text{ if } N =4.
\end{cases}
\]

Let us now turn to evaluating the right side of \eqref{proof gradient expansion main id}. Since 
\[ \int_{\Omega \setminus \bigcup \mathsf{b}_\jeps}  u_\eps^\frac{N+2}{N-2} \nabla_y G(x, x_\ieps) \diff x = \mathcal O(  \mu_\eps^\frac{N+2}{2}), \]
we only need to consider integrals over the balls $\mathsf{b}_\jeps$. On $\mathsf{b}_\ieps$, we split 
\[ \nabla_y G(x, x_\ieps) = \frac{1}{\omega_{N-1}(N-2)}\frac{x - x_\ieps}{|x-x_\ieps|^N} - \nabla_y H(x, x_\ieps). \]
To treat the singular term, write $u_\eps = B_\ieps + r_\ieps =  B_\ieps + W_\ieps + q_\ieps$ 
and expand $u_\eps^\frac{N+2}{N-2}$ in \eqref{proof gradient expansion main id}. By antisymmetry, the term involving $B_\ieps^\frac{N+2}{N-2}$ vanishes. Moreover, since $W_\ieps$ only contains spherical harmonics of degree $0$ and $2$, the term $B_\ieps^{\frac{4}{N-2}} W_\ieps$ also vanishes when integrated against $\frac{x-x_\ieps}{|x-x_\ieps|^N}$ (which is a radial function times a spherical harmonic of degree $1$). 
Thus
\begin{align*}
&\qquad \left| \int_{\mathsf{b}_\ieps} u_\eps^\frac{N+2}{N-2} \frac{x - x_\ieps}{|x-x_\ieps|^N} \diff x \right| \\
&= \int_{\mathsf{b}_\ieps} B_\ieps^{\frac{4}{N-2}} q_\ieps \frac{x - x_\ieps}{|x-x_\ieps|^N} \diff x  + \begin{cases}  \mathcal O\left(  (\eps \mu_\eps^{-\frac{N}{2}+3})^\frac{N+2}{N-2} + \mu_\eps^\frac{N+2}{2}\right)  = o( \eps \mu_\eps^{-\frac{N}{2} + 3}) + \mathcal O (\mu_\eps^{\frac{N+2}{2}} ) & \text{ if } N \geq 5, \\
 \mathcal O( +  \eps^3 \mu_\eps^{3} (\ln (\mu_\eps^{-1}))^3  + \mu_\eps^3) = o( \eps \mu_\eps \ln (\mu_\eps^{-1})) +\mathcal O( \mu_\eps^{3}) & \text{ if } N = 4,
\end{cases}
\end{align*} 
by Proposition \ref{proposition r eps} with $\theta = 0$.  

Let us extract the contribution from the term in $q_\ieps$. When $N = 4$, Proposition \ref{proposition q} yields 
\[  \int_{\mathsf{b}_\ieps} B_\ieps^{\frac{4}{N-2}} |q_\ieps| \frac{1}{|x-x_\ieps|^{N-1}} \diff x \lesssim \mu_\eps =  o(\eps \mu_\eps \ln (\mu_\eps^{-1})). \]
So for $N = 4$ the term is negligible. Let us now look at $N \geq 5$. By Proposition \ref{proposition p} with any $\nu \in (3,4)$, we have 
\[  \int_{\mathsf{b}_\ieps} B_\ieps^{\frac{4}{N-2}} |p_\ieps| \frac{1}{|x-x_\ieps|^{N-1}} \diff x \lesssim \eps \mu_\eps^{-\frac{N}{2} + 7 - \nu} + \mu_\eps^\frac{N+2}{2} =   o( \eps \mu_\eps^{-\frac{N}{2} + 3}) + \mathcal O (\mu_\eps^{\frac{N+2}{2}} ) \]
We now discuss the main contribution from $q_\ieps$, namely the term 
\[  \frac{N(N+2)}{\omega_{N-1}} \int_{\mathsf b_\ieps}B_\ieps^{\frac{4}{N-2}}  \widetilde{W}_\ieps \frac{(x- x_\ieps)_l}{|x- x_\ieps|^N} \diff x.  \]
Recall that $\widetilde{W}_\ieps(x) = \mu_\ieps^3 \sum_{|\alpha|=3} c_{\alpha,\eps} \widetilde{W}_{\alpha}(\frac{x-x_\ieps}{\mu_\ieps})$ with $c_{\alpha, \eps} = \partial_{\alpha}(q_\ieps)(x_\ieps)$. Now if $x^\alpha= x_j x_k x_l$ for some $j \neq k \neq l \neq j$, then $\widetilde{W}_{\alpha}(x) = f(x) Y_3(x/|x|)$ for some spherical harmonic $Y_3$ of degree 2; see Lemma \ref{lemma Wjkl}. Hence its integral over the ball $\mathsf{b}_\ieps$ against $|x-x_\ieps|^{-N} B_\ieps^\frac{4}{N-2} (x - x_\ieps)_l$ (which is a radial function times a spherical harmonic of degree $1$) vanishes. Thus only the terms with $x^\alpha = x_j^2 x_k$ remain. In that case, $\widetilde{W}_{\alpha}(x) = f(|x|) Y_e(x/|x|) + \widetilde{W}(x) \frac{x_k}{|x|}$, where $\widetilde{W}(x) = \widetilde{W}(|x|)$ is the function from Lemma \ref{lemma U}. Again, the term $f(|x|) Y_3(x/|x|)$ integrates to zero against $|x-x_\ieps|^{-N} B_\ieps^\frac{4}{N-2} (x - x_\ieps)_l$. Moreover, also the term $\widetilde{W}(x) \frac{x_k}{|x|}$ integrates to zero unless $k = l$. In summary, it remains to integrate against $|x-x_\ieps|^{-N} B_\ieps^\frac{4}{N-2} (x - x_\ieps)_l$ the term  
\begin{align*}
&\quad  \mu_\ieps^3 \sum_j \partial_{jjl}(q_\ieps)(x_\ieps) \widetilde{W}\left(\frac{x - x_\ieps}{\mu_\ieps}\right) \frac{(x - x_\ieps)_l}{|x - x_\ieps|} \\
  &= \mu_\ieps^3 \widetilde{W}\left(\frac{x - x_\ieps}{\mu_\ieps}\right) \frac{(x - x_\ieps)_l}{|x - x_\ieps|}\partial_l \Delta (q_\ieps)(x_\ieps) \\
  & =\mu_\ieps^{-\frac{N}{2} + 4} \widetilde{W}\left(\frac{x - x_\ieps}{\mu_\ieps}\right)  \frac{(x - x_\ieps)_l}{|x - x_\ieps|} \eps \partial_l V(x_\ieps). 
\end{align*} 
In conclusion, we can write 
\begin{align*}
& \quad \frac{N(N+2)}{\omega_{N-1}} \int_{\mathsf b_\ieps}B_\ieps^{\frac{4}{N-2}}  \widetilde{W}_\ieps \frac{(x- x_\ieps)_l}{|x- x_\ieps|^N} \diff x \\
&=  \frac{N(N+2)}{\omega_{N-1}} \eps \mu_\ieps^{-\frac{N}{2} + 4}  \partial_l V(x_\ieps) \int_{\mathsf b_\ieps}B_\ieps^{\frac{4}{N-2}} \widetilde{W}\left(\frac{x - x_\ieps}{\mu_\ieps}\right)  \frac{(x- x_\ieps)_l^2}{|x- x_\ieps|^{N+1}} \diff x \\
&= \frac{N(N+2)}{\omega_{N-1}} \eps \mu_\ieps^{-\frac{N}{2} + 3}  \partial_l V(x_\ieps) \int_{B(0, R_\eps)} B^\frac{4}{N-2}(z) \widetilde{W}(z) \frac{z_l^2}{|z|^{N+1}} \diff z,
\end{align*}
with $R_\eps = \delta_0 \mu_\ieps^{-1}$. Now we use the equation $-\Delta \widetilde{W} - N(N+2) B^\frac{4}{N-2} \widetilde{W} = -B |z|$ for $\widetilde{W}$. This gives 
\begin{align*}
& \quad \frac{N(N+2)}{\omega_{N-1}} \int_{\mathsf b_\ieps}B_\ieps^{\frac{4}{N-2}}  \widetilde{W}_\ieps \frac{(x- x_\ieps)_l}{|x- x_\ieps|^N} \diff x \\
&= \frac{1}{\omega_{N-1}} \eps \mu_\ieps^{-\frac{N}{2} + 3}  \partial_l V(x_\ieps) \left( \int_{B(0, R_\eps)} (-\Delta \widetilde{W})(z)  \frac{z_l^2}{|z|^{N+1}} \diff z +  \int_{B(0, R_\eps)} B(z)  \frac{z_l^2}{|z|^{N}} \diff z  \right) 
\end{align*}
The second summand cancels precisely with the term \eqref{V cancel term} from the left side pointed out above.
 The term in $-\Delta \widetilde{W}$, arguing as in the proof of Proposition \ref{proposition expansion}, gives
\begin{align*}
&\quad \frac{1}{\omega_{N-1}} \eps \mu_\eps^{-\frac{N}{2} + 3} \partial_l V(x_\ieps) \int_{B(0, R_\eps)} (-\Delta \widetilde{W}) \frac{z_l^2}{|z|^{N+1}} \diff z \\
&= - \frac{\partial_l V(x_\ieps)}{N} \eps \mu_\eps^{-\frac{N}{2} + 3} ((N-1)\widetilde{W}(R_\eps)R_\eps^{-1} + \widetilde{W}'(R_\eps))  \\
&= -\partial_l V(x_\ieps) \frac{a_N}{N} \eps \mu_\eps^{-\frac{N}{2} + 3} + o(\eps \mu_\eps^{-\frac{N}{2} + 3}),
\end{align*}
with $a_N$ as in Lemma \ref{lemma U}. 

This finishes the discussion of the term in $q_\ieps$. 

Now we evaluate the integral over $\mathsf{b}_\ieps$ against $\nabla_y H(x, x_\ieps)$, for which we  decompose again $u_\eps = B_\ieps + r_\ieps$.   Taylor expanding 
\[ \partial_{y_l} H(x, x_\ieps) = \partial_{y_l} H(x_\ieps, x_\ieps) + \nabla_x \partial_{y_l} H(x_\ieps, x_\ieps) \cdot (x - x_\ieps) + \mathcal O(|x-x_\ieps|^2), \]
and using that the gradient term cancels by antisymmetry, we find
\begin{align*}
- \int_{\mathsf{b}_\ieps} B_\ieps^\frac{N+2}{N-2}  \partial_{y_l} H(x, x_\ieps) \diff x &= - \frac{\omega_{N-1}}{N} \mu_\ieps^\frac{N-2}{2}  \partial_{y_l} H(x_\ieps, x_\ieps)  + \mathcal O( \mu_\eps^\frac{N+2}{2} \ln (\mu_\eps^{-1})) \\
&= -  \frac{\omega_{N-1}}{2N} \mu_\ieps^\frac{N-2}{2}  \partial_{x_l} \phi(x_\ieps)  + \mathcal O( \mu_\eps^{\frac{N+2}{2} - \delta})
\end{align*} 
which is (the diagonal part of) the main term we desired to extract. On the other hand, since $\nabla_y H(x, x_\ieps)$ is bounded, the principal remainder term in $r_\ieps$,  by Proposition \ref{proposition r eps} with $\theta \in [0,1)$, is bounded by
\begin{align*}
\int_{\mathsf{b}_\ieps} B_\ieps^{\frac{4}{N-2}} |r_\ieps| \diff x \lesssim  \eps \mu_\eps^{-\frac{N}{2} + 4 - \theta} + \mu_\eps^\frac{N+2}{2} = 
\begin{cases}
 o( \eps \mu_\eps^{-\frac{N}{2} + 3}) + \mathcal O( \mu_\eps^{\frac{N+2}{2} - \delta}) &\text{ if } N \geq 5, \\
 o( \eps \mu_\eps \ln (\mu_\eps^{-1}) ) + \mathcal O( \mu_\eps^{3 - \delta} ) & \text{ if } N = 4.
 \end{cases}
\end{align*}
Finally, on $\mathsf{b}_\jeps$ with $j \neq i$, analogous computations permit us to extract the remaining (off-diagonal) part of the main term as 
\[ \int_{\mathsf{b}_\jeps} u_\eps^\frac{N+2}{N-2} \partial_{y_l} G(x, x_\ieps) \diff x = \frac{\omega_{N-1}}{N} \partial_{y_l} G(x_\jeps, x_\ieps) \mu_\jeps^{\frac{N-2}{2}} + 
\begin{cases}
 o( \eps \mu_\eps^{-\frac{N}{2} + 3}) + \mathcal O( \mu_\eps^\frac{N+2}{2}) &\text{ if } N \geq 5, \\
o( \eps \mu_\eps \ln (\mu_\eps^{-1}) ) + \mathcal O( \mu_\eps^{3 - \delta}) & \text{ if } N = 4.
 \end{cases}
   \]
Combining everything, and observing that $\frac{2N}{N(N-2) \omega_{N-1}} \frac{a_N}{N} = d_N \frac{N-2}{2}$, with $d_N$ given by \eqref{d_N constant}, the proof is complete. 
\end{proof}

\section{Proof of Theorem \ref{theorem multibubble}}
\label{section proof of Theorem}

We now show how the expansions \eqref{expansion M proposition} and \eqref{expansion gradient} can be used to conclude the proof of Theorem \ref{theorem multibubble}. 

We introduce the vector $\bm \lambda_\eps \in (0,\infty)^n$ with components 
\[ (\bm \lambda_\eps)_i := \lambda_\ieps := \left(\frac{\mu_\ieps}{\mu_\oeps}\right)^\frac{N-2}{2}, \]
and note that $\lambda_\ieps$ is bounded away from $0$ and $\infty$ by Proposition \ref{proposition preliminaries multi}. 

Let us rewrite \eqref{expansion M proposition} and \eqref{expansion M proposition N =4} as 
\begin{equation}
\label{identity with lambda}
(M(\bm{x}_\eps) \cdot \bm{\lambda}_\eps)_i + \mathcal O(\mu_\eps^{2}) = 
\begin{cases}
 -d_N \eps \mu_\ieps^{-N+4} (V(x_\ieps) +o(1))  \lambda_\ieps  & \text{ if } N \geq 5, \\
- (8 \pi^2)^{-1} \eps \ln \mu_\ieps^{-1}  (V(x_\ieps) +o(1)) &\text{ if } N = 4. 
\end{cases}
\end{equation} 

By Perron-Frobenius theory (see \cite{Bahri1995}), the lowest eigenvalue $\rho(\bm x_\eps)$ of $M(\bm x_\eps)$ is simple and the associated eigenvector $\bm \Lambda(\bm x_\eps)$, normalized so that $(\bm \Lambda(\bm x_\eps))_1 = 1$, has strictly positive entries. 

Taking the scalar product of \eqref{identity with lambda} with $\bm \Lambda(\bm x_\eps)$ shows 
\begin{align}
\label{identity rho eps}
 \rho(\bm x_\eps) \langle \bm \Lambda(\bm x_\eps), \bm \lambda_\eps \rangle &= \langle \bm \Lambda(\bm x_\eps), M(\bm x_\eps) \cdot \bm \lambda_\eps \rangle  \\
 &=
\begin{cases} 
  - d_N \eps \sum_i  \mu_\ieps^{-N+4} (V(x_\ieps) + o(1)) \lambda_\ieps (\Lambda(x_\eps))_i + o(1) & \text{ if } N \geq 5, \\
 -(8 \pi^2)^{-1} \eps \sum_i \ln \mu_\ieps^{-1}  (V(x_\ieps) +o(1)) \lambda_\ieps (\Lambda(x_\eps))_i + o(1) & \text{ if } N = 4. 
  \end{cases} \nonumber
\end{align}
Since $\bm \Lambda(\bm x_\eps)$ and $\bm \lambda_\eps$ both have strictly positive entries, and since $V < 0$ by assumption, this shows that $0 < \rho(\bm x_\eps)$ for all $\eps > 0$. For the limit $\rho(\bm x_0)$, two cases are possible.

\textit{Case 1: $\rho(\bm x_0) > 0$. }

Assume $N \geq 5$ first. In this case, \eqref{identity rho eps} shows that $\displaystyle \lim_{ \eps \to 0} \eps \mu_\ieps^{-N+4} > 0$. (Note that this limit always exists up to a subsequence and is finite as a consequence of Lemma \ref{lemma eps lesssim mu}.)  

Introducing the variable 
\[ \kappa_\ieps := (\eps^{\frac{1}{-N+4}} \mu_\ieps)^\frac{N-2}{2} \]
we can write \eqref{expansion M proposition} as 
\begin{equation}
\label{identity with kappa}
(M(\bm{x}_\eps) \cdot \bm{\kappa}_\eps)_i = -d_N V(x_\ieps) \kappa_\ieps^{-q},
\end{equation}
with
\[ q := \frac{N-6}{N-2}. \]
Moreover, \eqref{expansion gradient} can be written in terms of $\bm \kappa_\eps$ as 
\begin{equation}
\label{identity gradient with kappa}
 (\tilde{M}^l(\bm x_\eps) \cdot \bm \kappa_\eps)_i = -d_N \frac{N-2}{2} \partial_{x_l} V(x_\ieps) \kappa_\ieps^{-q}. 
\end{equation}
Since $\partial_{\kappa_k} \langle \bm \kappa, M(\bm x) \bm \kappa \rangle = (M(\bm x) \cdot \kappa)_k$ and $\partial_{(x_k)_l} \langle \bm \kappa, M(\bm x) \bm \kappa \rangle = \frac{1}{2} \kappa_k (\tilde{M}^l(\bm x) \cdot \kappa)_k$



Thus $\bm \kappa_0$ is a critical point of $F: (0,\infty)^n \to \R$ defined by
\[
F (\bm{\kappa}) = \frac{1}{2} \sum_{i,j} m_{ij}(\bm{x}_0) \kappa_i \kappa_j - \frac{d_N}{1-q} \sum_i |V(x_{i,0})| \kappa_i^{1-q}.
\]

Since $\rho(\bm x_0) > 0$ in this case, $M(\bm x_0)$ is strictly positive definite. If additionally $q \geq 0$ (i.e. $N \geq 6$), then $D^2_{\bm \kappa} F(\kappa, \bm x_0)$ is strictly positive definite for every $\bm \kappa$. We obtain that $F(\bm \kappa, \bm{x}_0)$ is convex in the variable $\bm \kappa$ on $(0,\infty)$, hence it has a unique critical point. This is the desired characterization of $\kappa_0$, and hence of $\ds \lim_{\eps \to 0} \eps \mu_\ieps^{-N+4} = \kappa_{i,0}^{-2 \frac{N-4}{N-2}}$. 

If $N = 4$, we find in a similar way that $\ds \lim_{\eps \to 0} \eps \ln (\mu_\eps^{-1}) > 0$. To characterize the limit, we argue slightly differently. Since $ \eps \ln \mu_\ieps^{-1} = \eps \ln \mu_\oeps^{-1} + o(1) =: \kappa_0 +o(1)$, passing to the limit in \eqref{identity with lambda} gives
\begin{equation}
\label{identity with lambda N=4}
(M(\bm x_0) \cdot \bm \lambda_0)_i = \frac{1}{8 \pi^2} |V(x_{i,0})| \kappa_0 \lambda_{i,0}. 
\end{equation} 
Similarly, the identity from Proposition \ref{proposition expansion gradient} reads
\begin{equation}
\label{identity with lambda N=4}
(\tilde{M}^l(\bm x_0) \cdot \bm \lambda_0)_i = \frac{1}{8 \pi^2} |V(x_{i,0})| \kappa_0 \lambda_{i,0}.
\end{equation}
This shows that $(\bm \lambda_0, \bm x_0)$ is a critical for $\tilde{F}(\bm \lambda, \bm x)$ as given in \eqref{tilde F definition}. 

Let us finally discuss the property of $\kappa_0$. If we define $M_1(\kappa) := M(\bm x_0) - \frac{\kappa}{8 \pi^2} \text{diag}(|V(x_{i,0})|)$, this can be written as $M_1(\kappa_0) \cdot \bm \lambda_0 = 0$, i.e. $\bm \lambda_0$ is a zero eigenvalue of $M_1(\kappa_0)$. Since $M_1(\kappa)$ differs from $M(\bm x_0)$ only on the diagonal, the Perron-Frobenius arguments used above can still be applied to $M_1(\kappa)$. Thus $\bm \lambda_0$ must be the lowest eigenvector of $M_1(\kappa_0)$, because it has strictly positive entries. Since $V <0$, the lowest eigenvalue of $M_1(\kappa)$ is clearly a strictly monotone function of $\kappa$, so $\kappa_0$ is indeed unique with the property that the lowest eigenvalue of $M_1(\kappa_0)$ equals zero.

This completes the proof of Theorem \ref{theorem multibubble} in case $\rho(\bm x_0) > 0$. 

\textit{Case 2. $\rho(\bm x_0) = 0$. }

In this case, \eqref{identity rho eps} shows that $\ds \lim_{\eps \to 0} \eps \mu_\eps^{-N+4} = 0$ and that $\bm \lambda_0$ is an eigenvector with eigenvalue 0. Since $(\bm \lambda_0)_1 = 1 = (\bm \Lambda(\bm x_0))_1$ and $\rho(\bm x_0)$ is simple, we have in fact $\bm \lambda_0 = \bm \Lambda(\bm x_0)$, i.e. $\bm \lambda_0$ is precisely the lowest eigenvector of $M(\bm x_0)$, with eigenvalue $\rho(\bm x_0) = 0$.

For the following analysis, we decompose $\bm \lambda_\eps = \alpha_\eps \bm \Lambda (\bm x_\eps) + \bm \delta (x_\eps)$, where $\alpha_\eps \in \R$, $\bm \Lambda(\bm x_\eps)$ is the lowest eigenvalue of $M(\bm x_\eps)$ and $\bm \delta(\bm x_\eps) \bot \bm \Lambda (\bm x_\eps)$. Notice that $\alpha_\eps \to 1$ as a consequence of $\bm \lambda_\eps \to \bm \Lambda(\bm x_0)$. 

Here is the central piece of information which we need to conclude in this case. 

\begin{proposition}
\label{proposition bound on rho}
As  $\eps \to 0$, 
\begin{equation}
\label{delta bound}
|\bm \delta (\bm x_\eps)| =
\begin{cases}
 \mathcal O( \eps \mu_\eps^{-N+4} + \mu_\eps^2 \ln (\mu_\eps^{-1}) + |\rho(\bm x_\eps)|) & \text{ if } N \geq 5 , \\
\mathcal O( \eps \ln (\mu_\eps^{-1})  + \mu_\eps^2 \ln (\mu_\eps^{-1}) + |\rho(\bm x_\eps)|) & \text{ if } N = 4. 
\end{cases}
\end{equation} 
Suppose moreover that $\rho(\bm x_0) = 0$. Then, as $\eps \to 0$,  
\begin{equation}
\label{rho bound}
 \rho(\bm{x}_\eps) = 
 \begin{cases}
o(\eps \mu_\eps^{-N+4} + \mu_\eps^{2}) & \text{ if } N \geq 5, \\
o(\eps \ln (\mu_\eps^{-1}) + \mu_\eps^{2})  & \text{ if } N = 4.
\end{cases} 
\end{equation}
\end{proposition}

Before we prove Proposition \ref{proposition bound on rho}, let us use it to conclude the proof of Theorem \ref{theorem multibubble} in the present case $\rho(\bm x_0) = 0$. 

Taking the scalar product of identity \eqref{identity with lambda} with $\lambda_\ieps$ and using the properties of $\bm \Lambda(\bm x_\eps)$ and $\bm \delta_\eps$, we obtain 
\begin{align*}
\rho(\bm x_\eps) |\bm \Lambda (\bm x_\eps)|^2 \alpha_\eps^2  + \langle \bm \delta(\bm x_\eps), M(\bm x_\eps) \cdot \bm \delta(\bm x_\eps) \rangle  + \mathcal O(\mu_\eps^{2}) \\
= 
 \begin{cases} 
 -d_N \eps \sum_i \mu_\ieps^{-N+4}   (V(x_\ieps) +o(1))  \lambda_\ieps^2 & \text{ if } N \geq 5, \\
 - (8 \pi^2)^{-1} \eps \sum_i \ln \mu_\ieps^{-1} (V(x_\ieps) + o(1))  &\text{ if } N = 4.
\end{cases}
\end{align*}
The crucial information given by Proposition \ref{proposition bound on rho} is now that the terms in $\rho(\bm x_\eps)$ and in $\bm \delta(\bm x_\eps)$ on the left side are negligible. Since $V < 0$ and $\lambda_\ieps \sim 1$, the above identity then implies $\eps \mu_\ieps^{- N +4} = \mathcal O(\mu_\eps^2)$ if $N \geq 5$ resp. $\eps \ln (\mu_{\ieps}^{-1}) = \mathcal O(\mu_\eps^2)$  if $N = 4$, as claimed. 
%
This completes the proof of Theorem \ref{theorem multibubble}.

\begin{proof}
[Proof of Proposition \ref{proposition bound on rho}]
Arguing as in \cite[Lemma 5.5]{Koenig2022}, we get 
\begin{align*}
\partial_l^{x_i} \rho(\bm{x}_\eps) &=  \partial_l^{x_i} \langle \bm{\lambda}_\eps, M_a(\bm{x}) \cdot \bm{\lambda}_\eps \rangle |_{\bm{x} = \bm{x_\eps}}  + \mathcal O(|\rho(\bm{x}_\eps)|  + | \bm{\delta}(\bm{x}_\eps)|) \\
&= \lambda_\ieps (\tilde{M}^l_a(\bm{x}_\eps) \cdot \bm{\lambda}_\eps)_i + + \mathcal O(|\rho(\bm{x}_\eps)|  + | \bm{\delta}(\bm{x}_\eps)|). 
\end{align*}
Inserting the bound from Proposition \ref{proposition expansion gradient}, we thus get, for every $\delta > 0$, 
\begin{equation}
\label{nabla rho  a priori bound}
|\nabla \rho(\bm{x}_\eps)| =
\begin{cases}
 \mathcal O \left( \eps \mu_\eps^{-N+4} + \mu_\eps^{2- \delta} + |\rho(\bm{x}_\eps)|  + | \bm{\delta}(\bm{x}_\eps)|\right) & \text{ if } N \geq 5, \\
\mathcal O\left(\eps  \ln (\mu_\eps^{-1}) + \mu_\eps^{2 -\delta} + |\rho(\bm{x}_\eps)|  + | \bm{\delta}(\bm{x}_\eps)| \right) & \text{ if } N = 4.
 \end{cases} 
\end{equation} 
On the other hand, writing $M(\bm x_\eps) \cdot \bm \lambda_\eps = \alpha_\eps \rho(\bm x_\eps) \bm \Lambda (\bm x_\eps) + M(\bm x_\eps) \cdot \bm \delta (\bm x_\eps)$, \eqref{identity with lambda} implies 
\[ M(\bm x_\eps) \cdot \bm \delta (\bm x_\eps) =
\begin{cases}
 \mathcal O( \eps \mu_\eps^{-N+4} + \mu_\eps^2 \ln (\mu_\eps^{-1}) + |\rho(\bm x_\eps)|) & \text{ if } N \geq 5, \\
 \mathcal O( \eps \ln (\mu_\eps^{-1})  + \mu_\eps^2 \ln (\mu_\eps^{-1}) + |\rho(\bm x_\eps)|) & \text{ if } N = 4. 
\end{cases} 
  \]
Since $\rho(\bm x_\eps)$ is simple, $M(\bm x_\eps)$ is uniformly coercive on the subspace orthogonal to $\bm \Lambda (\bm x_\eps)$, which contains $\bm \delta(\bm x_\eps)$. Hence \eqref{delta bound} follows. 

Moreover, with \eqref{delta bound} we can simplify \eqref{nabla rho  a priori bound}  to 
\begin{equation}
\label{nabla rho bound without delta}
|\nabla \rho(\bm{x}_\eps)| =
\begin{cases}
 \mathcal O \left( \eps \mu_\eps^{-N+4} + \mu_\eps^{2- \delta} + |\rho(\bm{x}_\eps)|  \right) & \text{ if } N \geq 5, \\
\mathcal O\left(\eps  \ln (\mu_\eps^{-1}) + \mu_\eps^{2 -\delta} + |\rho(\bm{x}_\eps)|   \right) & \text{ if } N = 4.
 \end{cases} 
\end{equation}
Now we claim that there is $\sigma > 1$ such that 
\begin{equation}
\label{rho to nabla rho estimate}
\rho(\bm x_\eps) \lesssim |\nabla \rho(\bm x_\eps)|^\sigma.
\end{equation}
If we choose $\delta >0$ so small that $(2-\delta)\sigma > 2$, together with \eqref{nabla rho bound without delta} this yields
\[ \rho(\bm x_\eps) = 
\begin{cases}
o(\eps \mu_\eps^{-N+4} + \mu_\eps^{2}) + \mathcal O(\rho(\bm x_\eps)^\sigma) & \text{ if } N \geq 5, \\
o(\eps \ln (\mu_\eps^{-1}) + \mu_\eps^{2}) + \mathcal O(\rho(\bm x_\eps)^\sigma) & \text{ if } N = 4.
\end{cases} 
\]
Here we used that the assumption $\rho(\bm x_0) = 0$ implies that $\eps \mu_\eps^{-N+4} = o(1)$, resp. $\eps \ln (\mu_\eps^{-1}) = o(1)$, as observed above. Hence $(\eps \mu_\eps^{-N+4})^\sigma = o(\eps \mu_\eps^{-N+4})$ and $(\eps \ln (\mu_\eps^{-1}) )^\sigma = o(\eps \ln (\mu_\eps^{-1})) $. 

In the same way, since $\rho(\bm x_\eps) = o(1)$, we can absorb $ \mathcal O(\rho(\bm x_\eps)^\sigma) = o(\rho(\bm x_\eps))$ into the left side and \eqref{rho bound} follows, as desired. With these informations, we can return to \eqref{nabla rho bound without delta} to deduce the bound on $|\nabla \rho(x_\eps)|$ claimed in Theorem \ref{theorem multibubble}. 

So it remains only to justify \eqref{rho to nabla rho estimate}. This follows by arguing as in \cite[proof of Theorem 2.1]{Koenig2022} once we note that $\rho(\bm x)$ is an analytic function of $\bm x$. Indeed, $\rho(\bm x)$ is a simple eigenvalue of the matrix $M(\bm x)$. Hence it depends analytically on $\bm x$ if the entries of $M(\bm x)$ do so. But this is clearly the case: $G_0(\cdot,y)$ is harmonic, hence analytic on $\Omega \setminus \{y\}$, and $H_0(\cdot, y)$ is harmonic, hence analytic on all of $\Omega$, hence so is $\phi(x) = H(x,x)$. The proof is therefore complete. 
\end{proof}

\appendix

\section{Some computations}
\label{section appendix}
In the following we denote by 
\begin{equation}
\label{v0 v1 definition}
v(r) = \frac{1-r^2}{(1 + r^2)^{N/2}}, \qquad \widetilde{v}(r) = \frac{r}{(1 +r^2)^{N/2}}. 
\end{equation} 
These functions are chosen such that 
\[ \partial_\lambda B(x) = c v(|x|), \qquad \partial_{x_i} B(x) = \widetilde{c}\,  \widetilde{v}(|x|) \frac{x_i}{|x|}, \]
for some constants $c$, $\widetilde{c}$. In particular, $v$ and $\widetilde{v}$ are solutions to the homogeneous ODEs 
\begin{equation}
\label{ODE ell=0 hom}
-v'' - \frac{N-1}{r} v' - N(N+2) B^\frac{4}{N-2}v = 0 \quad \text{ on } (0, \infty) 
\end{equation} 
and 
\begin{equation}
\label{ODE ell=1 hom}
-\widetilde{v}'' - \frac{N-1}{r} \widetilde{v}' + \frac{N-1}{r^2} \widetilde{v} - N(N+2) B^\frac{4}{N-2}\widetilde{v} = 0 \quad \text{ on } (0, \infty),
\end{equation} 
respectively. In the following lemmas, we study the asymptotic properties at $0$ and $\infty$ of solutions to inhomogeneous versions of \eqref{ODE ell=0 hom} and \eqref{ODE ell=1 hom}. These will be crucial for handling the refined terms $q_\ieps$ and $p_\ieps$ in the main part of our argument. 

\begin{lemma}
\label{lemma W}
For $v$ as in \eqref{v0 v1 definition}, let $W$ be given by
\begin{equation}
\label{W definition}
W(r) = v(r) \int_0^r \frac{1}{s^{N-1} v(s)^2} \left( \int_0^s B(t) t^{N-1} v(t) \diff t \right) \diff s.   
\end{equation}
Then $W$ solves 
\[ -W'' - \frac{N-1}{r} W' - N(N+2) B^{\frac{4}{N-2}} W = - B \quad \text{ on } (0, \infty). 
\]
Moreover, 
\[
\frac{W(r)}{r^2} \to  \frac{1}{2N}  \qquad \text{ as } r \to 0. 
\]
As $r \to \infty$, 
\begin{equation}
\label{W asymptotics lemma}
 W(r) = 
 \begin{cases}
  \frac{\Gamma(\frac{N}{2})\Gamma(\frac{N-4}{2})}{(N-2) \Gamma(N-1)} + o(1)  & \text{ if } N \geq 5, \\
  \frac{1}{2} \ln r + o(\ln r) & \text{ if } N = 4,
  \end{cases}
\end{equation}
and 
\[W'(r) = \begin{cases}
 o(r^{-1}) & \text{ if } N \geq 5, \\
 o(r^{-1} \ln r) & \text{ if } N = 4.
 \end{cases} \]
\end{lemma}

\begin{proof}
It can be checked by a direct computation that $W$ satisfies the claimed equation. Let us show how the expression \eqref{W definition} can be derived using the method of variation of constants. 
We write $W = v \varphi$ for a new unknown function $\varphi$.  
Then $\psi := \varphi'$ needs to solve 
\[ \psi'(r) + (\frac{N-1}{r} + \frac{2 v'}{v}) \psi = \frac{B}{v}. \]
Again by the variation of constants, we may write $\psi = \eta \psi_0$, with 
\[ \psi_0(r) := \exp \left( - \int_1^r (\frac{N-1}{s}  +\frac{2 v'}{v} )\diff s \right)= \frac{1}{r^{N-1} v^2}. \]
Since $\psi_0'(r) + (\frac{N-1}{r} + \frac{2 v'}{v}) \psi_0 = 0$, it remains to solve
\[ \eta' = \frac{B}{v\psi_0} = B v r^{N-1}, \]
which gives
\[ \eta(r) = \int_0^r B s^{N-1} v \diff s = \int_0^r \frac{s^{N-1} (1 - s^2)}{(1+s^2)^{N-1}} \diff s. \]
If $N \geq 5$, this integral remains finite as $r \to \infty$ and we find, using the integral representation of the Beta function, 
\[ \lim_{r \to \infty} \eta(r) = - \frac{\Gamma(\frac{N}{2}) \Gamma(\frac{N-4}{2}) }{\Gamma (N-1)}. \]
On the other hand, if $N = 4$, the integral diverges and we have 
\[ \eta(r) = \left(- 1 +o(1)\right) \ln r \qquad \text{ as } r \to \infty. \]
Using $v(r) \sim - r^{-N+2}$, we moreover find 
\[ \psi_0(r) \sim  r^{N-3} \]
and hence 
\[ \psi(r) = \eta(r) \psi_0(r) \sim \begin{cases} -  \frac{\Gamma(\frac{N}{2})\Gamma(\frac{N-4}{2})}{\Gamma(N-1)} r^{N-3} &\text{ if } N \geq 5, \\
-   r \ln r & \text{ if } N = 4,
\end{cases} \]
respectively
\[ \varphi(r) = \int_0^r \psi(s) \diff s \sim 
\begin{cases}
 - \frac{1}{(N-2)} \frac{\Gamma(\frac{N}{2})\Gamma(\frac{N-4}{2} )}{\Gamma(N-1)} r^{N-2} & \text{ if } N \geq 5, \\
- \frac{1}{2}r^2 \ln r & \text{ if } N = 4.
 \end{cases}
  \]
By recalling $W = v \varphi$ the claimed asymptotic behavior of $W$ follows. 

Similarly, using $v'(r) \sim (N-2)^2 r^{-N+1}$ and the above asymptotics for $\varphi$ and $\psi$, we get 
\[ W'(r) = v'(r) \varphi(r) + v(r) \psi(r) = o(r^{-1}), \]
because the terms of size $r^{-1}$ cancel precisely, and similarly for $N = 4$. 
The claimed asymptotic behavior as $r \to 0$ can be read off directly from \eqref{W definition}, using that $B(r) \to 1$ and $v(r) \to 1$ as $r \to 0$. 
\end{proof}

A very similar argument, whose details we omit, yields the asymptotics of the ODE solution governing the correction term $\widetilde{W}_\ieps$. 

\begin{lemma}
\label{lemma U}
For $\widetilde{v}$ as in \eqref{v0 v1 definition}, let $\widetilde{W}$ be given by
\begin{equation}
\label{Wt definition}
\widetilde{W}(r) = \widetilde{v}(r) \int_0^r \frac{1}{s^{N-1} \widetilde{v}(s)^2} \left( \int_0^s B(t) t^{N} \widetilde{v}(t) \diff t \right) \diff s.   
\end{equation}
Then $\widetilde{W}$ solves 
\[ -\widetilde{W}'' - \frac{N-1}{r} \widetilde{W}' + \frac{N-1}{r^2} \widetilde{W}  - N(N+2) B^{\frac{4}{N-2}} \widetilde{W}= - B(r) r \quad \text{ on } (0, \infty). \]
Moreover, 
\[
\frac{\widetilde{W}(r)}{r^3} \to  \frac{1}{2(N+2)}  \qquad \text{ as } r \to 0
\]
and, when $N \geq 5$,
 \[ \lim_{r \to\infty} \widetilde{W}(r) r^{-1} = \lim_{r \to \infty} \widetilde{W}'(r) =  \frac{a_N}{N}, \]
with 
\[ a_N = \frac{N}{4} \frac{\Gamma(\frac{N}{2}) \Gamma(\frac{N-4}{2})}{\Gamma(N-1)}. \]
\end{lemma}


Using the inhomogeneous solutions $W$ and $\widetilde{W}$, we can now construct the functions $W_{jk}$ and $W_{jkl}$ from \eqref{Wjk equation} and \eqref{Wjkl equation}. 

\begin{lemma}
\label{lemma Wjk}
For multiindices $\alpha \in \N_0^N$, we consider functions $W_{\alpha}$ which satisfy
\begin{equation}
\label{Wjk equation lemma}
-\Delta W_{\alpha} - N(N+2) B^\frac{4}{N-2} W_{\alpha} = f_\alpha \quad \text{ on } \R^N, \quad W_{\alpha}(x) = \frac{1}{\alpha!} x^\alpha + o(|x|^2), 
\end{equation}
with 
\begin{equation}
\label{fjk lemma}
f_\alpha = \begin{cases}
0 & \text{ if } x^\alpha = x_j x_k \, \text{ for some } j \neq k, \\
-B & \text{ if } x^\alpha = x_j^2 \, \text{ for some } j
\end{cases}
\end{equation} 
\end{lemma}

\begin{proof}
If $x^\alpha = x_j x_k$ with $j \neq k$, we make the ansatz $W_\alpha = f(|x|) Y_{jk}(x/|x|)$, with $Y_{jk}(\omega) =  \omega_j \omega_k$ for $\omega \in \mathbb S^{N-1}$. Observing that $Y_{jk}$ is a spherical harmonic of degree $2$, $W_{\alpha}$ solves the equation in \eqref{Wjk equation lemma} if and only if $f$ solves the ODE 
\begin{equation}
\label{ode proof}
- f''(r) - \frac{N-1}{r} f'(r) + \frac{2N}{r^2} f(r)  + N(N+2) f(r) B^\frac{4}{N-2}(r) = 0 \quad \text{ on } (0, \infty). 
\end{equation} 
By the proof of \cite[Proposition A.1]{KL}, there is a solution to \eqref{ode proof} which satisfies $f(r) \sim r^2$ for $r \in (0, \infty)$. Up to replacing $f$ by a suitable scalar multiple, we may thus assume that $\ds \lim_{r \to 0} f(r) r^{-2} = 1$. It follows that 
\[ W_{\alpha}(x) = f(|x|) Y_{jk}(x /|x|) =  f(|x|) |x|^{-2} x_j x_k = (1 + o(1)) x_j x_k = \frac{1}{\alpha!} x^\alpha + o(|x|^2). \]
If on the other hand $x^\alpha = x_j^2$, we set
\[ W_{\alpha}(x) = f(|x|)Y_j(x/|x|) + W(|x|), \]
where $Y_j(\omega) = \frac{1}{2} \omega_j^2 - \frac{1}{2N}$, $f$ is a solution to \eqref{ode proof} with $\ds \lim_{r \to 0} f(r) r^{-2} = 1$ and $W$ is the function from Lemma \ref{lemma W}. Observing that $Y_j$ is a spherical harmonic of degree $2$, $W_{jj}$ satisfies the equation in \eqref{Wjk equation lemma}. Moreover, 
\[ W_{jj}(x) = f(|x|) \left(\frac{1}{2}\frac{x_j^2}{|x|^2} - \frac{1}{2N} \right) + W(|x|) =\frac{1}{\alpha!} x^\alpha + \left( W(|x|) - \frac{1}{2N} f(|x|) \right) + o(|x|^2) . \]
By Lemma \ref{lemma W}, we have $ W(|x|) - \frac{1}{2N} f(|x|) = W(|x|) - \frac{1}{2N}|x|^2 = o(|x|^2)$, and the proof is complete.
\end{proof}

\begin{lemma}
\label{lemma Wjkl}
For every multiindex $\alpha \in \N_0^N$ with $|\alpha| = 3$, there are functions $\widetilde{W}_\alpha$ satisfying
 \begin{equation}
\label{Wjkl equation lemma}
-\Delta \widetilde{W}_\alpha - N(N+2) B^\frac{4}{N-2} \widetilde{W}_\alpha = f_\alpha \quad \text{ on } \R^N, \quad \widetilde{W}_\alpha(x) = \frac{1}{\alpha!} x^\alpha + o(|x|^3) \quad \text{ as } x \to 0,
\end{equation}
with
\begin{equation}
\label{fjkl equation lemma}
f_\alpha = \begin{cases}
0 & \text{ if } x^\alpha = x_j x_k x_l \, \text{ for some } j \neq k \neq l \neq j, \\
-B x_l & \text{ if } x^\alpha = x_j^2 x_l \text{ for some } j, l. 
\end{cases}
\end{equation}
\end{lemma}

The proof is similar to the previous one. However, in the case $x^\alpha = x_j^2 x_l$, the needed decomposition is a bit more subtle, so let us give full details also here. 

\begin{proof}
If $x^\alpha = x_j x_k x_l$ with $j \neq k \neq l \neq j$, we make the ansatz $\widetilde{W}_\alpha = f(|x|) Y_{jkl}(x/|x|)$, with $Y_{jkl}(\omega) =  \omega_j \omega_k \omega_l$ for $\omega \in \mathbb S^{N-1}$. Observing that $Y_{jkl}$ is a spherical harmonic of degree $3$, $W_{\alpha}$ solves the equation in \eqref{Wjk equation lemma} if and only if $f$ solves the ODE 
\begin{equation}
\label{ode proof 2}
- f''(r) - \frac{N-1}{r} f'(r) + \frac{2N}{r^2} f(r)  + N(N+2) f(r) B^\frac{4}{N-2}(r) = 0 \quad \text{ on } (0, \infty). 
\end{equation} 
By the proof of \cite[Proposition A.1]{KL}, there is a solution to \eqref{ode proof} which satisfies $f(r) \sim r^3$ for $r \in (0, \infty)$. Up to replacing $f$ by a suitable scalar multiple, we may thus assume that $\ds \lim_{r \to 0} f(r) r^{-3} = 1$. It follows that 
\[ \widetilde{W}_{\alpha}(x) = f(|x|) Y_{jkl}(x /|x|) =  f(|x|) |x|^{-3} x_j x_k = (1 + o(1)) x_j x_k x_l = \frac{1}{\alpha!} x^\alpha + o(|x|^3). \]

If on the other hand $x^\alpha = x_j^2 x_l$, first assume that $j \neq l$. Observing that $x_j^2 x_l - x_i^2 x_l$ is a homogeneous harmonic polynomial of degree $3$ for every $i \notin \{j, l\}$, so is 
\[ \sum_{i\notin \{j,l\}} (x_j^2 x_l - x_i^2 x_l) \pm x_j^2 x_l \pm x_l^3 = (N-1) x_j^2 x_l - |x|^2 x_l + x_l^3.  \]
Similarly, $x_l^3 - 3 x_j^2 x_l$ for every $j \neq l$ is  a homogeneous harmonic polynomial of degree $3$, hence so is 
\[ \sum_{j \neq l} x_l^3 - 3 x_j^2 x_l \pm 3 x_l^3 = (N+2) x_l^3 - 3 |x|^2 x_l. \]
Subtracting appropriate scalar multiples of the found expressions from each other, we obtain that 
\[ \widetilde{Y}_{jl}(\omega) = \frac{1}{2} \omega_j^2 \omega_l - \frac{1}{2(N+2)} \omega_l \]
is a spherical harmonic of degree $3$. 

We now set 
\[ \widetilde{W}_{\alpha}(x) = f(|x|)\widetilde{Y}_{jl}(x/|x|) + \widetilde{W}(|x|) \frac{x_l}{|x|}, \]
$f$ is a solution to \eqref{ode proof 2} with $\ds \lim_{r \to 0} f(r) r^{-3} = 1$ and $\widetilde{W}$ is the function from Lemma \ref{lemma U}. Then $ \widetilde{W}_{\alpha}$ satisfies the PDE in \eqref{Wjkl equation lemma}. Moreover, 
\begin{align*}
\widetilde{W}_{\alpha}(x) &= f(|x|) \left(\frac{1}{2}\frac{x_j^2 x_l}{|x|^3} - \frac{1}{2(N+2)} \frac{x_l}{|x|} \right) + \widetilde{W}(|x|) \frac{x_l}{|x|} \\
& =   
\frac{f(|x|)}{|x|^3} \left( \frac{1}{2 x_j^2 x_l} - \frac{1}{2(N+2)} x_l |x|^2 \right) + \frac{\widetilde{W}(|x|)}{|x|^3} x_l |x|^2 \\
&= \frac{1}{2} x_j^2 x_l + o(|x|^3) = \frac{1}{\alpha!} x^\alpha + o(|x|^3).  
\end{align*} 
Here we crucially used that $\frac{\widetilde{W}(r)}{r^3} \to \frac{1}{2(N+2)}$ as $r \to 0$ by Lemma \ref{lemma U}. 

Finally, if $j = l$, that is, if $x^\alpha = x_j^3$, we use instead of $\widetilde{Y}_{jl}$ the function 
\[ \widetilde{Y}_j(\omega) = \frac{1}{6} \omega_j^3 - \frac{1}{2(N+2)} w_\alpha. \]
The rest of the argument is identical. 
\end{proof}


\section{Classical Asymptotic analysis}
\label{appendixB}
In this section we generalize the result of \cite{Druet2010} to $N\geq 3$ under appropriate assumptions. The proof is globally the same except at the level of Claim \ref{estim4} where some refined analysis is needed when $N\geq 4$. As already mentioned in \cite{Druet2010}, the proof follows \cite{DruetHebey09}.

\begin{proposition}
\label{proposition preliminaries multi bis}
Consider a sequence $\left(u_\eps\right)$ of $C^2$ solutions to 
\begin{equation}\label{eq1.1}
\left\{\begin{array}{ll}
{\ds -\Delta u_\eps + h_\eps u_\eps =N(N-2)u_\eps ^\frac{N+2}{N-2}}&{\ds \hbox{ in } \Omega}\\
\, &\,\\
{\ds u_\eps =0 }&{\ds \hbox{ on } \partial\Omega}\\
\, &\, \\
{\ds \ue >0}&{\ds \hbox{ in } \Omega}
\end{array}\right.
\end{equation}
where $\Omega$ is some smooth domain of $\R^N$ and
\begin{equation}\label{eq1.2}
h_\eps \rightarrow h_0 \hbox{ in } C^{0,\eta}(\Omega) \hbox{ as } \eps \rightarrow 0
\end{equation}
if $N=3$, or
$$h_\eps= \eps V$$
where $V\in C^{1}(\Omega) \cup C(\overline{\Omega})$ with $V<0$ on $\overline{\Omega}$ if $N\geq 4$.
 
Then either $\Vert u_\eps\Vert_\infty $ is bounded or, up to extracting a subsequence, there exists $n \in \N$ and points $x_{1,\eps},..., x_{n,\eps}$ such that the following holds. 
\begin{enumerate}[(i)]
\item $x_\ieps \to x_i \in \Omega$ for some $x_i \in \Omega$ with $x_i \neq x_j$ for $i \neq j$. 
\item $\mu_\ieps := u_\eps(x_\ieps)^{-\frac{2}{N-2}} \to 0$ as $\eps \to 0$ and $\nabla u_\eps(x_\ieps) = 0$ for every $i$. 
\item \label{item lambda i} $\lambda_{i,0} := \ds \lim_{\eps \to 0} \lambda_\ieps := \ds \lim_{\eps \to 0} \frac{\mu_\ieps^\frac{N-2}{2}}{\mu_\oeps^\frac{N-2}{2}}$ exists and lies in $(0, \infty)$ for every $i$. 
\item \label{item u to bubble} $\mu_{i,\eps}^{\frac{N-2}{2}} u_\eps(x_{i,\eps} + \mu_{i,\eps} x) \to B$ in $C^1_\text{loc}(\R^n)$. 
\item \label{item tilde G} There are $\nu_i > 0$ such that $\mu_{1,\eps}^{-\frac{N-2}{2}} u_{\eps} \to  \sum_i \nu_i G(x_{i,\eps}, \cdot) =: \widetilde{\mathcal G}$ uniformly in $C^1$ away from $\{x_1,...,x_n\}$, where $G$ is the Green function of $-\Delta +h_0$.
\item There is $C > 0$ such that $u_\eps \leq C \sum_i B_\ieps$ on $\Omega$.  
\end{enumerate}
\end{proposition}

The proof is divided into many steps. The first one consists in transforming a weak estimate such as \eqref{eq1.5} into a strong one such as \eqref{strestim}  around a concentration point, that is to say that at a certain scale $u_\eps$ behaves like a bubble. So we consider a sequence $u_\eps$ which satisfies the hypotheses of Proposition \ref{proposition preliminaries multi bis} and we also assume that we have a sequence $\left(x_\eps\right)$ of points in $\Omega$ and a sequence $\left(\rho_\eps\right)$ of positive real numbers with $0< 3 \rho_\eps \leq d(x_\eps, \partial\Omega)$  such that 
\beq\label{eq1.3}
\nabla \ue( x_\eps) =0
\eeq
and
\beq\label{eq1.4}
 \rho_\eps \left[ \sup_{B(x_\eps, \rho_\eps)}  u_\eps(x) \right]^\frac{2}{N-2}  \rightarrow + \infty \hbox{ as }\eps \rightarrow 0 \hskip.1cm.
\eeq
First, we prove that, under this extra assumption, the following holds~:
\begin{proposition}\label{estim}
If there exists $C_0 >0$ such that
\beq\label{eq1.5}
\vert x_\eps - x \vert^\frac{N-2}{2} u_\eps\leq C_0 \hbox{ in } B(x_\eps, 3\rho_\eps)\hskip.1cm,
\eeq
then there exists $C_1 >0$ such that
\beq
\label{strestim}
\begin{split}
&u_\eps(x_\eps) u_\eps(x) \leq C_1 \vert x_\eps - x \vert^{2-N}\hbox{ in } B(x_\eps, 2\rho_\eps)\setminus \{x_
\eps \}\hbox{ and }\\
&u_\eps(x_\eps) \vert\nabla u_\eps(x)\vert \leq C_1 \vert x_\eps - x \vert^{1-N}\hbox{ in } B(x_\eps, 2\rho_\eps)\setminus \{x_
\eps \}.
\end{split}
\eeq
Moreover, if $\rho_\eps \rightarrow 0$, then 
\be 
\rho_\eps^{N-2} u_\eps(x_\eps) u_\eps(x_\eps+\rho_\eps x) \rightarrow \frac{1}{\vert x\vert^{N-2}}+ b  \hbox{ in } C^1_{loc}(B(0, 2)\setminus \{0\})\hbox{ as }\eps\to 0
\ee
where $b$ is some harmonic function in $B(0,2)$ with $b(0)\leq 0$ and $\nabla b(0)=0$.
\end{proposition}

\subsection{Proof of Proposition \ref{estim}}
We divide the proof of the proposition into several claims. The first one gives the asymptotic behaviour of $u_\eps$ around $x_\eps$ at an appropriate small scale. 

\begin{claim}
\label{estim1}
After passing to a subsequence, we have that
\beq\label{eq1.7}
\mu_\eps^\frac{N-2}{2}  u_\eps(x_\eps + \mu_\eps x) \rightarrow B
\hbox{ in } C^1_{loc}(\R^3), \hbox{ as } \eps \rightarrow 0, 
\eeq
where $\mu_\eps = u_\eps\left(x_\eps\right)^{\frac{2}{2-N}}$.
\end{claim}

\medskip {\bf Proof of Claim \ref{estim1}.} Let $\tilde{x}_\eps \in \overline{B(x_\eps, \rho_\eps)}$ and $\tilde{\mu}_\eps>0$ be such that
\beq\label{eq1.8}
u_\eps(\tilde{x}_\eps) = \sup_{B(x_\eps, \rho_\eps)} u_\eps = \tilde{\mu}_\eps^{\frac{2-N}{2}} \hskip.1cm.
\eeq
Thanks to (\ref{eq1.4}), we have that 
\beq\label{eq1.9}
\tilde{\mu}_\eps\to 0 \hbox{ and }\frac{\rho_\eps}{\tilde{\mu}_\eps}\to +\infty\hbox{ as }\eps\to 0\hskip.1cm.
\eeq
Thanks to  (\ref{eq1.5}), we also have that
\beq
\label{eq1.10}
\vert x_\eps-\tilde{x}_\eps\vert= O(\tilde{\mu}_\eps).
\eeq
We set for $\displaystyle x \in \Omega_\eps =\left\{ x\in\R^N \hbox{ s.t. } \tilde{x}_\eps + \tilde{\mu}_\eps x \in \Omega \right\}$, 
\be
\tue (x)= \tilde{\mu}_\eps^{\frac{N-2}{2}}  u_\eps(\tilde{x}_\eps + \tilde{\mu}_\eps x) 
\ee
which verifies
\beq
\label{eq1.11}
\begin{split}
&-\Delta \tue +\tilde{\mu}_\eps^2 \tilde{h}_\eps \tue =N(N-2)\tue ^\frac{N+2}{N-2} \hbox{ in } \Omega_\eps\hskip.1cm, \\
&\tue(0) = \sup_{ B(\frac{x_\eps -\tilde{x}_\eps}{\tilde{\mu}_\eps}, \frac{\rho_\eps}{\tilde{\mu}_\eps})} \tue =1\hskip.1cm,
\end{split}
\eeq
where $\tilde{h}_\eps= h\left(\tilde{x}_\eps + \tilde{\mu}_\eps x\right)$. Thanks to  (\ref{eq1.9}) and (\ref{eq1.10}), we get that
\beq
\label{eq1.12}
B\left(\frac{x_\eps -\tilde{x}_\eps}{\tilde{\mu}_\eps}, \frac{\rho_\eps}{\tilde{\mu}_\eps}\right) \rightarrow \R^N \hbox{ as } \eps \rightarrow 0 \hskip.1cm.
\eeq
Now, thanks to (\ref{eq1.11}), (\ref{eq1.12}), and by standard elliptic theory, we get that, after passing to a subsequence, $\tue\rightarrow B$ in $C^1_{loc}(\R^N)$ as $\eps \rightarrow 0$, where $B$ satisfies
\be 
-\Delta B= N(N-2) B^\frac{N+2}{N-2} \hbox{ in } \R^N \hbox{ and } 0\leq B \leq1=U(0) \hskip.1cm.
\ee
Thanks to the work of Caffarelli, Gidas and Spruck \cite{CGS}, we know that
\be
B(x) = \left(1+ \vert x \vert^{2}  \right)^{-\frac{N-2}{2}} \hskip.1cm.
\ee
Moreover, thanks to (\ref{eq1.10}), we know that, after passing to a new subsequence, $\frac{x_\eps -\tilde{x}_\eps}{\tilde{\mu}_\eps} \rightarrow x_0 \hbox{ as } \eps \rightarrow 0$ for some $x_0 \in \R^N$. Hence, since $x_\eps$ is a critical point of $u_\eps$, $x_0$ must be a critical point  of $U$, namely  $x_0=0$. We deduce that $\frac{\mu_\eps}{\tilde{\mu}_\eps} \rightarrow 1$ where $\mu_\eps$ is as in the statement of the claim. Claim \ref{estim1} follows. \hfill $\square$

\medskip For $0\leq r \leq 3 \rho_\eps$, we set
$$\psi_\eps(r)= \frac{r^{\frac{N-2}{2}} }{\omega_{N-1} r^{N-1}} \int_{\partial B({x_\eps},r)} u_\eps \,  d\sigma, $$
where $d\sigma$ denotes the Lebesgue measure on the sphere $\partial B({x_\eps},r)$ and $\omega_{N-1}$ is the volume of the unit $(N-1)$-sphere. 
We easily check, thanks to Claim \ref{estim1}, that
\beq\label{eq1.13}
\psi_\eps(\mu_\eps r)= \left( \frac{r}{1+ r^2}\right)^\frac{N-2}{2} + o(1), \quad \psi_\eps ' (\mu_\eps r)= \frac{N-2}{2} \left( \frac{r}{1+ r^2}\right)^\frac{N}{2} \left(\frac{1}{r^2} - 1 \right) + o(1)\hskip.1cm.
\eeq
We define $r_\eps$ by
\be
r_\eps = \max \left\{ r\in[ 2\mu_\eps, \rho_\eps] \hbox{ s.t. }   \psi_\eps ' (s) \leq  0 \hbox{ for } s \in [2 \mu_\eps, r] \right\} \hskip.1cm.
\ee
Thanks to (\ref{eq1.13}), the set on which the maximum is taken is not empty for $\eps$ small enough, and moreover 
\beq\label{eq1.14}
\frac{r_\eps}{\mu_\eps} \rightarrow + \infty   \hbox{ as } \eps \rightarrow 0\hskip.1cm.
\eeq
We now prove the following: 

\begin{claim} \label{estim2} 
There exists $C >0$, independent of $\eps$, such that
\be
\begin{split}
&u_\eps(x) \leq C \mu_\eps^\frac{N-2}{2} \vert x_\eps - x \vert^{2-N}\hbox{ in } B(x_\eps, 2r_\eps)\setminus \{x_
\eps \}\hbox{ and } \\
&\vert \nabla u_\eps(x) \vert \leq C  \mu_\eps^\frac{N-2}{2} \vert x_\eps - x \vert^{1-N}\hbox{ in } B(x_\eps, 2r_\eps)\setminus \{x_
\eps \}\hskip.1cm.
\end{split}
\ee
\end{claim}

\medskip {\bf Proof of Claim \ref{estim2}.}  We first prove that for any given $0< \nu <\frac{1}{2}$, there exists $C_\nu >0 $ such that 
\beq\label{eq1.15}
\ue(x) \leq C_\nu \left( \mu_{\eps}^{\frac{N-2}{2}(1-2\nu)} \vert x-x_\eps \vert^{(2-N)(1-\nu)} + \alpha_\eps 
\left(\frac{r_{\eps}}{\vert x-x_\eps \vert}\right)^{(N-2)\nu} \right)
\eeq
for all $x\in B\left(x_{\eps},2r_{\eps}\right)$ and $\eps$ small enough, where
\beq\label{eq1.15bis}
\alpha_\eps=  \sup_{ \partial B\left(x_\eps, r_\eps\right)} \ue \hskip.1cm.
\end{equation}
First of all, we can use (\ref{eq1.5}) and apply the Harnack inequality, see Lemma \ref{harn}, to get the existence of some $C>0$ such that 
\beq\label{eq1.16}
\frac{1}{C} \max_{\partial B(x_\eps, r)} \left(\ue+r\left\vert \nabla u_\eps\right\vert\right) \leq \frac{1}{\omega_{N-1} r^{N-1}} \int_{\partial B({x_\eps},r)} u_\eps d\sigma \leq C \min_{\partial B(x_\eps, r)}\ue  
\eeq
for all  $0  <r < \frac{5}{2} \rho_\eps $ and all $ \eps>0$. Hence, thanks to (\ref{eq1.13}) and (\ref{eq1.14}), we have that 
$$\vert x-x_\eps\vert^\frac{N-2}{2} \ue (x) \leq C \psi_\eps(r) 
\leq C \psi_\eps(R \mu_{\eps})=C \left( \frac{R}{1+ R^2}\right)^\frac{N-2}{2} + o(1)$$
for all $R\geq  2$, all $r\in[ R \mu_{\eps},r_\eps ]$, all $\eps$ small enough and all $x\in \partial B\left(x_{\eps},r\right)$. Thus we get that 
\beq\label{eq1.17}
\sup_{B(x_\eps, r_\eps)\setminus B(x_\eps, R \mu_\eps)}  \vert x-x_\eps\vert^\frac{N-2}{2} \ue (x) = e(R)+o(1)\, ,
\eeq
where $e(R)\rightarrow 0$ as $R \rightarrow +\infty$. Let  Let $\mathcal G(x,y)=\frac{1}{(N-2)\omega_{N-1}} \frac{1}{\vert x-y\vert^{N-2}}$, in particular
$$-\Delta \mathcal G( \, ,y)=\delta_y \text{ on }\R^N .$$ 
We fix $0<\nu < \frac{1}{2}$ and we set
 \be 
 \Phi_{\eps,\nu} = \mu_\eps^{\frac{N-2}{2}(1-2\nu)} {\mathcal G}(x_\eps,x)^{1-\nu}+\alpha_\eps \bigl(r_\eps^{N-2} {\mathcal G}(x_\eps,x)\bigr)^{\nu} .
 \ee
Then (\ref{eq1.15}) reduces to proving that
\be
\sup_{B(x_\eps, 2 r_\eps)} \frac{\ue}{ \Phi_{\eps,\nu}} = \mathcal O(1)\hskip.1cm.
\ee
We let $y_\eps \in \overline{B(x_\eps, 2 r_\eps)\setminus\{x_{\eps}\}} $ be such that  
\be
\sup_{B(x_\eps, 2 r_\eps)} \frac{\ue}{ \Phi_{\eps,\nu}} =  \frac{\ue(y_\eps)}{ \Phi_{\eps,\nu}(y_\eps)}.
\ee
We are going to consider the various possible behaviors of the sequence $\left(y_\eps\right)$.

First of all, assume that there is $R < \infty$ such that 
\be
\frac{\vert x_{\eps} -y_\eps\vert}{\mu_\eps} \rightarrow R \hbox{ as } \eps \rightarrow 0\hskip.1cm.
\ee
Thanks to Claim \ref{estim1}, we have in this case that  
\be
\mu_\eps^\frac{N-2}{2}  \ue(y_\eps) \rightarrow (1+R^2)^{-\frac{N-2}{2}} \hbox{ as } \eps \rightarrow 0.
\ee
On the other hand, we can write that 
\begincal
\mu_\eps^\frac{N-2}{2}  \Phi_{\eps,\nu}(y_\eps) &= &\left( \frac{\mu_\eps^{N-2}}{(N-2)\omega_{N-1} \vert x_{\eps} -y_\eps\vert^{N-2}}\right)^{1-\nu} + \mathcal O\left(\alpha_\eps \mu_\eps^\frac{N-2}{2} \left(\frac{r_\eps}{\vert x_{\eps} -y_\eps\vert}\right)^{(N-2)\nu}\right)  \\
 &=&\left( (N-2)R^{N-2}\omega_{N-1}\right)^{\nu-1} + \mathcal O\left((r_\eps^{\frac{N-2}{2}} \alpha_\eps) \mu_\eps^{\frac{N-2}{2}(1-2\nu)} r_\eps^{\frac{1}{2}(2\nu-1)} \right)  \\
 &=&\left( (N-2)R^{N-2}\omega_{N-1}\right)^{\nu-1} + o(1) ,
\fincal
if $R>0$, and $\mu_\eps^\frac{N-2}{2}  \Phi_{\eps,\nu}(y_\eps) \rightarrow +\infty$ as $\eps \rightarrow 0$ if $R=0$. In any case, $\left(\frac{\ue(y_\eps)}{ \Phi_{\eps,\nu(y_\eps)}}\right)$ is bounded. 

Assume now that there exists $\delta>0$ such that $y_\eps \in B(x_\eps ,r_\eps)\setminus B(x_\eps ,\delta r_\eps)$. Thanks to Harnack's inequality (\ref{eq1.16}), we get that $\ue(y_\eps) =\mathcal O(\alpha_\eps)$ which easily gives that $\frac{\ue(y_\eps)}{ \Phi_{\eps,\nu(y_\eps)}} = \mathcal O (1)$.

Hence, we are left with the following situation:
\beq\label{eq1.20}
 \frac{\vert x_\eps -y_\eps\vert}{r_\eps} \rightarrow 0 \, \hbox{ and } \,  \frac{\vert x_{\eps} -y_\eps\vert}{\mu_\eps} \rightarrow +\infty \, \hbox{ as } \, \eps \rightarrow 0\hskip.1cm.
\eeq
Thanks to the definition of $y_\eps$, we can then write that
\be
\frac{-\Delta u_\eps(y_\eps)}{u_\eps(y_\eps)} \geq \frac{-\Delta \Phi_{\eps, \nu}(y_\eps)}{\Phi_{\eps, \nu} (y_\eps)}. 
\ee
Thanks to the definition of $\Phi_{\eps, \nu}$ and multiplying by $\vert x_\eps -y_\eps \vert^2$, this gives 
\begin{align*}
&\vert x_\eps -y_\eps \vert^{2}(-h_\eps(y_\eps) +N(N-2) u_\eps(y_\eps)^\frac{4}{N-2}) \geq\\ & \nu(1-\nu)\frac{\vert x_\eps -y_\eps \vert^{2} }{\Phi_{\eps, \eta} (y_\eps)} \left(  \alpha_\eps r_\eps^{(N-2)\nu} \frac{\vert \nabla \mathcal{G}(x_\eps,y_\eps)\vert^2}{\mathcal{G}(x_\eps,y_\eps)^2} \mathcal{G}(x_\eps,y_\eps)^{\nu} \right. \\
&  +\left. \mu_\eps^{\frac{N-2}{2}(1-2\nu)} \frac{\vert \nabla \mathcal{G}(x_\eps,y_\eps)\vert^2}{\mathcal{G}(x_\eps,y_\eps)^2} \mathcal{G}(x_\eps,y_\eps)^{1-\nu}\right)\hskip.1cm .
\end{align*}

Thanks to (\ref{eq1.17}), the left-hand side goes to $0$ as $\eps \rightarrow 0$. Then, thanks to  (\ref{eq1.20}), we get that 
\be
o(1)\geq (N-2)^2\nu(1-\nu) + o(1)
\ee
which is a contradiction, and shows that this last case can not occur. This ends the proof of (\ref{eq1.15}).

\medskip We now claim that there exists $C >0$, independent of $\eps$, such that 
\beq
\label{eq1.21}
\ue(x) \leq C \left( \mu_{\eps}^{\frac{N-2}{2}} \vert x-x_\eps \vert^{2-N} +  \alpha_\eps \right) \hbox{ in }  B(x_\eps, r_\eps)\hskip.1cm.
\eeq
Thanks to Claim \ref{estim1} and (\ref{eq1.16}), this holds for all sequences $y_\eps \in B(x_\eps, r_\eps)\setminus\{ x_\eps \} $ such that $\vert y_\eps -x_\eps\vert=\mathcal O(\mu_\eps)$ or $\frac{\vert y_\eps -x_\eps\vert}{r_\eps} \not\rightarrow 0$. Thus we may assume from now that
\be
\frac{\vert y_\eps -x_\eps\vert}{\mu_\eps} \rightarrow +\infty \hbox{ and }\frac{\vert y_\eps -x_\eps\vert}{r_\eps} \rightarrow 0 \hbox{ as } \eps \rightarrow 0 \hskip.1cm.
\ee

Let us consider $\mathcal{G}_\epsilon$ the Green function of the operator $-\Delta +h_\eps$. This function exists since, by Appendix \ref{Appcoer}, the operator is coercive, moreover it follows the following classical estimate, see \cite{aubin} or the nice notes \cite{robert},
\begin{equation}
\label{green}
\sup_{x\not= y} \vert x-y \vert^{n-2} \vert\mathcal{G}_\eps (x,y) \vert +  \vert x-y\vert^{n-1} \vert \nabla \mathcal{G}_\eps (x,y) \vert =O(1).
\end{equation}

Thanks to the Green representation formula, we have 
\begincal 
\ue(y_\eps)&=& \int_{B(x_\eps, r_\eps)} \mathcal{G}_\eps(y_\eps, \,.\,) (-\Delta u_\eps +h_\eps u_\eps) \,\diff x \\ 
&+& \mathcal O\left(r_\eps^{-(N-2)}\int_{\partial B(x_\eps, r_\eps)} \left\vert {\partial_\nu} \ue\right\vert \,d\sigma + r_\eps^{-(N-1)}\int_{\partial B(x_\eps, r_\eps)} u_\eps \,d\sigma\right)\, . 
\fincal
This gives with (\ref{eq1.15bis}), (\ref{eq1.16}) and \eqref{green} that 
\beq
\label{eq1.22}
\ue(y_\eps)= \mathcal O\left(\int_{B(x_\eps, r_\eps)} \vert x-y_\eps \vert^{-(N-2)} \ue^\frac{N+2}{N-2} \diff x\right)   + \mathcal O \left( \alpha_\eps\right) \hskip.1cm.
\eeq

Using (\ref{eq1.15}) with $\nu=\frac{1}{N+2}$, and $1<p<\frac{N}{N-2}$ we can write that 
\be
\begin{split}
&\int_{B(x_\eps, r_\eps)} \vert x-y_\eps \vert^{2-N} \ue^\frac{N+2}{N-2} \diff x \\
&\quad=\int_{B(x_\eps, \mu_\eps)} \frac{\ue^\frac{N+2}{N-2}}{\vert x-y_\eps \vert^{N-2}}  dx +\int_{B(x_\eps, r_\eps)\setminus B(x_\eps, \mu_\eps)} \frac{\ue^\frac{N+2}{N-2}}{\vert x-y_\eps \vert^{N-2}}  dx\\  
& \quad= \mathcal O\left(\mu_{\eps}^{\frac{N-2}{2}} \vert y_\eps-x_\eps \vert^{2-N} \right)  +  \alpha_\eps^\frac{N+2}{N-2} r_\eps  \int_{B(x_\eps, r_\eps)\setminus B(x_\eps, \mu_\eps)} \frac{1}{\vert x-y_\eps \vert^{N-2}}  \frac{1}{\vert x-x_\eps \vert}\diff x\\
&\qquad+ \mu_{\eps}^{\frac{N}{2}}\int_{B(x_\eps, r_\eps)\setminus B(x_\eps, \mu_\eps)} \frac{1}{\vert x-y_\eps \vert^{N-2}}  \frac{1}{\vert x-x_\eps \vert^{N+1}}dx  \\
& \quad= \mathcal O\left(\mu_{\eps}^{\frac{N-2}{2}} \vert y_\eps-x_\eps \vert^{2-N} \right)  \\
&\qquad+  \alpha_\eps^\frac{N+2}{N-2} r_\eps  \left(\int_{B(x_\eps, r_\eps)\setminus B(x_\eps, \mu_\eps)} \frac{1}{\vert x-y_\eps \vert^{p(N-2)}}   \diff x\right)^\frac{1}{p} \left(\int_{B(x_\eps, r_\eps)\setminus B(x_\eps, \mu_\eps)} \frac{1}{\vert x-x_\eps \vert^{p'}}  \diff x\right)^\frac{1}{p'} \\
&\qquad + \mathcal O\left(   \frac{\mu_{\eps}^{\frac{N}{2}}}{\vert y_\eps-x_\eps \vert^{N+1}}\int_{(B(x_\eps, r_\eps)\setminus B(x_\eps, \mu_\eps))\cap B(y_\eps, \frac{\vert x_\eps -y_\eps\vert}{2})} \frac{1}{\vert x-y_\eps \vert^{N-2}}  \diff x \right)\\  
&\qquad + \mathcal O\left(\frac{\mu_{\eps}^{\frac{N}{2}}}{\vert x_\eps-y_\eps \vert^{N-2}} \int_{(B(x_\eps, r_\eps)\setminus B(x_\eps, \mu_\eps))\setminus B(y_\eps, \frac{\vert x_\eps -y_\eps\vert}{2})}  \frac{1}{\vert x-x_\eps \vert^{N+1}}\diff x\right) \\
& \quad = \mathcal O\left(\mu_{\eps}^{\frac{N-2}{2}} \vert y_\eps-x_\eps \vert^{2-N} \right) + \mathcal O\left( \alpha_\eps^\frac{N+2}{N-2} r_\eps^{2}\right)\hskip.1cm.
\end{split}
\ee
 Thanks to (\ref{eq1.14}) and to (\ref{eq1.17}), this leads to 
 \be
\int_{B(x_\eps, r_\eps)} \vert x-y_\eps \vert^{2-N} \vert -\Delta  \ue \vert \diff x= \mathcal O(\mu_{\eps}^{\frac{N-2}{2}} \vert y_\eps-x_\eps \vert^{2-N} + \alpha_\eps).
\ee
which, thanks to (\ref{eq1.22}), proves (\ref{eq1.21}). 

In order to end the proof of the first part of Claim \ref{estim2}, we just have to prove that
\beq
\label{eq1.23}
\alpha_\eps =\sup_{ \partial B(x_\eps, r_\eps)} \ue = \mathcal O \left( \mu_{\eps}^{\frac{N-2}{2}} r_\eps^{2-N}  \right)\hskip.1cm.
 \eeq
For that purpose, we use the definition of $r_\eps$ to write that 
\be 
\psi_\eps(\beta r_\eps) \geq  \psi_\eps( r_\eps) 
\ee
for all $0<\beta <1$. Using (\ref{eq1.16}), this leads to
\be 
r_\eps^\frac{N-2}{2} \left(  \sup_{ \partial B(x_\eps, r_\eps)} \ue \right) \leq C (\beta r_\eps)^\frac{N-2}{2} \left(  \sup_{ \partial B(x_\eps, \beta r_\eps)} \ue \right) \hskip.1cm.
\ee 
Thanks to (\ref{eq1.21}), we obtain that 
\be 
  \sup_{ \partial B(x_\eps, r_\eps)} \ue  \leq C \beta ^\frac{N-2}{2} \left(\mu_{\eps}^\frac{N-2}{2} (\beta r_\eps)^{2-N}  +\sup_{ \partial B(x_\eps,  r_\eps)} \ue \right) \hskip.1cm.
\ee 
Choosing $\beta$ small enough clearly gives (\ref{eq1.23}) and thus the pointwise estimate on $\ue$ of Claim \ref{estim2}. The estimate on $\nabla \ue$ then follows from standard elliptic theory.\hfill $\square$

\medskip We now prove the following:

\begin{claim}\label{estim3}
If $r_\eps \rightarrow 0$ as $\eps\rightarrow 0$, then, up to passing to a subsequence,
\be
r_\eps^{N-2} u_\eps(x_\eps) u_\eps(x_\eps+r_\eps x) \rightarrow \frac{1}{ \vert x\vert^{N-2}}+ b \hbox{ in } C_{loc}^1\left(B\left(0, 2\right)\setminus\{0\}\right)\hbox{ as } \eps \rightarrow 0
\ee
where $b$ is some harmonic function in $B(0,2)$. Moreover, if $r_\eps < \rho_\eps$, then 
$b(0)=1$.
\end{claim}

\medskip {\bf Proof of Claim \ref{estim3}.} We set, for $x \in B(0,2)$,
\be
\tue (x)= \mu_\eps^{\frac{2-N}{2}}  r_\eps^{N-2} u_\eps(x_\eps + r_\eps x) 
\ee
which verifies
\beq
\label{eq1.24}
-\Delta \tue +r_\eps^2 \tilde{h}_\eps \tue =N(N-2)\left( \frac{\mu_\eps}{r_\eps}\right)^2 \tue ^\frac{N+2}{N-2} \hbox{ in }  B(0,2) 
\eeq
 where $\tilde{h}_\eps= h(x_\eps + r_\eps x)$. Thanks to Claim \ref{estim2}, there exists $C>0$ such that
 \beq
 \label{eq1.25}
 \tue(x) \leq   \frac{C}{\vert x\vert^{N-2}} \hbox{ in } B(0,2)\setminus \{0\}) .
\eeq
Then, thanks to standard elliptic theory, we get that, after passing to a subsequence, $\tue\rightarrow U$ in $C^1_{loc}\left(B(0,2)\setminus \left\{0\right\}\right)$ as $\eps \rightarrow 0$ where $U$ is a non-negative solution of
\be
-\Delta U = 0 \hbox{ in } B(0,2) \setminus \{0\} \hskip.1cm.
\ee
Then, thanks to the B\^{o}cher theorem on singularities of harmonic functions, we get that  
\be
U(x)=\frac{\lambda}{\vert x\vert^{N-2}} +b(x)
\ee
where $b$ is some harmonic function in $B(0,2)$ and $\lambda\ge 0$. Now, integrating (\ref{eq1.24}) on $B\left(0,1\right)$, we get that 
\be
\int_{\partial B(0,1)}  \partial_\nu \tue d\sigma =\int_{B(0,1)}  \left(r_\eps^2 \tilde{h}_\eps \tue -N(N-2)\left( \frac{\mu_\eps}{r_\eps}\right)^2 \tue^\frac{N+2}{N-2} \right)dx
\ee
Thanks to \eqref{eq1.25}, and since $r_\eps \to 0$ by hypothesis,
\be
\int_{B(0,1)}  r_\eps^2 \tilde{h}_\eps \tue dx \rightarrow 0 \quad \hbox{ as } \, \eps \rightarrow 0 
\ee
and, thanks to \eqref{eq1.25} and Claim \ref{estim1}, 
\be
N(N-2)\int_{B(0,1)}  \left( \frac{\mu_\eps}{r_\eps}\right)^2 \tue^\frac{N+2}{N-2} dx \rightarrow N(N-2) \int_{\R^N} B^\frac{N+2}{N-2}  dx = (N-2) \omega_{N-1} \hbox{ as } \eps \rightarrow 0\hskip.1cm.
\ee
On the other hand, we have that 
\be
\int_{\partial B(0,1)}  \partial_\nu \tue d\sigma \rightarrow (2-N) \omega_{N-1} \lambda \hbox{ as } \eps \rightarrow 0\hskip.1cm.
\ee
We deduce that $\lambda=1$, which proves the first part of Claim \ref{estim3}.

\medskip Now, if $r_\eps <\rho_\eps$, we have thanks to the definition of $r_\eps$ that
\be
\psi_\eps'(r_\eps) =0\hskip.1cm.
\ee
Setting  $\tilde{\psi}_\eps(r) = \left( \frac{r_\eps}{\mu_\eps}\right)^{\frac{N-2}{2}} \psi_\eps(r_\eps r)$ for $0<r<2$, we see that
\be
\tilde{\psi}_\eps(r) \rightarrow  \frac{r^\frac{N-2}{2}}{\omega_{N-1} r^{N-1}} \int_{\partial B(0,r)} U d\sigma =r^{-\frac{N-2}{2}} + r^{\frac{N-2}{2}} b(0)\hskip.1cm. 
\ee
We deduce that $b(0)=1$, which ends the proof of Claim \ref{estim3}.\hfill $\square$

\medskip We prove at last the following:

\begin{claim}
\label{estim4}
Using the notations of Claim \ref{estim3}, we have that $b(0)\leq 0$ and $\nabla b (0) =0$.
\end{claim}

\medskip {\bf Proof of Claim \ref{estim4}.} We use the notation of the proof of  Claim \ref{estim3}.
Let us apply the Poho\v{z}aev identity (\ref{pohozaev1}) from Appendix \ref{Pohozaevidentities} to $\tue$ in $B(0,1)$. We obtain that 
\be
 \frac{1}{2} \int_{B(0,1)} r_\eps^2 \left( (N-2) \tilde{h}_\eps \tue^2 + \tilde{h}_\eps \langle x ,\nabla \tue^2  \rangle \right) dx = 
 \tilde{B}^\eps_1 +\tilde{B}^\eps_2
 \ee
 where 
 \be
 \begin{split}
 \tilde{B}^\eps_1=&  \int_{\partial B(0,1)} \left(\partial_\nu \tue \right)^2 +\frac{N-2}{2} \tue \partial_\nu \tue  - \frac{\vert \nabla \tue\vert^2}{2}  d\sigma \quad  \hbox{ and }\\
\tilde{B}^\eps_2 =& \frac{(N-2)^2}{2} \int_{\partial B(0,1)}  \left(\frac{\mu_\eps}{r_\eps} \right)^2  \tue^{2^*}  d\sigma \hskip.1cm.
 \end{split}
 \ee 
Thanks to Claim \ref{estim3}, we can pass to the limit to obtain that the right hand side is equal to 
 \be
  \int_{\partial B(0,1)} \left(\partial_\nu U \right)^2 + \frac{N-2}{2} U\partial_\nu U - \frac{\vert \nabla U\vert^2}{2}  d\sigma  \hskip.1cm.
 \ee
Since $b$ is harmonic, it is easily checked that it is just $-\frac{(N-2)^2\omega_{N-1} b(0)}{2}$. Moreover, when $N=3$, thanks to \eqref{eq1.25} and the dominated convergence theorem,  the  left side goes to zero,  which proves that $b(0)=0$. If $N\geq 4$, we have to make a more precise expansion of the left hand side. First integrating by parts we get
\begin{align*}
&\frac{1}{2} \int_{B(0,1)} r_\eps^2 \left( (N-2) \tilde{h}_\eps \tue^2 + \tilde{h}_\eps \langle x ,\nabla \tue^2 \rangle \right) dx \\
&= -\int_{B(0,1)} r_\eps^2 \left( \tilde{h}_\eps \tue^2 +\frac{1}{2} \tilde{u}_\eps^2 \langle x ,\nabla \tilde{h}_\eps \rangle\right) dx +o(1)\\
&=-\tilde{h}_\eps(0)r_\eps^2 \int_{B(0,1)} \tue^2 dx -r_\eps^2 \int_{B(0,1)} (\tilde{h}_\eps-\tilde{h}_\eps(0))\tue^2 + \frac{1}{2}\tilde{u}_\eps^2  \langle x ,\nabla \tilde{h}_\eps \rangle dx +o(1)\\
&= \eps r_\eps^2\left(-V(x_\eps) \int_{B(0,1)} \tue^2 dx + \mathcal O\left(  r_\eps \int_{B(0,1)}   \vert x\vert \tue^2  dx \right)\right)+o(1)
\end{align*}

Then, thanks to Claim \ref{estim1} and Claim \ref{estim2}, we have easily for $N\geq 5$ that 
\begin{align}
\int_{B(0,1)} \tue^2 \,  dx =\left( \frac{r_\eps}{\mu_\eps} \right)^{N-4} \left(\int_{\R^N} B^2 \diff x + o(1) \right)
\end{align}
and
\begin{align}
\int_{B(0,1)}   \vert x\vert \tue^2 \diff x  = \mathcal O\left( \left( \frac{r_\eps}{\mu_\eps} \right)^{N-4-1} \right).
\end{align}
In particular
\begin{equation}
\label{estsign}
\lim_{\eps \rightarrow 0 } -\eps r^{2}_\eps V(x_\eps) \int_{B(0,1)} \tue^2 dx = -\frac{(N-2)^2\omega_{N-1} b(0)}{2}.
\end{equation}
Hence, using the fact that $V<0$, we obtain that $b(0)\leq 0$ for $N\geq 5$.
Similarly, for $N=4$,
\begin{align}
\int_{B(0,1)} \tue^2 dx =(1+o(1))\log \left( \frac{r_\eps}{\mu_\eps} \right)
\end{align}
and
\begin{align}
\int_{B(0,1)}  \vert x\vert \tue^2 \diff x  = \mathcal O(1),
\end{align}
which also proves that $b(0)\leq 0$. 
In order to prove the second part of Claim \ref{estim4}, we apply the Poho\v{z}aev identity (\ref{pohozaev4}) of Appendix \ref{Pohozaevidentities} to $\tue$ in $B(0,1)$. We obtain that 
\beq
\label{eq1.26}
\begin{split}
&\qquad \int_{\partial B(0,1)} \left( \frac{\vert \nabla \tue \vert^2}{2} \nu  -  \partial_\nu \tue \nabla \tue \right) \,d\sigma \\
&\quad = - \int_{B(0,1)} r_\eps^2 \tilde{h}_\eps \frac{\nabla \tue^2 }{2} \diff x+   \int_{\partial B(0,1)} \frac{(N-2)^2}{2}\left(\frac{\mu_\eps}{r_\eps} \right)^2  \tue^{2^*} \nu \, d\sigma \hskip.1cm.
\end{split}
\eeq 
It is clear that
\be
\int_{\partial B(0,1)} \left( \frac{\vert \nabla \tue \vert^2}{2} \nu -\partial_\nu \tue \nabla \tue  \right) d\sigma  \rightarrow \int_{\partial B(0,1)} \left(\frac{\vert \nabla U \vert^2}{2} \nu - \partial_\nu U \nabla U   \right) d\sigma \quad  \hbox{ as } \eps\rightarrow 0\hskip.1cm.
\ee
Moreover, thanks to the fact that $b$ is harmonic, we easily get that
\be
\int_{\partial B(0,1)} \left( \frac{\vert \nabla U \vert^2}{2} \nu - \nabla U \partial_\nu U\right)\, d\sigma = (N-2) \omega_{N-1} \nabla b(0)\hskip.1cm.
\ee
It remains to deal with the right-hand side of (\ref{eq1.26}). It is clear that
\be
\int_{\partial B(0,1)} \left(\frac{\mu_\eps}{r_\eps} \right)^2 \tue^{2^*}\nu d\sigma \rightarrow 0 \quad  \hbox{ as } \eps \rightarrow 0\hskip.1cm.
\ee 
Then we rewrite the first term of the right-hand side of (\ref{eq1.26}) as
\begin{align*}
 \int_{B(0,1)} r_\eps^2 \tilde{h}_\eps \frac{\nabla \tue^2 }{2} \diff x &= - \int_{ B(0,1)} r_\eps^2  \frac{\nabla \tilde{h}_\eps }{2}\tue^2 \diff x +o(1) = \mathcal O\left( \eps r_\eps^{3} \int_{ B(0,1)} \tue^2  \diff x \right)\, .
\end{align*}
Then, thanks to \eqref{estsign}, we have
$$
\lim_{\eps \rightarrow 0 }  \int_{B(0,1)} r_\eps^2 \tilde{h}_\eps \frac{\nabla \tue^2 }{2} \diff x =0  \, .
$$
Finally collecting the above informations, and passing to the limit $\eps\to 0$ in (\ref{eq1.26}), we get that $\nabla b(0)=0$, which achieves the proof of Claim \ref{estim4}.
\hfill$\square$

\medskip We are now in a position to end the proof of Proposition \ref{estim}. 

\begin{proof} 
[Proof of Proposition \ref{estim}]
If $\rho_\eps \rightarrow 0$ as $\eps\rightarrow 0$, then we deduce the proposition from Claims \ref{estim3} and \ref{estim4}. If $\rho_\eps\not\to 0$ as $\eps\to 0$, then claims \ref{estim3} and \ref{estim4} give that $r_\eps\not \to 0$ as $\eps\to 0$. Then, using the Harnack inequality (\ref{eq1.16}), one can extend the result of Claim \ref{estim2} to $B(x_\eps,2\rho_\eps)\setminus\{x_\eps\}$, which proves the first part of Proposition \ref{estim} when $\rho_{\eps} \not\rightarrow 0$.
\end{proof}

\subsection{Proof of Proposition \ref{proposition preliminaries multi bis}}

Let us now turn to the proof of Proposition \ref{proposition preliminaries multi bis}. This is done in two steps.
In Claim \ref{iso1}, mimicking \cite{DruetHebey09}, we exhaust  a family of critical points of $\ue$, $\left(x_{1,\eps},\dots, x_{N_\eps,\eps}\right)$, such that each sequence $\left(x_{i_\eps, \eps}\right)$ satisfies the assumptions  of Proposition  \ref{estim} with
$$\rho_\eps= \min_{1\leq i\leq N_{\eps}, i\not=i_\eps} \left\{ \vert x_{i,\eps} -x_{i_\eps, \eps} \vert, d(x_{i_\eps, \eps}, \partial\Omega) \right\}\hskip.1cm.$$
In Claim \ref{iso4}, we prove that these concentration points are in fact isolated. In particular, this shows that $(\ue)$ develops only finitely many concentration points.

First of all, we extract sequences (whose number is {\it a priori} not bounded) of critical points of $u_\eps$ which are candidates to be the blow-up points.
\begin{claim}
\label{iso1} 
There exists $ D >0$ such that for all $\eps >0$, there exists $n_\eps \in {\mathbb N}^{*}$ and $N_\eps$ critical points of $\ue$, denoted by $(x_{1,\eps},\dots, x_{n_\eps,\eps})$ such that :
\be
\begin{split}
&d(x_{i,\eps},\partial\Omega ) \ue(x_{i,\eps})^\frac{2}{N-2}  \geq 1 \hbox{ for all } i\in [1,n_{\eps}]\hskip.1cm,\\
&\vert x_{i,\eps}-x_{j,\eps}\vert \ue(x_{i,\eps})^\frac{2}{N-2}  \geq 1 \hbox{ for all } i \not= j \in  [1,n_{\eps}]\hskip.1cm, 
\end{split}
\ee
and 
\be
\left( \min_{i\in [1,n_\eps ]} \vert x_{i,\eps} -x \vert  \right) \ue(x)^\frac{2}{N-2} \leq D
\ee
for all $x\in \Omega$ and all $\eps>0$.
\end{claim}

\medskip {\bf Proof of Claim \ref{iso1}.} First of all, we claim that 
\beq
\label{eq2.2}
\left \{ x \in\Omega \hbox{ s.t. } \nabla \ue(x)=0 \hbox{ and } d(x,\partial\Omega) \ue(x)^\frac{2}{N-2}\geq 1 \right \}\not=\emptyset
\eeq
for $\eps$ small enough. Let us prove (\ref{eq2.2}). Let $y_\eps \in \Omega$ be a point where $\ue$ achieves its maximum. We set $\mu_\eps= \ue(y_\eps)^{-\frac{2}{N-2}}\rightarrow 0$ as $\eps\rightarrow 0$. We set also  for all $x\in\Omega_\eps= \{x\in\R^N \hbox{ s.t. } y_\eps +\mu_\eps x \in\Omega \}$,
\be
\tue(x) =\mu_\eps^\frac{N-2}{2} \ue(y_\eps +\mu_\eps x)\, , 
\ee
which verifies
 \be
 -\Delta \tue +\mu_\eps^2 \tilde{h}_\eps \tue = N(N-2)\tue^\frac{N+2}{N-2} \hbox{ in } \Omega_{\eps},
 \ee
 where $\tilde{h}_\eps=h(y_\eps + \mu_\eps x)$. Note that $0\leq \tue\leq \tue(0)= 1$. Thanks to standard elliptic theory, we get that $\tue \rightarrow U$ in $C^{1}_{loc}(\Omega_0)$ where $U$ satisfies
\be 
-\Delta U = U^\frac{N+2}{N-2} \hbox{ in } \Omega_0 \hbox{ and } 0\leq U \leq 1,
\ee
and where $\ds \Omega_0 =\lim_{\eps \rightarrow 0} \Omega_\eps $. Moreover, $U\not\equiv 0$ by Harnack's inequality, see \cite[Theorem 4.17]{HanLin}. Thanks to \cite[Theorem 2]{Dancer}, we have $\Omega_0=\R^N$, which proves that $d(y_\eps,\partial\Omega)\ue(y_\eps)^\frac{2}{N-2} \rightarrow +\infty $ as $\eps\to 0$. This ends the proof of (\ref{eq2.2}).

Now, applying Lemma \ref{iso0}, see Appendix \ref{generallemma}, for $\eps$ small enough, there exist $n_\eps \in \N^{*}$ and $n_\eps$ critical points of $\ue$, denoted by $(x_{1,\eps},\dots, x_{n_\eps,\eps})$, such that :
\be
\begin{split}
&d(x_{i,\eps},\partial\Omega ) \ue(x_{i,\eps})^\frac{2}{N-2}  \geq 1 \hbox{ for all } i\in [1,n_{\eps}]\hskip.1cm,\\
&\vert x_{i,\eps}-x_{j,\eps} \vert \ue(x_{i,\eps})^\frac{2}{n-2}  \geq 1 \hbox{ for all } i \not= j \in  [1,n_{\eps}] \hskip.1cm,
\end{split}
\ee
and 
\beq
\label{eq2.3}
\left( \min_{i\in [1,n_\eps ]} \vert x_{i,\eps} -x \vert  \right) \ue(x)^\frac{2}{N-2} \leq 1
\eeq
for every critical point $x$ of $\ue$ such that $d(x,\partial\Omega)\ue(x)^\frac{2}{N-2} \geq 1$. It remains to show that there exists $D>0$ such that 
\be
\left( \min_{i\in [1,n_\eps ]} \vert x_{i,\eps} -x \vert  \right) \ue(x)^\frac{2}{N-2} \leq D
\ee
for all $x\in \Omega$. We proceed by contradiction, assuming that 
\beq
\label{eq2.4}
\sup_{x\in\Omega} \left( \left( \min_{i\in [1,n_\eps ]} \vert x_{i,\eps} -x \vert  \right) \ue^\frac{2}{N-2}(x)\right) \rightarrow +\infty
\eeq
as $\eps \rightarrow 0$. Let $z_\eps\in \Omega$ be such that
\be
\left( \min_{i\in [1,n_\eps ]} \vert x_{i,\eps} -z_\eps \vert  \right) \ue(z_\eps)^\frac{2}{N-2} = 
\sup_{x\in\Omega} \left( \left( \min_{i\in [1,n_\eps ]} \vert x_{i,\eps} -x \vert  \right) \ue(x)^\frac{2}{N-2}\right).
\ee
We set $\hat{\mu}_\eps =\ue(z_\eps)^{-\frac{2}{N-2}}$ and $S_\eps =\{ x_{1,\eps}, \dots, x_{n_\eps,\eps} \}$. Thanks to (\ref{eq2.4}), we check that 
\be 
\hat{\mu}_\eps \rightarrow 0\hbox{ as }\eps\rightarrow 0
\ee
and that 
\beq
\label{eq2.5}
\frac{d(S_\eps, z_\eps)}{\hat{\mu}_\eps} \rightarrow +\infty \hbox{ as } \eps\rightarrow 0 \hskip.1cm.
\eeq
Then we set, for all $x\in \hat{\Omega}_{\eps}=\{ x\in \R^3 \hbox{ s.t. } z_\eps + \hat{\mu}_\eps x \in \Omega \}$, 
\be 
\hue(x) = \hat{\mu}_\eps^\frac{N-2}{2} \hue(z_\eps+\hat{\mu}_\eps x) \, , 
\ee
which verifies
\be
-\Delta \hue +\hat{\mu}_\eps^2 \hat{h}_\eps \hue = N(N-2)\hue^\frac{N+2}{N-2} \hbox{ in } \Omega_{\eps}\, ,
\ee
where $\hat{h}_\eps=h(z_\eps + \hat{\mu}_\eps x)$. Note that $\hue(0)=1$ and also that 
\be
\lim_{\eps\rightarrow 0} \sup_{B(0,R)\cap \Omega_\eps} \hue = 1
\ee
for all $R>0$ thanks to (\ref{eq2.4}) and (\ref{eq2.5}). Standard elliptic theory gives then that $\hue \rightarrow \hat{U}$ in $C^{1}_{loc}(\hat{\Omega}_0)$ where $U$ satisfies
\be 
-\Delta \hat{U} = N(N-2)\hat{U}^\frac{N+2}{N-2} \hbox{ in } \hat{\Omega}_0 \hbox{ and } 0\leq \hat{U} \leq1
\ee
with $\ds \hat{\Omega}_0=\lim_{\eps \rightarrow 0} \hat{\Omega}_\eps $. As above, we deduce that $\hat{\Omega}_0=\R^N$, which gives that 
\beq
\label{eq2.6}
\lim_{\eps\rightarrow 0} d(z_\eps, \partial\Omega) \ue^\frac{2}{N-2}(z_\eps)\rightarrow +\infty \hskip.1cm.
\eeq
Moreover, thanks to \cite{CGS}, we know that
\be
\hat{U}(x) = \frac{1} { \left(1+ \vert x \vert^{2}  \right)^{\frac{N-2}{2}} } \hskip.1cm.
\ee
Since $\hat{U}$ has a strict local maximum at $0$, there exists $\hat{x}_\eps$, a critical point of $\ue$, such that $\vert z_{\eps} - \hat{x}_\eps \vert  = o(\hat{\mu}_\eps)$ and $\hat{\mu}_\eps \ue (\hat{x}_\eps)^{2}\rightarrow 1$ as $\eps \rightarrow 0$. Thanks to (\ref{eq2.5}) and (\ref{eq2.6}), this contradicts (\ref{eq2.3}) and proves Claim \ref{iso1}.\hfill$\square$

\medskip We define
\be
d_\eps = \min \left\{ d(x_{i, \eps},x_{j, \eps}), d(x_{i, \eps},\partial\Omega) \hbox{ s.t. } 1\leq i < j \leq n_\eps \right\}
\ee
and prove: 

\begin{claim}
\label{iso4}
There exists $d>0$ such that  $\de \geq d$.
\end{claim}

\medskip {\bf Proof of Claim \ref{iso4}.} Assume that $\de \rightarrow 0 $ as $\eps \rightarrow 0 $. There are two cases to consider~: either the distance between two critical points goes to $0$, or one of them goes to the boundary.

Up to reordering the concentration points, we can assume that 
\be
\de=d(x_{1,\eps}, x_{2,\eps}) \hbox{ or } d(x_{1,\eps}, \partial\Omega) \hskip.1cm. 
\ee
For $x\in \Omega_\eps =\{x\in\R^3 \hbox{ s.t. } x_{1,\eps} + d_{\eps} x \in \Omega \}$, we set
\be
\tue (x)= \de^{\frac{N-2}{2}}  u_\eps(x_{1,\eps} + \de x) 
\ee
which verifies
\be
-\Delta \tue +\de^2 \tilde{h}_\eps \tue =N(N-2)\tue ^\frac{N+2}{N-2} \hbox{ in }  \Omega_\eps, 
\ee
 where $\tilde{h}_\eps= h(x_{1,\eps} + \de x)$. We have, up to a harmless rotation, 
\be
\lim_{\eps \rightarrow 0} \Omega_\eps =\Omega_0=   \R^N \hbox{ or } ]-\infty;d[ \times \R^{N-1} \hbox{ where } d\geq 1\hskip.1cm.
\ee
We also set
\be
\tilde{x}_{i,\eps} =\frac{x_{i,\eps} - x_{1,\eps}}{\de}\hskip.1cm.
\ee
We claim that, for any sequence $i_\eps \in [1, n_\eps]$ such that
\beq
\label{eq2.7} 
\tue(\tilde{x}_{i_\eps,\eps})= \mathcal O(1) \hskip.1cm,
\eeq
we have that
\beq 
\label{eq2.8}
\sup_{B(\tilde{x}_{i_\eps,\eps},\frac{1}{2})} \tue = \mathcal O(1) \hskip.1cm.
\eeq
Indeed, let $y_{\eps} \in\overline{B(\tilde{x}_{i_\eps,\eps},\frac{1}{2})}$ be such that $\displaystyle \sup_{B(\tilde{x}_{i_\eps,\eps},\frac{1}{2})} \tue = \tue( y_{\eps} ) $ and assume by contradiction that 
\beq
\label{2.9}
\tue( y_{\eps} )^\frac{2}{N-2} \rightarrow +\infty \hbox{ as }\eps \rightarrow 0\hskip.1cm.
\eeq
Thanks to the definitions of $\de$ and $y_\eps$ and to the last assertion of Claim \ref{iso1}, we can  write that
\be
\vert \de (y_\eps-\tilde{x}_{i_\eps,\eps}) \vert \ue(x_{1,\eps} +\de y_\eps)^\frac{2}{N-2} \leq D
\ee
so that 
\beq
\label{2.10}
\vert y_\eps - \tilde{x}_{i_\eps,\eps}\vert = o(1)\hskip.1cm.
\eeq
For $x \in B(0, \frac{1}{3 \hat{\mu}_\eps})$ and $\eps$ small enough, we set 
\be
\hue (x)= \hat{\mu}_\eps^{\frac{N-2}{2}}  \tue (y_\eps + \hat{\mu}_\eps x)\, ,
\ee
where $\hat{\mu}_\eps= \ue(y_\eps)^{-\frac{2}{N-2}}$. It satisfies
\be
-\Delta \hue +(\hat{\mu}_\eps d_\eps)^2 \hat{h}_\eps \hue =\hue ^\frac{N+2}{N-2} \hbox{ in }  B(0, \frac{1}{3 \hat{\mu}_\eps}) \quad \hbox{ and }\hue(0) = \sup_{B(0, \frac{1}{3 \hat{\mu}_\eps})} \hue =1, 
\ee
where $\hat{h}_\eps= \tilde{h}_\eps(y_\eps + \hat{\mu}_\eps x)$. Thanks to (\ref{2.9}), $B(0, \frac{1}{3 \hat{\mu}_\eps}) \rightarrow \R^N$ as $\eps \rightarrow +\infty$. Then $\left(\hue\right)$ is uniformly locally bounded and, by standard elliptic theory, $\hue$ converges to $\hat{U}$ in $C^{1}_{loc}(\R^N)$ where $\hat{U}$ satisfies
\be 
-\Delta \hat{U} = \hat{U}^\frac{N+2}{N-2} \hbox{ in } \R^N \hbox{ and } 0\leq \hat{U} \leq1=\hat{U}(0) \hskip.1cm.
\ee
Thanks to the classification of Caffarelli--Gidas--Spruck \cite{CGS} and to the fact that  $\frac{\tilde{x}_{i_\eps,\eps} - y_{\eps}}{\hat{\mu}_\eps}$ is bounded, we can write that
\be
\liminf_{\eps\to 0} \frac{\tue (x_{i_\eps,\eps})}{\tue(y_\eps)} >0
\ee
which is a contradiction with (\ref{eq2.7}) and (\ref{2.9}), and achieves the proof of (\ref{eq2.8}).

\medskip For $R>0$, we set $S_{R,\eps} = \left\{ \tilde{x}_{i,\eps} \, \vert \,  \tilde{x}_{i,\eps}\in B(0, R)\right\}$. Thanks to the definition of $\de$, up to a subsequence, $S_{R,\eps} \rightarrow S_{R}$ as $\eps \rightarrow 0$, where $S_{R}$ is a non-empty finite set, then up to performing a diagonal extraction, we can define the countable set 
\be 
S=\bigcup_{R>0} S_{R}\hskip.1cm.
\ee 
Thanks to the previous definition, we are ready to prove the following assertion~:
\beq
\label{eq2.11}
\forall\,  i_{\eps} \in [1,n_\eps ]\hbox{ s.t. } d(x_{i_\eps ,\eps}, x_{1,\eps}) = \mathcal O(\de)\, ,\qquad  \tue(\tilde{x}_{i_\eps ,\eps}) \rightarrow +\infty \, \hbox{ as }\eps\to 0\hskip.1cm.
\eeq
Assume that there exists $i_\eps$ such that $d(x_{i_\eps ,\eps}, x_{1,\eps}) = \mathcal O(\de)$ with $\tue(\tilde{x}_{i_\eps ,\eps})$ bounded, then for all sequences $j_\eps$ such that $d(x_{j_\eps ,\eps}, x_{1,\eps}) = \mathcal O(\de)$,  $\tue(\tilde{x}_{j_\eps ,\eps})$ is bounded. Indeed, if  there exists a sequence $j_\eps$ such that $d(x_{j_\eps ,\eps}, x_{1,\eps}) = \mathcal O(\de)$  and $\tue(\tilde{x}_{j_\eps ,\eps})\rightarrow +\infty $ as $\eps\rightarrow 0$, thanks to  Claim \ref{iso1}, we can apply Proposition \ref{estim} with $x_\eps = \tilde{x}_{j_\eps,\eps}$ and $\rho_{\eps} = \frac{\de}{3}$. We obtain that up to a subsequence  $\tue\rightarrow 0$ in $C^{1}_{loc}(B(\tilde{x},\frac{2}{3}))\setminus \{ \tilde{x} \}$, where $\ds \tilde{x}= \lim_{\eps \rightarrow 0} \tilde{x}_{j_\eps,\eps }$. But $\left(\tue\right)$ is uniformly bounded in $B(\tilde{y},\frac{1}{2})$, where  $\ds \tilde{y}= \lim_{\eps \rightarrow 0} \tilde{x}_{i_\eps,\eps }$. We thus obtain thanks to Harnack's inequality that $\tue(\tilde{x}_{i_\eps ,\eps})\rightarrow 0 $ as $\eps\rightarrow 0$,  which is a contradiction with the first or the second  assertion of Claim \ref{iso1}.

Thus we have proved that for every sequence $j_\eps$ such that $d(x_{j_\eps ,\eps}, x_{1,\eps}) = \mathcal O(\de)$, $\tue(\tilde{x}_{j_\eps ,\eps})$ is bounded. This proves that $\left(\tue\right)$ is uniformly bounded in a neighborhood of any finite subset of $S$. But thanks to Claim \ref{iso1}, $\tue$ is bounded in any compact subset of $\Omega_{0}\setminus S$. This clearly proves that $\tue$ is uniformly bounded on any compact of $\Omega_0$. Then, by standard elliptic theory, $\tue \rightarrow U$ in $C^{1}_{loc}(\Omega_0)$ as $\eps \rightarrow 0$, where $U$ is a nonnegative solution of
\be
-\Delta U = U^\frac{N+2}{N-2} \hbox{ in } \Omega_0 \hskip.1cm.
\ee 
But, thanks to the first or second assertion of Claim \ref{iso1}, we know that $U(0) \geq 1 $, hence we have necessarily that  $\Omega_0=\R^N$, and thus  $U$ possesses at least two critical points, namely $0$ and $\ds \check{x}_2 =\lim_{\eps \rightarrow 0}  \check{x}_{2,\eps}$. Thanks to the classification of Caffarelli--Gidas--Spruck \cite{CGS}, this is impossible. This ends the proof of (\ref{eq2.11}).

\medskip We are now going to consider two cases, depending on $\Omega_0$. 

\medskip {\it Case 1 : $\Omega_0= \R^N$.  } In this case, up to a subsequence, $d_\eps=d(x_{1,\eps},x_{2,\eps})$ and $\ds S=\{0,\tilde{x}_2=\lim_{\eps \rightarrow 0} \tilde{x}_{2,\eps}, \dots  \}$ contains at least two points. Applying Proposition \ref{estim} with $x_\eps = \tilde{x}_{i,\eps}$ and $\rho_{\eps} = \frac{\de}{3}$, we obtain that 
\be
\tue(0)\tue(x) \rightarrow H= \frac{1}{\vert x \vert^{N-2}} +\frac{\lambda_{2}}{\vert x-\tilde{x}_2  \vert^{N-2}}+ \tilde{b} \hbox{ in }  C^{1}_{loc} (\R^N \setminus S) \hbox{ as } \eps\rightarrow 0
\ee
where $\tilde{b}$ is a harmonic function in $\Omega_{0}\setminus \{S\setminus \{0, \check{x}_2 \}\}$, and $\lambda_2 >0$. Moreover $\tilde{b}(0)\leq -\lambda_2$. We prove in the following that $\tilde{b}$ is nonnegative, which will give a contradiction and end the study of this case. To check that $\tilde{b}$ is nonnegative, for any positive number $r$, we rewrite $H$ as 
\be
H= \sum_{\tilde{x}_i \in S\cap B(0,r)} \frac{\lambda_{i}}{\vert x-\tilde{x}_i \vert^{N-2}} + \hat{b}_r,
\ee
where $\lambda_i>0$. Then, taking $R>r$ large enough, we get  that  $\hat{b}_r > \frac{-1}{r^{N-2}}$ on $\partial B(0,R)$. Moreover, for any $\tilde{x}_j \in B(0,R)\setminus B(0,r)$, there exist a neighborhood $V_{j,r}$ of $\tilde{x}_j$ such that $ \hat{b}_r >0$ on $V_{j,r} $. Thanks to the maximum principle, $\hat{b}_r > \frac{-1}{r^{N-2}}$ on $B(0,R)$, hence it is decreasing and lower bounded, then $\hat{b}_r \rightarrow \hat{b}$ on every compact set as $r \rightarrow +\infty$, we get that $\ds H=  \sum_{\tilde{x}_i \in S} \frac{\lambda_{i}}{\vert x-\tilde{x}_i \vert^{N-2}} + \hat{b}$ with $\hat{b}\geq 0$, which proves that $\tilde{b}\geq 0$. This is the contradiction we were looking for, and this ends the proof of Claim \ref{iso4} in this first case.

\medskip {\it Case 2 : $\Omega_0= ]-\infty, d[\times \R^{N-1}$.  } We still denote $S=\{0=\tilde{x}_1, \tilde{x}_2, \dots \}$ and we apply Proposition \ref{estim} with $ x_\eps= x_{i,\eps}$ and $\rho_\eps= \frac{d_\eps}{3}$ to get that
\be
\tue(0)\tue(x) \rightarrow H= \sum_{\tilde{x}_i \in S}  \frac{\lambda_i}{\vert x-\tilde{x}_i \vert^{N-2}} + \tilde{b} \hbox{ in } C^{1}_{loc} (\Omega_0 \setminus S)\, ,
\ee
 where $\lambda_i>0$, and $\tilde{b}$ is some harmonic function in $\Omega_0$. We extend $H$ to $\R^N$ by setting
$$
\hat{H}(x) = \left\{
\begin{array}{ll}
H(x) &\hbox{ if } x_1 \leq d, \\
-H(s(x)) &\hbox{ otherwise, }
\end{array}\right.
$$
where $s$ is the reflection with respect to the hyperplane $ \{ d\} \times \R^{N-1} $.  We also extend $\tilde{b}$ by setting  
\be
\hat{H} =  \sum_{\tilde{x}_i \in S}  \left( \frac{\lambda_i}{\vert x-\tilde{x}_i \vert^{N-2}} -  \frac{\lambda_i}{\vert s(x)-\tilde{x}_i\vert^{N-2}} \right) +\hat{b}\hskip.1cm.
\ee
It is clear that $\hat{b}$ is harmonic on $\R^N$ and satisfies $\hat{b} \geq 0$ in $\Omega_0$ and $\hat{b}\le 0$ in $\R^N\setminus \Omega_0$. This can be proved as in Case 1. For $\mathcal{G}_R$ the Green function of the Laplacian on the ball $B(0, R)$ centered in $0$ with radius $R$, we get thanks to the Green representation formula that 
\be
\hat{b} (x) = \int_{\partial B(0,R)} \partial_\nu \mathcal{G}_R (x,y) \hat{b}(y) d\sigma \, .
\ee 
Since 
$$\partial_\nu {\mathcal G}_R\left(x,y\right)= \frac{R^2-\vert x\vert^2}{\omega_{N-1} R \left\vert x-y\right\vert^N} \qquad \text{ on } \partial B(0,R), $$
this gives that 
\be
\partial_1 \hat{b} (0) = \frac{N}{\omega_{N-1} R^{N}} \int_{\partial B(0,R)} y_1 \hat{b}(y) d\sigma \hskip.1cm.
\ee 
Now we decompose $\partial B(0,R)$ into three sets, namely 
\begincal
A&=&\{y\in  \partial B(0,R)\hbox{ s.t. }  y_1 \geq d \}\hskip.1cm,\\ 
B&=&\{ y\in  \partial B(0,R) \hbox{ s.t. } 0 \leq y_1\leq d \}\hskip.1cm,\\
C&=&\{ y\in \partial B(0,R) \hbox{ s.t. } y_1\leq 0 \}\hskip.1cm.
\fincal 
In $A$ and $B$, we have that $y_1 \hat{b}(y) \leq d \hat{b}(y)$, and in $C$, we have that $y_1 \hat{b}(y) \leq 0$. Since $\hat{b}\ge 0$ in $C$, we arrive at 
\be
\partial_1 \hat{b} (0) \leq \frac{Nd}{\omega_{N-1} R^{N}} \int_{A\cup B}  \hat{b}(y) d\sigma \leq  \frac{Nd}{\omega_2 R^{N}} \int_{\partial B(0,R)} \hat{b}(y) d\sigma = \frac{Nd \hat{b}(0)}{R}\hskip.1cm.
\ee 
Passing to the limit $R\rightarrow +\infty$ gives that $\partial_1 \hat{b}(0) \leq 0$. In order to obtain a contradiction,  we rewrite $H$ in a neighborhood of $0$ as
 \be
 H(x) = \frac{1}{\vert x \vert^{N-2}} + \check{b}(x)\, , 
 \ee 
 where  
 \be
 \check{b}(x)= \hat{b}(x) -\frac{1}{\vert s(x) \vert^{N-2} } + \sum_{\check{x}_i \in S \setminus \{0\}} \lambda_i \left( \frac{1}{\vert x-\check{x}_i \vert^{N-2}} -  \frac{1}{\vert s(x) - \check{x}_i \vert^{N-2}} \right) \hskip.1cm. 
 \ee
As is easily checked, $\partial_1  \check{b}(0) <0$, which is a contradiction with Proposition \ref{estim}. This ends the proof of  Claim \ref{iso4} in this second case.\\

\begin{proof}
[Proof of Proposition \ref{proposition preliminaries multi bis}]
It only remains to prove (v) and (vi) of Proposition \ref{proposition preliminaries multi bis}. Assertion (vi) is true locally around each concentration point by applying the first part of  Proposition \ref{estim}, and extending it to the whole domain using Harnack's inequality. Finally (v) follows directly from (vi). Indeed, all the $\mu_{i,\eps}$ are comparable by Harnack's inequality, then multiplying the equation by $\mu_{1,\eps}^{-\frac{N-2}{2}}$ and passing to the limit thanks to (vi) gives the desired result.
\end{proof} 

\section{Necessity of coercivity}
\label{Appcoer}
In this section, we briefly recall why the operator $-\Delta + h$ is necessarily coercive as soon as there exists a blowing-up sequence satisfying \eqref{eq1.1}.

\begin{lemma} If there exists $u \in C^{2,\eta}_0(\Omega)$ such that $u>0$ and $-\Delta u + h u >0$ on $\Omega$, then $-\Delta + h$ is coercive.
\end{lemma}
\begin{proof}
See Appendix B of \cite{Druet2004} for the case where $\Omega$ is a compact manifold. The proof applies verbatim for a domain with Dirichlet boundary condition.
\end{proof}

In particular, the operator $-\Delta + h_\eps$ must be coercive for every $\eps > 0$. But in fact, $-\Delta + h$ must also be coercive under our assumption. Indeed, this is proved in Appendix B of \cite{Druet2004}, when $\Omega$ is a compact manifold and under the assumption that
there exists a finite number of sequences $(x_i^\eps)_{1\leq i\leq k}\in \Omega $ and $\mu_i^\eps \rightarrow 0 $
such that
$$\frac{1}{C}\sum_{i=1}^k B_{i,\eps} \leq u_\eps \leq C \sum_{i=1}^k B_{i,\eps} $$
for some $C >0$, where $B_{i,\eps}(x)=B\left( \frac{x-x_\eps^i}{\mu_\eps^i} \right)$. This hypothesis is clearly verified thanks to Proposition \ref{proposition preliminaries multi bis}. Now the proof in the domain case with Dirichlet boundary data follows verbatim the one presented in Appendix B of \cite{Druet2004}.
 
\section{Harnack's inequality}

\begin{lemma}
\label{harn}
Let $u_\eps$  satisfy the hypotheses of Proposition \ref{proposition preliminaries multi bis}. Then there exists $C>0$ depending only on $C_0$ and $\Vert h\Vert_\infty$ such that  
\beq
\frac{1}{C} \max_{\partial B(x_\eps, r)} \left(\ue+r\left\vert \nabla u_\eps\right\vert\right) \leq \frac{1}{\omega_{N-1} r^{N-1}} \int_{\partial B({x_\eps},r)} u_\eps d\sigma \leq C \min_{\partial B(x_\eps, r)}\ue  
\eeq
for all $r\in[0 , \frac{5}{2} \rho_\eps ]$ and all $ \eps>0$.
\end{lemma}

The proof follows \cite[Lemma 1.3]{DruetHebey09}.

\begin{proof}  Let $0  <r_\epsilon < \frac{5}{2} \rho_\eps $.
We set  
\be
\tue (x)= r_\eps^{\frac{N-2}{2}}  u_\eps(\tilde{x}_\eps + r_\eps x) 
\ee
which verifies
\beq
-\Delta \tue +r_\eps^2 \tilde{h}_\eps \tue =N(N-2)\tue ^\frac{N+2}{N-2} \hbox{ in } B\left(0, \frac{\rho_\eps}{r_\eps}\right)\hskip.1cm, 
\eeq
where $\tilde{h}_\eps= h\left(\tilde{x}_\eps + r_\eps x\right)$. Thanks to \eqref{eq1.5}, we have 
$$ \tue \leq \frac{C_0}{\vert x\vert^{N-2}} ,$$
in particular $\tue$ is uniformly bounded on $B(0,2)\setminus B(0,\frac{1}{2})$. Hence, applying the Moser--Harnack inequality  \cite[Theorem 4.17]{HanLin}, we have for all $x\in  B(0,3/2)\setminus B(0,\frac{2}{3})$ and $0<r < \frac{1}{6}$ that
$$\max_{B(x,r)} \tue \leq C\left( \min_{B(x,r/2)} \tue + r \Vert \tue \Vert_\infty \Vert -r_\eps^2 \tilde{h}_\eps  +N(N-2)\tue ^\frac{4}{N-2} \Vert_N \right),$$
with $C>0$ depending only on $N$. Then taking $r$ small enough depending only on $C_0$ and $\Vert h_\infty \Vert_\infty$, we have 
$$\max_{B(x,r)} \tue \leq  C \min_{B(x,r/2)} \tue .$$
Then using a covering argument, we get 
$$\max_{B(0,5/4)\setminus B(0,4/5)} \tue \leq  C \min_{B(0,5/4) \setminus B(0,4/5)} \tue .$$
Finally, using standard elliptic theory,
$$\max_{B(0,7/6)\setminus B(0,6/7)} \vert \nabla \tue\vert  \leq C \max_{B(0,7/6)\setminus B(0,6/7)} \tue ,$$
which achieves the proof.
\end{proof}

\section{General Poho\v{z}aev's identities}
\label{Pohozaevidentities}

For the sake of completeness, we derive here several forms of the classical Poho\v{z}aev identity \cite{Pohozaev} we used in this paper. Assume that $u $ is a $C^2$ solution of
 \be
-\Delta u = N(N-2)u^\frac{N+2}{N-2} -hu \hbox{ in } \Omega \hskip.1cm.
\ee
Multiplying this equation by $\langle x,\nabla u \rangle$ and integrating by parts, one easily gets that 
\beq
 \label{pohozaev1}
 \frac{1}{2} \io \left( (N-2) h u^2 +  h \langle x,\nabla u^2 \rangle \right) dx = 
 B_1 +B_2 ,
 \eeq
 where 
 \be
 \begin{split}
 B_1&=  \ibo \left( \langle x, \nabla u \rangle \partial_\nu u  + \frac{N-2}{2} u\partial_\nu u- \langle x,\nu \rangle \frac{\vert \nabla u\vert^2}{2} \right)d\sigma \hbox{ and }\\
B_2 &= \frac{(N-2)^2}{2}\ibo \langle x,\nu \rangle   \frac{u^{2^*}}{2^*}  d\sigma \hskip.1cm.
 \end{split}
 \ee
Hence, if $u=0$ on $\partial\Omega$, we get that 
\beq
\label{pohozaev2}
\io h\left( (N-2)u^2 + \langle x,  \nabla u^2 \rangle \right) dx = 
\ibo  \langle x,\nu \rangle \left(\partial_\nu u \right)^{2} d\sigma \hskip.1cm.
\eeq
Integrating by parts again, we get the Poho\v{z}aev identity in its usual form~: 
\beq
\label{pohozaev3} 
\io \left( h + \frac{\langle x,\nabla h \rangle}{2} \right)u^2 dx = -\frac{1}{2}
\ibo  \langle x,\nu \rangle \left(\partial_\nu u \right)^{2} d\sigma \hskip.1cm.
\eeq 
In a similar way, multiplying the equation by $\nabla u$ and integrating by parts, one can derive the following Poho\v{z}aev's identity~:
\beq
\label{pohozaev4}
\ibo \left( \frac{\vert \nabla u \vert^2}{2} \nu  - \partial_\nu u \nabla u  -   \frac{(N-2)^2}{2}u^{2^*}\nu \right) d\sigma  = -\io h \frac{\nabla u^2 }{2} dx\hskip.1cm.
\eeq 

\section{A general simple lemma on functions}\label{generallemma}


\begin{lemma}
\label{iso0} 
Let $\Omega$ be a smooth bounded domain of $\R^N$ and $u\in C_0^1\left(\Omega\right)$  positive on $\Omega$. Assume that 
$$ K_u := \{x\in \Omega \hbox{ s.t. } \nabla u (x)=0 \hbox{ and } d(x,\partial\Omega) u^\frac{2}{N-2}(x)\geq 1 \} $$
is non-empty. 

Then there exist $n \in \N^{*}$ and $n$ points of $K_u$, denoted by $(x_{1},\dots, x_{n})$, such that 
\be
\begin{split}
&\vert x_{i}-x_{j}\vert u(x_{i})^\frac{2}{N-2}  \geq 1 \hbox{ for all } i \not= j \in  [1,n] 
\end{split}
\ee
and 
\be
\left( \min_{i\in [1,n]} \vert x_{i} -x \vert  \right) u(x)^\frac{2}{N-2} \leq 1 \qquad \text{ for all } x \in K_u. 
\ee
\end{lemma}

\begin{proof}
Let $K_0 := K_u$. By assumption, $K_0$ is non-empty. Moreover, it is clear that $K_0$ is compact. We let $x_1\in K_0 $ and $K_1 \subset K_0$ be such that 
\be
u(x_1)= \max_{K_0} u
\ee 
and
\be 
K_1=\left\{ x\in K_0 \hbox{ s.t. } \vert x_1 -x\vert u(x)^\frac{2}{N-2} \geq 1 \right\} \hskip.1cm.
\ee
Then we proceed by induction. Assume that we have constructed $K_0 \supset \dots \supset K_p$ and $x_1, \dots, x_p$ such that $x_i \in K_{i-1}$ for all $i\in [1,p]$. If $K_p \neq \emptyset$, we let $x_{p+1}\in K_p$ be such that 
\be
u(x_{p+1})= \max_{K_p} u
\ee 
and we define $K_{p+1} \subset K_{p}$ by
\beq
\label{K p definition}
K_{p+1}=\left\{ x\in K_p \hbox{ s.t. }  \min_{i\in [1,p+1]} \vert x - x_i \vert u(x)^\frac{2}{N-2} \geq 1 \right\} \hskip.1cm.
\eeq
We claim that for any $x_1,...,x_p$ constructed in this way, we have
\beq
\label{iso01}
\vert x_i - x_j \vert u (x_i)^\frac{2}{N-2} \geq 1 \hbox{ for all } i\not=j \in [1,p]. 
\eeq  
 We prove (\ref{iso01}) by induction. For $p = 1$, there is nothing to prove. Suppose now that \eqref{iso01} is true for some $p\geq1$ and that $K_p \neq \emptyset$. Since $x_{p+1} \in K_p$, by definition of $K_p$, we have 
 \begin{equation}
 \label{induction1}
  \vert x_{p+1} - x_i \vert u(x_{p+1})^\frac{2}{N-2} \geq 1 \quad \text{ for all } i \in [1, p]. 
\end{equation} 
 Moreover, for any $i \in [1,p]$, we have $K_{i-1} \supset K_{p}$, and hence $u(x_i) \geq u(x_{p+1})$, since $x_i$ and $x_{p+1}$ are defined to be the maxima of $u$ over these sets. In particular, $u(i) \geq u(x_{p+1})$. Thus \eqref{induction1} implies 
\[
  \vert x_{p+1} - x_i \vert u(x_{i})^\frac{2}{N-2} \geq 1 \quad \text{ for all } i \in [1, p]. 
  \]
By the induction assumption, \eqref{iso01} is already true when both $i$ and $j$ are in $[1,p]$. Thus we have proved \eqref{iso01} for all $i \neq j \in [1, p+1]$.

Next, we observe that \eqref{iso01} implies the lower bound $|x_i - x_j| \geq \frac{1}{\|u\|_{L^\infty(\Omega)}} > 0$. Hence, the construction of the $x_p$ must stop after finitely many steps because  $\Omega$ is bounded.

Thus, there is $n\in \N^{*}$ such that $K_n=\emptyset$. Fix any $x \in K_u$. We claim that
\beq
\label{iso04}
\left( \min_{i\in [1,n]} \vert x_{i} -x \vert  \right) u^\frac{2}{N-2}(x) \leq 1. 
\eeq
Together with (\ref{iso01}), this will end the proof of  the lemma.  Since $K_n=\emptyset$, there exists $p \in [1, n]$ such that $x\in  K_{p-1}$ and $x\not\in K_{p}$. By the definition \eqref{K p definition} of the set $K_p$, we must have
\be
 \min_{i\in [1,p]} \vert x - x_i \vert u(x)^\frac{2}{N-2} <  1\hskip.1cm.
\ee
Since trivially $\min_{i \in [1,n]} |x - x_i| \leq \min_{i \in [1,p]} |x - x_i|$, inequality \eqref{iso04} follows. As already explained, this proves the lemma.
\end{proof}

\bibliographystyle{plain}
\bibliography{brezis-peletier}

\end{document}